\documentclass[11pt]{amsart}
\usepackage[top=1in, bottom=1in, left=1in, right=1in]{geometry}
\usepackage{amsfonts}
\usepackage{amsmath}
\usepackage{comment}
\usepackage{amssymb}
\DeclareUnicodeCharacter{0306}{}
\usepackage{bbm}
\usepackage{comment}
\usepackage{mathrsfs}
\usepackage[utf8]{inputenc}
\usepackage[T1]{fontenc}
\numberwithin{equation}{section}
\usepackage{times}
\usepackage[backend=biber, sorting=nyt, url=false,
maxnames = 100,
doi=false]{biblatex}
\usepackage{siunitx}
\usepackage[usenames,dvipsnames]{color}
\usepackage{comment}
\usepackage{xcolor}
\usepackage{mathtools}
\usepackage{bm}
\usepackage{esvect}
\usepackage{hyperref}
\hypersetup{
    colorlinks = true,
linkcolor={black},
urlcolor={blue},
citecolor={blue},    
urlcolor = {blue},
citebordercolor = {0.33 .58 0.33},
 linkbordercolor = {0.99 .28 0.23},
 breaklinks=true}
 
 \newcommand{\kommentar}[1]{}

\newtheorem{thm}{Theorem}[section]

\newtheorem{coro}[thm]{Corollary}
\newtheorem{lem}[thm]{Lemma}

\newcommand{\dd}{\;\mathrm{d}}

\theoremstyle{remark}
\newtheorem{rem}[thm]{Remark}

\usepackage{geometry}
\title[Non-vanishing and One Level Density for Dirichlet $L$-functions]{Non-vanishing and One Level Density for Dirichlet $L$-functions Along Short Averages}

\author{Debmalya Basak}
\address{
Debmalya Basak : Max Planck Institute for Mathematics,
Vivatsgasse 7, 53111 Bonn, Germany}
\email{basakd@mpim-bonn.mpg.de}

\begin{document}
\nocite{*}
\setcounter{tocdepth}{1}
\begin{abstract}
Assuming the Generalized Riemann Hypothesis, it is known that at least half of the central values $L(\frac{1}{2},\chi)$ are non-vanishing as $\chi$ ranges over primitive characters modulo $q$. Unconditionally, this is known on average over both $\chi$ modulo $q$ and $Q/2 \leq q \leq 2Q$. We prove that for any $\delta>0$, there exist $\eta_1,\eta_2>0$ depending on $\delta$ such that the non-vanishing proportion for $L(\frac{1}{2},\chi)$ as $\chi$ ranges modulo $q$ with $q$ varying in short intervals of size $Q^{1-\eta_1}$ around $Q$ and in arithmetic progressions with moduli up to $Q^{\eta_2}$  is larger than $\frac{1}{2}-\delta$. Furthermore, by studying the one-level density of low-lying zeros of $L(s, \chi)$, we show that under the Generalized Riemann Hypothesis, non-vanishing proportions exceeding $\frac{1}{2}$ can be obtained while still averaging over short ranges of $q$. 
\end{abstract}
\maketitle
\setcounter{tocdepth}{1}
\section{Introduction}
\subsection{Motivation} The study of non-vanishing of $L$-functions at the central point \(s = \frac{1}{2}\) is a fundamental problem in number theory with profound arithmetic implications. For instance, the Birch and Swinnerton-Dyer conjecture links the order of the central zero of an elliptic curve $L$-function with its arithmetic rank. Iwaniec and Sarnak \cite{IS2000} provided another crucial insight by showing that at least 50\% of \(L\)-functions in certain families of cusp forms do not vanish at the central point, and that any improvement on this proportion would rule out the existence of Landau--Siegel zeros. The mollifier method has been instrumental in achieving the majority of results concerning the non-vanishing of $L$-functions. For notable applications of the mollifier method, the reader is referred to the works of Bohr--Landau \cite{BL1914}, Bui \cite{B2012}, Conrey \cite{C1989}, Iwaniec--Sarnak \cite{IS2000}, Khan--Milićević--Ngo \cite{KMN2022}, Khan--Ngo \cite{KN2016}, Michel--VanderKam \cite{MV2000}, Selberg \cite{S1946} and Soundararajan \cite{KS1995, KS2000}. Additionally, for exploring connections between the non-vanishing of $L$-functions and Landau--Siegel zeros, the reader is referred to the works of Bui, Pratt, and Zaharescu \cite{BPZ2021, BPZ2024}, \v{C}ech--Matomäki \cite{CM2024}, Iwaniec--Sarnak \cite{IS2000} and the references there-in.
\par
In this paper, we investigate the classical problem of non-vanishing for primitive Dirichlet \(L\)-functions. It is widely conjectured that no such \(L\)-function \(L(s, \chi)\) vanishes at the central point. One may refer to Chowla \cite{C1965} for a variant of this conjecture pertaining to quadratic Dirichlet characters. Balasubramanian and Murty \cite{BM1992} showed unconditionally that a positive proportion of the family of $L$-functions associated to primitive characters modulo \(q\) do not vanish at the central point. Iwaniec and Sarnak \cite{IS1999} established that at least one third of the \(L\)-functions in this family do not vanish at the central point. Michel and VanderKam \cite{MV2000} achieved the same proportion of non-vanishing by introducing a new twisted mollifier. Bui \cite{B2012} improved this to 34.11\%, and this was subsequently raised to $\tfrac{7}{19}$ by Qin and Wu \cite{qin2025nonvanishingdirichletlfunctionscentral}, which is currently the best known result for non-prime moduli $q$. For prime moduli, the works of Khan and Ngo \cite{KN2016}, and Khan, Milićević and Ngo \cite{KMN2022} show that the non-vanishing proportion can be improved to $\frac{5}{13}$. Considering the family of quadratic Dirichlet characters,  Soundararajan \cite{KS2000} showed that at least $\frac{7}{8}$ of these \(L\)-functions do not vanish at \(s = \frac{1}{2}\). More recently, Baluyot and Pratt \cite{BP2022} considered the case of quadratic Dirichlet characters with prime moduli and established that more than nine percent of the central values of these $L$-functions are non-zero. 
\par
Assuming the Generalized Riemann Hypothesis (GRH), it is known (see Balasubramanian--Murty \cite{BM1992}, Sica \cite{S1998}) that at least half of the primitive characters \(\chi \bmod q\) satisfy \(L(\tfrac{1}{2}, \chi ) \neq 0\). Unconditionally, it is reasonable to anticipate a higher proportion of non-vanishing than the existing results by incorporating an additional averaging over the moduli \(q\). Indeed, Iwaniec and Sarnak \cite{IS1999} asserted that by averaging over the moduli, one can show unconditionally that at least half of the central values are non-zero. In this context, Pratt \cite{P2019} obtained a non-vanishing proportion exceeding $\tfrac{1}{2}$ among the family of primitive characters modulo $q$ with $Q/2 \leq q \leq 2Q$. It was clarified in private communications with Pratt, Matomäki, and \v{C}ech that the proof in \cite{P2019} only produces a non-vanishing proportion of $\tfrac{1}{2}$. See \cite[Abstract]{P2024} for reference.

The study of non-vanishing of $L$-functions at the central point is closely related to the one-level density of their low-lying zeros, that is, zeros near the central point. The Katz--Sarnak philosophy predicts that for any smooth test function $\phi$ and any natural family $\mathcal{F}$ of automorphic objects,
\begin{align}\label{Eq: OLD}
    \frac{1}{\#\mathcal{F}} 
    \sum_{\pi \in \mathcal{F}} 
    \sum_{\gamma_{\pi}} 
    \phi\!\left( 
        \frac{\log \mathfrak{c}_{\pi}}{2\pi} \, \gamma_{\pi}
    \right)
    \ \longrightarrow\ 
    \int_{\mathbb{R}} \phi(x)\, K_{\mathcal{F}}(x)\, \mathrm{d}x
    \qquad \textrm{as} \quad \#\mathcal{F} \to \infty,
\end{align}
where $\gamma_{\pi}$ runs over ordinates of the zeros of the $L$-function associated to $\pi$, $\mathfrak{c}_{\pi}$ denotes its analytic conductor, and $K_{\mathcal{F}}(x)$ is a function determined by the symmetry type of the family. For further background on this topic, we refer the reader to the works of Katz and Sarnak \cite{Katz--Sarnak, Katz--Sarnak2} and Iwaniec--Luo--Sarnak \cite{Iwaniec--Luo--Sarnak}.

There has been extensive work over the years providing evidence for \eqref{Eq: OLD}. In all such results, one needs to impose restrictions on the support of the Fourier transform of $\phi$, due to limitations arising from the relevant trace formulas. A major challenge is to enlarge this support beyond the “diagonal’’ range, that is, the range accessible through a direct application of the trace formula. Progress in this direction typically relies on deep hypotheses concerning the distribution of primes, often extending beyond GRH. For example, Iwaniec, Luo, and Sarnak \cite{Iwaniec--Luo--Sarnak} proved that for holomorphic cusp forms of even weight, one can unconditionally obtain one-level density results for $\operatorname{supp}\widehat{\phi} \subset (-1,1)$, which can be extended to $(-2,2)$ under GRH. They also showed that the support can be further enlarged to $\left(-\tfrac{22}{9},\,\tfrac{22}{9}\right)$ under an additional deep hypothesis concerning bounds for exponential sums over primes in arithmetic progressions.
\par
A notable breakthrough in this direction was recently obtained by Drappeau, Pratt, and Radziwi{\l}{\l} \cite{DPR2023}, who provided the first example of a family of $L$-functions for which one can unconditionally enlarge the support of the Fourier transform beyond the ``diagonal'' range. Considering the family of primitive Dirichlet characters $\chi \bmod q$, they proved that one may take
\[
\operatorname{supp}\widehat{\phi} \subset \left(-2-\frac{50}{1093},\, 2+\frac{50}{1093}\right)
\]
by incorporating an additional averaging over the moduli $q$ in the range $Q/2 \le q \le 3Q$. This extra averaging over $q$ is an essential ingredient in their method.

\par
In both the contexts of non-vanishing and one-level densities, the additional averaging over the moduli $q$ significantly enlarges the size of the family of $L$-functions under consideration. Rather than a family of roughly $q$ $L$-functions of conductor $q$, one obtains a family of size approximately $Q^2$, with conductors of size around $Q$. Motivated by Chowla's conjecture which addresses the non-vanishing problem for each individual \(q\), it is natural to try and obtain results on average over smaller sets of moduli $q$. Our first objective is to show that even when the averaging is restricted to a shorter range (by which we mean a power savings in the number of $q$'s used around $Q$), unconditionally one can still obtain non-vanishing proportions close to $\tfrac{1}{2}$. Secondly, we establish one-level density results for primitive Dirichlet $L$-functions with support extending beyond the ``so-called" diagonal range $(-2,2)$, while still averaging over short ranges of $q$. As an application, assuming GRH, we demonstrate that non-vanishing proportions exceeding $\tfrac{1}{2}$ can be achieved even with short-range averaging over the moduli. Finally, we also extend our results to families where the moduli $q$ are restricted to various arithmetic progressions.

\subsection{Main Results} To state our main results, we introduce some notations. Let $q$ be a positive integer. Let $\chi$ be a Dirichlet character modulo $q$ and let
\[L(s, \chi)=\sum_{n=1}^\infty \dfrac{\chi(n)}{n^s}, \quad \operatorname{Re}(s)>1, \]
be the associated Dirichlet $L$-function. Suppose $Q>0$ be sufficiently large, $\eta_1,\eta_2 \geq 0$ and let $T=Q^{\eta_1}$. We will work with the function $\mathcal{H}_T$ periodic mod 1 and which on $[-\frac{1}{2},\frac{1}{2}]$ is given by
\begin{align}\label{Defining h}
\mathcal{H}_T(t)= \begin{cases}T(1-T|t|), & |t| \leq \frac{1}{T}, \\ 0, & \frac{1}{T} \leq |t| \leq \frac{1}{2} .\end{cases}
\end{align} 
Later on, we will break $\mathcal{H}_T(t)$ as a sum of two parts having special properties as far as their sizes and Fourier coefficients are concerned.  Our first result is the following.
\begin{thm}\label{Main Theorem}
Let \(\eta_1, \eta_2 \geq 0\) be fixed such that 
\begin{align}\label{Eta 1-2 Constraint}
9\eta_1 + \eta_2 < \frac{1}{16}. 
\end{align}
Then there exists a constant \(c(\eta_1, \eta_2) > 0\) depending only on \(\eta_1\) and \(\eta_2\) with the following property. Let \(\Psi\) be a fixed nonnegative smooth function compactly supported in \([\frac{1}{2},\frac{3}{2}]\) with \(\Psi(1) > 0\). Consider \(a,D \textrm{ and } Q\in \mathbb{N}\) such that $1 \leq a \leq D \leq Q^{\eta_2}$ with \(\gcd(a, D) = 1\). Let $\varepsilon>0$. Then for \(Q\) sufficiently large in terms of $\eta_1, \eta_2$ and $\varepsilon$,
\begin{align}\label{Main Theorem Inequality}
\sum_{q \equiv a \bmod D} \Psi \bigg(\frac{q}{Q}\bigg) \mathcal{H}_{Q^{\eta_1}} &\bigg( \frac{q}{Q}\bigg) \frac{q}{\varphi(q)} \sum_{\substack{\chi\bmod q\\ \chi\:  \textup{primitive}\\ L(\frac{1}{2}, \chi)\neq 0}}1 \notag \\
&\geq (c(\eta_1, \eta_2)-\varepsilon) \sum_{q \equiv a \bmod D} \Psi \bigg( \frac{q}{Q}\bigg) \mathcal{H}_{Q^{\eta_1}} \bigg( \frac{q}{Q}\bigg) \frac{q}{\varphi(q)}\sum_{\substack{\chi\bmod q\\ \chi\: \textup{primitive}}}1,
\end{align}
where \(\mathcal{H}_{Q^{\eta_1}}(t)\) is defined by \eqref{Defining h}. Moreover,
\begin{align}\label{Limit Result}
\lim_{\substack{\eta_1 \to 0 \\ \eta_2 \to 0}} c(\eta_1, \eta_2) = \frac{1}{2}.
\end{align}
\end{thm}
In the context of low-lying zeros of primitive Dirichlet $L$-functions and their one-level density, we prove the following result. 
\begin{thm}\label{Main Theorem A}
Let \(\eta_1, \eta_2 \geq 0\) be fixed such that 
\begin{align}\label{Eta 1-2 Constraint A}
86\eta_1 + 6\eta_2 < \frac{50}{1093}. 
\end{align}
Then there exists a constant \(\widetilde{c}(\eta_1, \eta_2) > 0\) depending only on \(\eta_1\) and \(\eta_2\) with the following property. Let \(\Psi\) be a fixed nonnegative smooth function compactly supported in \([\frac{1}{2}, \frac{3}{2}]\) with \(\Psi(1) > 0\). Suppose $\phi$ be a nonnegative smooth function such that $\textrm{supp } \widehat{\phi} \subset (-2-\widetilde{c}(\eta_1,\eta_2),2+\widetilde{c}(\eta_1, \eta_2))$. Consider \(a,D, \textrm{ and } Q \in \mathbb{N}\) such that $1 \leq a \leq D \leq Q^{\eta_2}$ with \(\gcd(a, D) = 1\). Then as $Q \to \infty$,
\begin{align}\label{Main Theorem Inequality A}
\sum_{q \equiv a \bmod D} \Psi \bigg(\frac{q}{Q}\bigg) \mathcal{H}_{Q^{\eta_1}}\bigg(\frac{q}{Q}\bigg) &\sum_{\substack{\chi\bmod q\\ \chi\:  \textup{primitive}}} \sum_{\gamma_{\chi}} \phi\left ( \frac{\log Q}{2 \pi} \gamma_{\chi} \right) \notag \\
&= \widehat{\phi}(0) \sum_{q \equiv a \bmod D} \Psi \bigg( \frac{q}{Q}\bigg) \mathcal{H}_{Q^{\eta_1}} \bigg( \frac{q}{Q}\bigg) \sum_{\substack{\chi\bmod q\\ \chi\: \textup{primitive}}}1+o(Q^2D^{-1}),
\end{align}
where \(\mathcal{H}_{Q^{\eta_1}}(t)\) is defined by \eqref{Defining h} and $\frac{1}{2}+i\gamma_{\chi}$ denotes the non-trivial zeros of $L(s,\chi)$. Since we do not assume the Generalized Riemann Hypothesis, the $\gamma_{\chi}$'s are allowed to be complex. Moreover,
\begin{align}\label{Limit Result A}
\lim_{\substack{\eta_1 \to 0 \\ \eta_2 \to 0}} \widetilde{c}(\eta_1, \eta_2) = \frac{50}{1093}.
\end{align}
\end{thm}
\begin{rem}
In Theorem~\ref{Main Theorem A}, the function $\widehat{\phi}$ is compactly supported, which implies that $\phi$, originally defined on $\mathbb{R}$, admits an analytic continuation to the entire complex plane. Moreover, the conclusion of Theorem~\ref{Main Theorem A} remains valid for any smooth function $\phi$ satisfying $\operatorname{supp}\widehat{\phi} \subset (-2-\widetilde{c}(\eta_1,\eta_2),\,2+\widetilde{c}(\eta_1,\eta_2))$, provided there exists a nonnegative smooth function $\phi_1$ with $\operatorname{supp}\widehat{\phi}_1 \subset (-2-\widetilde{c}(\eta_1,\eta_2),\,2+\widetilde{c}(\eta_1,\eta_2))$ such that $\phi_1(x) \geq \phi(x)$ for all $x \in \mathbb{R}$.
\end{rem}
The one-level density estimate from Theorem~\ref{Main Theorem A} allows us to improve Theorem~\ref{Main Theorem} under GRH, and thereby surpass the $\tfrac{1}{2}$ threshold for non-vanishing while still considering short averages over $q$.

\begin{thm}\label{Main Theorem C}
Let \(\eta_1, \eta_2 \geq 0\) be fixed such that 
\begin{align}\label{Eta 1-2 Constraint C}
86\eta_1 + 6\eta_2 < \frac{50}{1093}. 
\end{align}
Then there exists a constant \(c^{*}(\eta_1, \eta_2) > 0\) depending only on \(\eta_1\) and \(\eta_2\) with the following property. Let \(\Psi\) be a fixed nonnegative smooth function compactly supported in \([\frac{1}{2}, \frac{3}{2}]\) with \(\Psi(1) > 0\). Consider \(a,D, \textrm{ and } Q \in \mathbb{N}\) such that $1 \leq a \leq D \leq Q^{\eta_2}$ with \(\gcd(a, D) = 1\).  Then for \(Q\) sufficiently large in terms of $\eta_1, \eta_2$ and $\varepsilon$,
\begin{align}\label{Main Theorem Inequality C}
\sum_{q \equiv a \bmod D} \Psi \bigg(\frac{q}{Q}\bigg) \mathcal{H}_{Q^{\eta_1}} &\bigg( \frac{q}{Q}\bigg) \sum_{\substack{\chi\bmod q\\ \chi\:  \textup{primitive}\\ L(\frac{1}{2}, \chi)\neq 0}}1 \notag \\
&\geq (\frac{1}{2}+c^{*}(\eta_1, \eta_2)-\varepsilon) \sum_{q \equiv a \bmod D} \Psi \bigg( \frac{q}{Q}\bigg) \mathcal{H}_{Q^{\eta_1}} \bigg( \frac{q}{Q}\bigg) \sum_{\substack{\chi\bmod q\\ \chi\: \textup{primitive}}}1,
\end{align}
where \(\mathcal{H}_{Q^{\eta_1}}(t)\) is defined by \eqref{Defining h}. Moreover,
\begin{align}\label{Limit Result C}
\lim_{\substack{\eta_1 \to 0 \\ \eta_2 \to 0}} c^{*}(\eta_1, \eta_2) = \frac{25}{2236}.
\end{align}
\end{thm}
\begin{rem}
The sum over $q$ in \eqref{Main Theorem Inequality}, \eqref{Main Theorem Inequality A} and \eqref{Main Theorem Inequality C} ranges in short intervals and also in arithmetic progressions. More precisely, it is supported in the range $\lvert q-Q \rvert \leq Q^{1-\eta_1}$ with $q \equiv a \bmod D$. Therefore, with the choice $D= \lfloor Q^{\eta_2} \rfloor$, the size of the family of primitive Dirichlet $L$ functions under consideration is reduced to around $Q^{2-\eta_1-\eta_2}$. 
\end{rem}
\begin{rem}
Our proofs provide explicit rational functions for $c(\eta_1,\eta_2)$, $\widetilde{c}(\eta_1,\eta_2)$ and $c^{*}(\eta_1,\eta_2)$ in terms of $\eta_1$ and $\eta_2$. See \eqref{C_eta definition}, \eqref{ceta definition A} and \eqref{c_3 definition} below.
\end{rem}

\subsection{Outline of the Proofs}
The central idea in the proofs of Theorem \ref{Main Theorem} and \ref{Main Theorem A} is to introduce an exponential oscillator along with the smooth function $\Psi$. The extra exponential oscillator helps us in detecting short averages of the moduli $q$. The proof of Theorem \ref{Main Theorem} makes use of the Iwaniec–Sarnak and Michel–VanderKam mollifiers, refined to incorporate the additional exponential oscillation. On the other hand, the proof of Theorem \ref{Main Theorem A} is based on Linnik's dispersion method, further adapted to handle the short interval averaging over the moduli. In both arguments, a variant of Drappeau’s result on Deshouillers–Iwaniec type bounds for Kloosterman sums is applied to control certain exponential sums where the moduli $q$ vary in short intervals and in arithmetic progressions.
\subsection{Structure of the Paper}
In Section \ref{sec: Initial Steps}, we briefly discuss the mollifier method and begin with some preliminaries towards the proof of Theorem \ref{Main Theorem}. Section \ref{sec: Drappeau Result} is devoted to bounds on sums of Kloosterman sums in residue classes. The error terms arising from the mollified second moment are addressed using different analytic techniques in Section \ref{sec: Lemma 2 Part II}. In Section \ref{Sec: Kloosterman Sum Applications}, we apply the bounds on Kloosterman sums obtained in Section \ref{sec: Drappeau Result} to treat certain exponential sums and prove Theorem \ref{Main Theorem}. In Section \ref{Theorem 2 Initial Setup}, we discuss the initial setup for the proof of Theorem \ref{Main Theorem A}. Section \ref{sec: Dispersion} is dedicated to the dispersion method, and finally in Section \ref{sec: Proof of Theorem A: Final Details}, we provide the proofs of Theorem \ref{Main Theorem A} and \ref{Main Theorem C}.
\subsection{Notations}
We employ some standard notation that will be used throughout the paper.
\begin{itemize} 
\item The expressions $f(X)=\mathcal{O}(g(X))$, $f(X) \ll g(X)$, and $g(X) \gg f(X)$ are equivalent to the statement that $|f(X)| \leq C|g(X)|$ for all sufficiently large $X$, where $C>0$ is an absolute constant. A subscript of the form $\ll_{\alpha}$ means the implied constant may depend on the parameter $\alpha$. Dependence on several parameters is indicated similarly, as in $\ll_{\alpha, \lambda}$.
\item The function $\varphi$ denotes the Euler totient function. The notation $\mathrm{e}(x)$ stands for $\exp(2\pi i x)$. The function $\tau$ denotes the divisor function. The function $\Lambda$ denotes the von Mangoldt function.
\item For a variable $a$ and a parameter $A$, the notation $a \sim A$ implies that $A \leq a \leq 2A$. The notation $a \asymp A$ implies that there exist absolute constants $c_1,c_2>0$ such that $c_1A \leq a \leq c_2A$.
\item We write $(m,n^{\infty})$ to denote the product $\prod_{p^{\nu} \mid \mid m, p \mid n} p^{\nu}.$ For $d \in \mathbb{N}$, we define $\operatorname{ker}(d) = \prod_{p \mid d} p$. We also write $d \mid n^{\infty}$ if $\operatorname{ker}(d) \mid n$. We use $(m,n)$ and $[m,n]$ to denote the gcd and lcm of $m$ and $n$ respectively.
\item The notation $\sum_{\chi \bmod q}^{+}$ indicates that the sum is restricted to even primitive characters $\chi \bmod q$.
\item We define the Fourier transform $f$ of a function $f$ as
$$
\widehat{f}(\xi)=\int_{\mathbb{R}} f(t) \mathrm{e}(-\xi t) \dd t.
$$
If $f$ is smooth and compactly supported, the above is well-defined and we have
$$
f(t)=\int_{\mathbb{R}} \widehat{f}(\xi) \mathrm{e}(\xi t) \dd \xi .
$$
\item We denote by $\varepsilon$ an arbitrarily small positive quantity that may vary from one line to the next, or even within the same line. Thus we may write $X^{2\varepsilon} \leq X^{\varepsilon}$ with no reservations.
\end{itemize}

\section{Proof of Theorem \ref{Main Theorem}: Initial Setup} \label{sec: Initial Steps}
\subsection{A Survey of the Mollifier Method} 
The proof of Theorem \ref{Main Theorem} begins with the well-known mollifier method. To keep our discussion self-contained, we briefly review the mollification technique. From the work of Iwaniec--Sarnak \cite{IS1999}, it follows that
\begin{align*}
\sum_{\substack{\chi\bmod q\\ \chi\: \textup{primitive}}} L(\tfrac{1}{2},\chi ) &\sim \sum_{\substack{\chi\bmod q\\ \chi\: \textup{primitive}}}1, \\
\textrm{and} \quad \sum_{\substack{\chi\bmod q\\ \chi\: \textup{primitive}}}  \lvert L(\tfrac{1}{2},\chi ) \rvert^2 &\sim \frac{\varphi(q)}{q} \log q\sum_{\substack{\chi\bmod q\\ \chi\: \textup{primitive}}}1.
\end{align*}
The additional logarithmic factor in the second moment indicates significant variation in the magnitude of $L(\frac{1}{2},\chi)$. Therefore for each character $\chi$, we associate a function $\mathcal{M}(\chi)$ of the form
$$
\mathcal{M}(\chi)=\sum_{m \leq y} \frac{x_m \chi(m)}{\sqrt{m}}
$$
where $X=\left(x_m\right)$ is a sequence of real numbers supported on $1 \leq m \leq y$ with $x_1=1$ and $x_m \ll 1$. The function $\mathcal{M}(\chi)$, known as a mollifier, is designed to smooth out or ``mollify'' the large values of $L\left(\frac{1}{2}, \chi\right)$ when averaging over $\chi$. Consider the sums
$$
\mathcal{S}_1=\sum_{\substack{\chi\bmod q\\ \chi\: \textup{primitive}}}L(\tfrac{1}{2}, \chi ) \mathcal{M}(\chi) \quad \text { and } \quad \mathcal{S}_2=\sum_{\substack{\chi\bmod q\\ \chi\: \textup{primitive}}}\left|L(\tfrac{1}{2}, \chi )\mathcal{M}(\chi)\right|^2.
$$
The Cauchy--Schwarz inequality yields
$$
\sum_{\substack{\chi\bmod q\\ \chi\: \textup{primitive}, \, L\left(\frac{1}{2}, \chi\right) \neq 0}}  1 \geq \frac{\left|\mathcal{S}_1\right|^2}{\mathcal{S}_2}.
$$
To optimize the ratio $\left|\mathcal{S}_1\right|^2/\mathcal{S}_2$ with respect to the sequence $X=\left(x_m\right)$, Iwaniec and Sarnak chose
$$
\mathcal{M}(\chi)=\sum_{m \leq y} \frac{\mu(m) \chi(m)}{\sqrt{m}}\left(1-\frac{\log m}{\log y}\right)
$$
where $y=q^{\vartheta}$ with $0<\vartheta<\frac{1}{2}$. With this choice, they established a non-vanishing proportion of $\tfrac{1}{3}$ among the family of primitive characters $\chi$ modulo $q$.  
\par
For $i \in \{ 1,2\}$, let $y_i = q^{\vartheta_i}$, and let $P_i$ be smooth polynomials satisfying $P_i(0)=0$ and $P_i(1)=1$. Building on the work of Iwaniec and Sarnak, and motivated by the approximate functional equation for $L(s,\chi)$, Michel and VanderKam \cite{MV2000} considered a two piece mollifier of the form:
\begin{align}\label{MV-Mollifier Definition}
\mathcal{M}(\chi)=\mathcal{M}_{\mathrm{IS}}(\chi)+\mathcal{M}_{\mathrm{MV}}(\chi),
\end{align}
where 
\begin{align}
\mathcal{M}_{\mathrm{IS}}(\chi) &= \sum_{\ell \leq y_1} \frac{\mu(\ell) \chi(\ell)}{\sqrt{\ell}}P_1\left(\frac{\log (y_1/\ell)}{\log y_1} \right ), \label{IS Mollifier} \\
\mathcal{M}_{\mathrm{MV}}(\chi) &= \varepsilon(\overline{\chi}) \sum_{\ell \leq y_2} \frac{\mu(\ell)\overline{\chi(\ell)}}{\sqrt{\ell}}P_2\left(\frac{\log (y_2/\ell)}{\log y_2} \right ) \label{MV Mollifier},\\
\noalign{\noindent\hspace*{0pt}\text{and}} \varepsilon(\chi) &= \frac{1}{\sqrt{q}} \sum_{h \bmod q} \chi(h) \mathrm{e}\left(\frac{h}{q}\right) \label{Epsilon_chi definition}.
\end{align}
It is worth noting that Soundararajan \cite{KS1995} previously utilized a mollifier of this form to study mean values of the Riemann zeta function. The detailed analysis of the off-diagonal terms arising from this two piece mollifier imposes the restriction $\vartheta_1+\vartheta_2 <\tfrac{1}{2}$. With this constraint, the non-vanishing proportion obtained is still $\tfrac{1}{3}$. Bui \cite{B2012} improved this proportion to 34.11\% by considering the two piece mollifier
\begin{align}\label{Bui Mollifier 1}
\mathcal{M}(\chi)=\mathcal{M}_{\mathrm{IS}}(\chi)+\mathcal{M}_{\mathrm{B}}(\chi),
\end{align}
where $\mathcal{M}_{\mathrm{IS}}(\chi)$ is defined by \eqref{IS Mollifier} and
\begin{align}\label{Bui Mollifier 2}
\mathcal{M}_{\mathrm{B}}(\chi)=\frac{1}{\log q} \sum_{m n \leq y_2} \frac{(\log * \mu)(m) \mu(n) \overline{\chi}(m) \chi(n)}{\sqrt{m n}} P_2\left(\frac{\log (y_2/mn)}{\log y_2} \right).
\end{align}
Recently, Qin and Wu \cite{qin2025nonvanishingdirichletlfunctionscentral} established an average estimate for a quartic sum of Kloosterman sums, thereby improving the proportion to $\frac{7}{19}$, which is currently the best known result for a non-prime moduli $q$. For $q$ prime, Khan, Milićević and Ngo \cite{KMN2022} showed that using an unbalanced mollifier of the form
\begin{align}
\mathcal{M}(\chi)=c_1\mathcal{M}_{\mathrm{IS}}(\chi)+c_2\mathcal{M}_{\mathrm{MV}}(\chi),
\end{align}
where $c_1,c_2 >0$, a non-vanishing proportion of $\tfrac{5}{13}$ can be achieved. This remains the best known proportion for prime moduli, although it is still far from the proportion $\tfrac{1}{2}$ known under GRH.
\par
Unconditionally, we may hope to establish higher proportions of non-vanishing than the existing results by averaging over the moduli $q$. In this context, in \cite{P2019}, Pratt considered the three piece mollifier
\begin{align}\label{Pratt Mollifier}
\mathcal{M}(\chi)=\mathcal{M}_{\mathrm{IS}}(\chi)+\mathcal{M}_{\mathrm{B}}(\chi)+\mathcal{M}_{\mathrm{MV}}(\chi),
\end{align}
where $\mathcal{M}_{\mathrm{IS}}(\chi), \mathcal{M}_{\mathrm{B}}$ and $\mathcal{M}_{\mathrm{MV}}(\chi)$ are defined by \eqref{IS Mollifier}, \eqref{Bui Mollifier 2} and \eqref{MV Mollifier} respectively. See also \cite{P2024} for an update to \cite{P2019}. We also highlight the recent work of \v{C}ech and Matomäki \cite{cech2025optimalitymollifiers}, which studies the optimality of mollifiers in a very broad setting. In particular, they show that the Michel–Vanderkam mollifier is optimal within a wide class of balanced two-piece mollifiers. In the present paper, we will work with a two piece mollifier as in \eqref{MV-Mollifier Definition}.
\subsection{The Mollified Moments}\label{subsection : Notations} In what follows, we lay some preliminary groundwork for the proof of Theorem \ref{Main Theorem}. Let $\Psi$ be a fixed nonnegative smooth function compactly supported in $[\frac{1}{2},\frac{3}{2}]$ which satisfies $\Psi(1)>0$. Let \(\eta_1, \eta_2 \geq 0\) be fixed satisfying 
\[
0 \leq 9 \eta_1 + \eta_2 < \frac{1}{16}.\]
Let $Q \in \mathbb{N}$ be sufficiently large. Define $T=Q^{\eta_1}$ and consider \(a,D \in \mathbb{N}\) such that $1 \leq a \leq D \leq Q^{\eta_2}$ with \((a, D) = 1\). Define the set \(\mathcal{Q}(a, D) = \{ q \in \mathbb{N} : q \equiv a \bmod D \}\). Let $\mathcal{H}_{T}(t)$ be given by \eqref{Defining h}. By Fourier expansion, we have
\begin{align}\label{Fourier Expansion}
\mathcal{H}_T(t)=\sum_{k \in \mathbb{Z}} b(k) \mathrm{e}(k t)
\end{align}
where $\mathrm{e}(x) = \exp(2\pi i x)$, $b(0)=1$ and for any $k \neq 0$,
\begin{align*}
b(k)=\frac{T^2}{\pi^2 k^2} \sin ^2\left(\frac{\pi k}{T}\right) .
\end{align*}
Note that for $|k| \leq T / 2$, we have $\left|b(k)\right| \gg 1$. Fix $\varepsilon>0$ arbitrarily small and set 
\begin{align}\label{Choice of K}
K=Q^{2\eta_1+\varepsilon}.
\end{align}
We write 
\begin{align}\label{Breaking H}
\Phi(t) = \Psi(t)\mathcal{H}_T(t) &= \sum_{\lvert k \rvert \leq K }b(k) \Psi(t)\mathrm{e}(k t)+\sum_{\lvert k \rvert > K }b(k) \Psi(t) \mathrm{e}(k t)=\Phi_1(t)+\Phi_2(t),
\end{align}
say. We will show later in Section \ref{Sec: Kloosterman Sum Applications} that the contribution from $\Phi_2(t)$ to both sides of \eqref{Main Theorem Inequality} is small. Consequently, the bulk of our analysis will focus on $\Phi_1(t)$.
\par
We say a Dirichlet character $\chi \bmod q$ is even if $\chi(-1)=1$. Let $\mathscr{C}_q$ denote the set of primitive characters $\chi \bmod q$ and let $\mathscr{C}_q^{+}$ be the subset of even characters in $\mathscr{C}_q$. We write $\varphi^{+}(q)=\frac{1}{2} \varphi^*(q)$ where
\begin{align}\label{Primitive Characters Count}
\varphi^*(q)=\sum_{k \mid q} \varphi(k) \mu \left(\frac{q}{k}\right)=\left|\mathscr{C}_q\right|,
\end{align}
(see \cite[Lemma 4.1]{BM2011} for a proof of \eqref{Primitive Characters Count}). It is not difficult to show that $\left|\mathscr{C}_q^{+}\right|=$ $\varphi^{+}(q)+\mathcal{O}(1)$. In addition, we use the notation $\sum_{\chi \bmod q}^{+}$ to indicate that the summation is restricted to $\chi \in \mathscr{C}_q^{+}$. Our analysis will focus on even primitive characters. The odd characters can be addressed using a parallel argument. Let $y_i = Q^{\vartheta_i}$ for $i \in \{ 1,2\}$. We choose our mollifier $\mathcal{M}(\chi)$ as
\begin{align}\label{Mollifier Definition}
\mathcal{M}(\chi)=\mathcal{M}_{\mathrm{IS}}(\chi)+\mathcal{M}_{\mathrm{MV}}(\chi),
\end{align}
where $\mathcal{M}_{\mathrm{IS}}(\chi)$ and $\mathcal{M}_{\mathrm{MV}}(\chi)$ are defined by \eqref{IS Mollifier} and \eqref{MV Mollifier} respectively. Consider the sums 
\begin{align}
\mathcal{S}_{1} &=  \sum_{q \in \mathcal{Q}(a,D)} \Phi \left ( \frac{q}{Q}\right) \dfrac{q}{\varphi(q)}\hspace{0.2cm} \sideset{}{^+}\sum_{\chi \bmod q}L(\tfrac{1}{2}, \chi ) \mathcal{M}(\chi), \label{1st Moment}\\
\textrm{and} \quad \mathcal{S}_{2} &= \sum_{q \in \mathcal{Q}(a,D)} \Phi\left ( \frac{q}{Q}\right) \dfrac{q}{\varphi(q)} \hspace{0.2cm} \sideset{}{^+}\sum_{\chi \bmod q}\left|L(\tfrac{1}{2}, \chi )\mathcal{M}(\chi)\right|^2 \label{2nd Moment}.
\end{align}
By the Cauchy--Schwarz inequality, we obtain
\begin{align}\label{CS Inequality}
\bigg \lvert \sum_{q \in \mathcal{Q}(a,D)} \Phi\left ( \frac{q}{Q}\right) \dfrac{q}{\varphi(q)}\hspace{0.2cm}\sideset{}{^+}\sum_{\substack{\chi \bmod q \\ L(\frac{1}{2}, \chi) \neq 0}}1 \bigg \rvert \geq \frac{\lvert \mathcal{S}_{1} \rvert^2}{\lvert \mathcal{S}_{2} \rvert }.
\end{align}
Hence the proof of Theorem \ref{Main Theorem} primarily relies on deriving asymptotic estimates for $\mathcal{S}_1$ and $\mathcal{S}_2$.
\begin{lem}\label{Lemma 1st Moment}
Let \(\eta_1, \eta_2 \geq 0\) be fixed, and let $\Phi(t)$ be given by \eqref{Breaking H}. Suppose $0<\vartheta_1,\vartheta_2 <\frac{1}{2}$. Then
\[
\mathcal{S}_{1} = \left( P_1(1)+P_2(1)+o(1) \right ) \sum_{q \in \mathcal{Q}(a,D)} \Phi \left ( \frac{q}{Q}\right)\dfrac{q}{\varphi(q)} \varphi^{+}(q),  
\]
where $\mathcal{S}_1$ is defined in \eqref{1st Moment} and we denote $P_i\left(\frac{\log (y_i/x)}{\log y_i} \right )$ as $ P_i[x]$.
\end{lem}
\begin{proof}
Here we won't require any extra averaging over $q$. Indeed by \cite[Eq. 5.4]{IS1999} and \cite[Section 3]{MV2000}, the desired conclusion follows immediately.
\end{proof}
Regarding $\mathcal{S}_2$, we obtain the following result.
\begin{lem}\label{Lemma 2nd Moment}
Let \(\eta_1, \eta_2 \geq 0\) be fixed such that 
\begin{align}\label{Eta1 and Eta2 Conditions}
9\eta_1 + \eta_2 < \frac{1}{16}. 
\end{align}
Let $\Phi(t)$ be given by \eqref{Breaking H}. Suppose 
\begin{align}\label{Mollifier Length}
0 < \vartheta_1, \vartheta_2  <\frac{1}{2}-80\eta_1-6\eta_2. 
\end{align}
Then
\[
\mathcal{S}_{2} = \left( \lambda_1+\lambda_2+2P_1(1)P_2(1)+o(1) \right )\sum_{q \in \mathcal{Q}(a,D)} \Phi \left ( \frac{q}{Q}\right)\dfrac{q}{\varphi(q)}\varphi^{+}(q),  
\]
where 
\begin{align}
\lambda_1 &= P_1(1)^2+\frac{1}{\vartheta_1}\int_{0}^1 P_1^\prime(x)^2 \dd x \quad
\textrm{and} \quad \lambda_2 = P_2(1)^2+\frac{1}{\vartheta_2}\int_{0}^1 P_2^\prime(x)^2 \dd x \label{Lambda Definition}.
\end{align}
Here $\mathcal{S}_2$ is defined in \eqref{2nd Moment} and we denote $P_i\left(\frac{\log (y_i/x)}{\log y_i} \right )$ as $ P_i[x]$.
\end{lem}
The proof of Lemma \ref{Lemma 2nd Moment} is more intricate and will be presented across the next few sections. In the remainder of this section, we deal with the main terms arising from $\mathcal{S}_2$. Using \eqref{Breaking H}, we write
\begin{align*}
    \mathcal{S}_2 = \mathcal{S}_{2,1}+\mathcal{S}_{2,2},
\end{align*}
where
\begin{align}
\mathcal{S}_{2,1} &=\sum_{q \in \mathcal{Q}(a,D)}\Phi_1 \left ( \frac{q}{Q}\right) \dfrac{q}{\varphi(q)} \hspace{0.2cm} \sideset{}{^+}\sum_{\chi \bmod q}\left|L(\tfrac{1}{2}, \chi )\mathcal{M}(\chi)\right|^2 \label{Tail Separation 1} \\
\textrm{and} \quad \mathcal{S}_{2,2} &= \sum_{q \in \mathcal{Q}(a,D)}  \Phi_2 \left ( \frac{q}{Q}\right) \dfrac{q}{\varphi(q)} \hspace{0.2cm} \sideset{}{^+}\sum_{\chi \bmod q}\left|L(\tfrac{1}{2}, \chi )\mathcal{M}(\chi)\right|^2. \label{Tail Separation 2}
\end{align}
As mentioned before, the contribution from \(\Phi_2(t)\), and consequently from \(\mathcal{S}_{2,2}\), will be small. Therefore, we will focus on \(\mathcal{S}_{2,1}\). To derive the main terms, we will not need to average over \(q\). We first write
\begin{align*}
\left|L(\tfrac{1}{2}, \chi )\mathcal{M}(\chi)\right|^2 =\left|L(\tfrac{1}{2}, \chi )\right|^2\left(\left|\mathcal{M}_{\mathrm{IS}}(\chi)\right|^2+ \left|\mathcal{M}_{\mathrm{MV}}(\chi)\right|^2+2\operatorname{Re}\mathcal{M}_{\mathrm{IS}}(\chi)\mathcal{M}_{\mathrm{MV}}(\overline{\chi})\right).
\end{align*}
Using results from \cite{IS1999} (also see \cite[Section 2.3]{B2012}, we have for $0<\vartheta_1 <\frac{1}{2}$,
\begin{align*}
\sum_{q \in \mathcal{Q}(a,D)}  \Phi_1 \left ( \frac{q}{Q}\right)\dfrac{q}{\varphi(q)}\hspace{0.2cm} \sideset{}{^+}\sum_{\chi \bmod q}\left|L(\tfrac{1}{2}, \chi )\right|^2 \left|\mathcal{M}_{\mathrm{IS}}(\chi)\right|^2 = (\lambda_1+o(1)) \sum_{q \in \mathcal{Q}(a,D)}  \Phi_1 \left ( \frac{q}{Q}\right) \dfrac{q}{\varphi(q)} \varphi^{+}(q),
\end{align*}
where $\lambda_1$ is given by $\eqref{Lambda Definition}$. The same analysis holds for $\left|\mathcal{M}_{\mathrm{MV}}(\chi)\right|^2$ with $\lambda_1$ replaced by $\lambda_2$. Therefore, we are left with to establish the following estimate:
\begin{align}\label{Reduction Step 1}
\sum_{q \in \mathcal{Q}(a,D)} \Phi_1\left ( \frac{q}{Q}\right)\dfrac{q}{\varphi(q)}&\hspace{0.2cm} \sideset{}{^+}\sum_{\chi \bmod q}\left|L(\tfrac{1}{2}, \chi )\right|^2\left(\mathcal{M}_{\mathrm{IS}}(\chi)\mathcal{M}_{\mathrm{MV}}(\overline{\chi})\right) \notag \\
&= (P_1(1)P_2(1)+o(1))\sum_{q \in \mathcal{Q}(a,D)} \Phi_1\left ( \frac{q}{Q}\right) \dfrac{q}{\varphi(q)}\varphi^{+}(q).
\end{align}
The majority of the work will be dedicated to proving \eqref{Reduction Step 1}. We collect two preliminary lemmas.
\begin{lem}\label{Approximate FE}
Let $\chi$ be a primitive even character modulo $q$. Let $G(s)$ be an even polynomial satisfying $G(0)=1$, and which vanishes to second order at $s = \frac{1}{2}$. Then
\[
\left|L(\tfrac{1}{2}, \chi )\right|^2 = 2 \underset{\substack{m,n}}{\sum \sum} \frac{\chi(m)\overline{\chi}(n)}{\sqrt{mn}} \mathcal{Z}\left(\frac{m n}{q}\right),
\]
where 
\begin{align}\label{FE Integral}
\mathcal{Z}(x) = \frac{1}{2\pi i}\int_{(1)} \frac{\Gamma^2\left(\frac{s}{2}+\frac{1}{4}\right)}{\Gamma^2\left(\frac{1}{4}\right)}\frac{G(s)}{s}\pi^{-s} x^{-s} \, \mathrm{d} s.
\end{align}
\end{lem}
\begin{proof}
See \cite{IS1999}. The proof follows along the lines of \cite[Theorem 5.3]{IK2004}
\end{proof}
\begin{lem}\label{Character Orthogonality}
Let $(mn,q)=1$. Then
\[
\sideset{}{^+}\sum_{\chi \bmod q}\chi(m)\overline{\chi}(n) = \frac{1}{2}\underset{\substack{vw=q \\ w \mid m \pm n}}{\sum \sum} \mu(v)\varphi(w).
\]
\end{lem}
\begin{proof} See \cite[Lemma 4.1]{BM2011}.
\end{proof}
Using the definition of our mollifier $\mathcal{M}(\chi)$ and Lemma \ref{Approximate FE}, we can write
\begin{align} \label{Functional Equation}
&\sideset{}{^+}\sum_{\chi \bmod q}\left|L(\tfrac{1}{2}, \chi )\right|^2 \mathcal{M}_{\mathrm{IS}}(\chi) \mathcal{M}_{\mathrm{MV}}(\overline{\chi}) \notag\\
& =2 \underset{\substack{\ell_1 \leq y_1 , \, \ell_2 \leq y_2 \\
\left(\ell_1 \ell_2,q\right)=1}}{\sum \sum} \frac{\mu(\ell_1) \mu(\ell_2) P_1[\ell_1] P_2[\ell_2]}{\sqrt{\ell_1 \ell_2}} \underset{\substack{(m n,q)=1}}{\sum \sum} \frac{1}{\sqrt{mn}} \mathcal{Z}\left(\frac{m n}{q}\right) \sideset{}{^+}\sum_{\chi \bmod q} \varepsilon(\chi)\chi\left(m \ell_1 \ell_2\right)\overline{\chi}(n).
\end{align}
The main term of \eqref{Reduction Step 1} arises when $m\ell_1 \ell_2=1$ in \eqref{Functional Equation}. The contribution from such terms is  
\[
2P_1(1)P_2(1) \sum_{q \in \mathcal{Q}(a,D)} \Phi_1\left ( \frac{q}{Q}\right) \dfrac{q}{\varphi(q)} \hspace{0.2cm} \sideset{}{^+}\sum_{\chi \bmod q} \varepsilon(\chi)\sum_{n} \frac{\overline{\chi}(n)}{\sqrt{n}} \cdot \mathcal{Z}\left(\frac{n}{q}\right).
\]
At this stage, following the analysis carried out by Pratt \cite[Section 4]{P2019}, we see that the main term is
\begin{align}\label{Reduction Step 2}
(P_1(1)P_2(1)+o(1)) \sum_{q \in \mathcal{Q}(a,D)} \Phi_1\left ( \frac{q}{Q}\right)\dfrac{q}{\varphi(q)}\varphi^{+}(q).
\end{align}
Therefore we are left with the terms in \eqref{Functional Equation} for which $m\ell_1 \ell_2 \neq 1$. These terms contribute to error terms. To treat these error terms, we will follow the framework of Pratt \cite{P2019}, supplemented with additional tools to handle the exponential oscillator that detects short intervals. The details will be carried out in Sections \ref{sec: Lemma 2 Part II} and \ref{Sec: Kloosterman Sum Applications}.
\section{Bounds on Sums of Kloosterman Sums in Residue Classes}\label{sec: Drappeau Result}
To bound the error terms from \eqref{Functional Equation} we need to establish cancellations in sums of Kloosterman sums in residue classes. In this regard, we first recall the celebrated work of Deshouillers and Iwaniec \cite{DI1982}.
\begin{lem}\label{Deshouillers-Iwaniec Result}
Let $\boldsymbol{C}, \boldsymbol{D}, \boldsymbol{N}, \boldsymbol{R}, \boldsymbol{S} \geq 1$ and $\left(b_{n, r, s}\right)$ be a sequence supported inside $(0, \boldsymbol{N}] \times(\boldsymbol{R}, 2 \boldsymbol{R}] \times(\boldsymbol{S}, 2 \boldsymbol{S}] \cap \mathbb{N}^3$. Let $g: \mathbb{R}_{+}^5 \rightarrow$ $\mathbb{C}$ be a smooth function compactly supported in $[\boldsymbol{C}, 2 \boldsymbol{C}] \times [\boldsymbol{D}, 2 \boldsymbol{D}] \times\left(\mathbb{R}_{+}\right)^3$, satisfying the bound
\begin{align}\label{Derivative Bounds}
\frac{\partial^{\nu_1+\nu_2+\nu_3+\nu_4+\nu_5}}{\partial c^{\nu_1} \partial d^{\nu_2} \partial n^{\nu_3} \partial r^{\nu_4} \partial s^{\nu_5}}g(c, d, n, r, s) \ll_{\nu_1, \nu_2, \nu_3, \nu_4, \nu_5} c^{-\nu_1} d^{-\nu_2} n^{-\nu_3} r^{-\nu_4} s^{-\nu_5}
\end{align}
for all fixed $\nu_j \geq 0$ and $1 \leq j \leq 5$. Then 
\begin{align}\label{DI Estimate}
\underset{\substack{(r d, s c)=1}}{\sum_c \sum_d \sum_n \sum_r \sum_s} &b_{n, r, s} g(c, d, n, r, s) \mathrm{e}\left(n \frac{\overline{r d}}{s c}\right) \notag \\
& \ll_{\varepsilon}(\boldsymbol{CDNRS})^{\varepsilon} K(\boldsymbol{C}, \boldsymbol{D}, \boldsymbol{N}, \boldsymbol{R}, \boldsymbol{S})\left\|b_{\boldsymbol{N}, \boldsymbol{R}, \boldsymbol{S}}\right\|_2,
\end{align}
where $\overline{r d}$ denotes the multiplicative inverse of $r d \hspace{0.05cm}\bmod \hspace{0.05cm} s c$, $\left\|b_{\boldsymbol{N}, \boldsymbol{R}, \boldsymbol{S}}\right\|_2=\left(\sum_{n, r, s}\left|b_{n, r, s}\right|^2\right)^{\frac{1}{2}}$ and
\begin{align}\label{K Definition}
K^2(\boldsymbol{C}, \boldsymbol{D}, \boldsymbol{N}, \boldsymbol{R}, \boldsymbol{S})=\boldsymbol{CS}(\boldsymbol{RS} +\boldsymbol{N})(\boldsymbol{C}+\boldsymbol{DR})+\boldsymbol{C}^2 \boldsymbol{DS} \sqrt{(\boldsymbol{RS}+\boldsymbol{N}) \boldsymbol{R}}+\boldsymbol{D}^2 \boldsymbol{NR}.
\end{align}
\end{lem}
\begin{proof}
See \cite[Theorem 12]{DI1982}.
\end{proof}
\par
In Lemma \ref{Deshouillers-Iwaniec Result}, it is vital that the variables \(c\) and \(d\) are linked to a smooth weight function \(g(c, d)\). For the variable \(d\), this allows us to transition to complete Kloosterman sums modulo \(sc\). For the variable \(c\), this is important because, in the context of modular forms, the relevant quantity of interest is the average of Kloosterman sums over moduli attached to a smooth weight. For more results related to these topics, the reader is referred to the works of Blomer--Mili\'cevi\'c \cite{BM2015a, BM2015b}, Bettin--Chandee \cite{Bettin-Chandee}, Drappeau--Pratt--Radziwi\l\l \hspace{0.025cm} 
 \cite{DPR2023} and Litchman \cite{L2023}.
\par
In our argument, the exponential phase in Lemma \ref{Deshouillers-Iwaniec Result} takes a form such that our smooth variable $c$ lies in arithmetic progressions. In this context, Drappeau \cite{D2017} extended the result of Deshouillers--Iwaniec by utilizing modular forms with multiplier systems defined by Dirichlet characters. This approach allows for a more general framework that incorporates congruence conditions on the smooth variables $c$ and $d$. We state below Drappeau's result.
\begin{lem}\label{Drappeau's Theorem}
Let $\boldsymbol{C}, \boldsymbol{D}, \boldsymbol{N}, \boldsymbol{R}, \boldsymbol{S} \geq 1$, and $q, c_0, d_0 \in \mathbb{N}$ be given with $\left(c_0 d_0, q\right)=1$. Let $\left(b_{n, r, s}\right)$ be a sequence supported inside $(0, \boldsymbol{N}] \times(\boldsymbol{R}, 2 \boldsymbol{R}] \times(\boldsymbol{S}, 2 \boldsymbol{S}] \cap \mathbb{N}^3$. Let $g: \mathbb{R}_{+}^5 \rightarrow$ $\mathbb{C}$ be a smooth function compactly supported in $[\boldsymbol{C}, 2 \boldsymbol{C}] \times [\boldsymbol{D}, 2 \boldsymbol{D}] \times\left(\mathbb{R}_{+}\right)^3$, satisfying the bound
\begin{align}\label{Derivative Bounds II}
\frac{\partial^{\nu_1+\nu_2+\nu_3+\nu_4+\nu_5}}{\partial c^{\nu_1} \partial d^{\nu_2} \partial n^{\nu_3} \partial r^{\nu_4} \partial s^{\nu_5}}g(c, d, n, r, s) \ll_{\nu_1, \nu_2, \nu_3, \nu_4, \nu_5}\left\{c^{-\nu_1} d^{-\nu_2} n^{-\nu_3} r^{-\nu_4} s^{-\nu_5}\right\}^{1-\varepsilon_1}
\end{align}
for some small $\varepsilon_1>0$ and all fixed $\nu_j \geq 0$. Then
\begin{align}\label{Drappeau's Result}
\underset{\substack{c \equiv c_0 \text { and } d \equiv d_0\bmod q \\ (qr d, s c)=1}}{\sum_c \sum_d \sum_n \sum_r \sum_s} &b_{n, r, s} g(c, d, n, r, s) \mathrm{e}\left(n \frac{\overline{r d}}{s c}\right) \notag \\
& \ll_{\varepsilon, \varepsilon_1} (q\boldsymbol{CDNRS})^{\varepsilon+\mathcal{O}\left(\varepsilon_1\right)} q^{\frac{3}{2}} K(\boldsymbol{C}, \boldsymbol{D}, \boldsymbol{N}, \boldsymbol{R}, \boldsymbol{S})\left\|b_{\boldsymbol{N}, \boldsymbol{R}, \boldsymbol{S}}\right\|_2,
\end{align}
where $\left\|b_{\boldsymbol{N}, \boldsymbol{R}, \boldsymbol{S}}\right\|_2=\left(\sum_{n, r, s}\left|b_{n, r, s}\right|^2\right)^{\frac{1}{2}}$ and $K^2(\boldsymbol{C}, \boldsymbol{D}, \boldsymbol{N}, \boldsymbol{R}, \boldsymbol{S})$ is given by \eqref{K Definition}.
\end{lem}
\begin{proof}
See \cite[Theorem 2.1]{D2017}.
\end{proof}
\begin{rem}
The term $\boldsymbol{D}^2 \boldsymbol{NR}$ in Lemma \ref{Deshouillers-Iwaniec Result} and \ref{Drappeau's Theorem} replaces the term $\boldsymbol{D}^2 \boldsymbol{NRS^{-1}}$ from \cite[Theorem 12]{DI1982} and \cite[Theorem 2.1]{D2017} following the corrections suggested by Bombieri--Friedlander--Iwaniec \cite{BFI2019}.
\end{rem}
Theorem \ref{Drappeau's Theorem}, as currently stated, requires further modifications for our purposes due to the following reasons. Firstly, the loss factor $(q\boldsymbol{CDNRS})^{\mathcal{O}(\varepsilon_1)}$ on the right-hand side of \eqref{Drappeau's Result} will influence the choice of the lengths of our mollifiers, which, in turn, determine the final proportion of non-vanishing. Therefore, it is necessary to explicitly compute the dependence of the term $(q\boldsymbol{CDNRS})^{\mathcal{O}(\varepsilon_1)}$ on $\varepsilon_1$. Secondly, the bounds for the derivatives with respect to different variables will differ in our case. While a uniform bound as in \eqref{Derivative Bounds II} can be applied, it is not optimal and affects the final proportion of non-vanishing. To address these issues, we propose the following variant of Lemma \ref{Drappeau's Theorem} to suit our needs.
\begin{lem}\label{New Drappeau's Theorem}
Let $\boldsymbol{C}, \boldsymbol{D}, \boldsymbol{N}, \boldsymbol{R}, \boldsymbol{S} \geq 1$, and $q, c_0, d_0 \in \mathbb{N}$ be given with $\left(c_0 d_0, q\right)=1$. Let $\left(b_{n, r, s}\right)$ be a sequence supported inside $(0, \boldsymbol{N}] \times(\boldsymbol{R}, 2 \boldsymbol{R}] \times(\boldsymbol{S}, 2 \boldsymbol{S}] \cap \mathbb{N}^3$. Let $g: \mathbb{R}_{+}^5 \rightarrow$ $\mathbb{C}$ be a smooth function compactly supported in $[\boldsymbol{C}, 2 \boldsymbol{C}] \times [\boldsymbol{D}, 2 \boldsymbol{D}] \times\left(\mathbb{R}_{+}\right)^3$, satisfying the bound
\begin{align*}
\frac{\partial^{\nu_1+\nu_2+\nu_3+\nu_4+\nu_5}}{\partial c^{\nu_1} \partial d^{\nu_2} \partial n^{\nu_3} \partial r^{\nu_4} \partial s^{\nu_5}}g(c, d, n, r, s) \ll_{\nu_1, \nu_2, \nu_3, \nu_4, \nu_5} c^{-\nu_1(1-\varepsilon_1)} d^{-\nu_2(1-\varepsilon_2)} n^{-\nu_3(1-\varepsilon_3)} r^{-\nu_4(1-\varepsilon_4)} s^{-\nu_5(1-\varepsilon_5)}
\end{align*}
for some fixed $0<\varepsilon_j<\frac{1}{2}$ and all fixed $\nu_j \geq 0$ where $1 \leq j \leq 5$. Then
\begin{align}\label{New Drappeau's Result}
\underset{\substack{c \equiv c_0 \text { and } d \equiv d_0\bmod q \\ (qr d, s c)=1}}{\sum_c \sum_d \sum_n \sum_r \sum_s} &b_{n, r, s} g(c, d, n, r, s) \mathrm{e}\left(n \frac{\overline{r d}}{s c}\right) \notag \\
& \ll_{\varepsilon}(\boldsymbol{CDNRS})^{\varepsilon+12(\sum_{j}\varepsilon_j)} q^{\frac{3}{2}+\varepsilon} K(\boldsymbol{C}, \boldsymbol{D}, \boldsymbol{N}, \boldsymbol{R}, \boldsymbol{S})\left\|b_{\boldsymbol{N}, \boldsymbol{R}, \boldsymbol{S}}\right\|_2,
\end{align}
where $\left\|b_{\boldsymbol{N}, \boldsymbol{R}, \boldsymbol{S}}\right\|_2=\left(\sum_{n, r, s}\left|b_{n, r, s}\right|^2\right)^{\frac{1}{2}}$ and $K^2(\boldsymbol{C}, \boldsymbol{D}, \boldsymbol{N}, \boldsymbol{R}, \boldsymbol{S})$ is given by \eqref{K Definition}.
Moreover, when $\boldsymbol{S}=1$, the factor $q^{\frac{3}{2}+\varepsilon}$ on the right-hand side of \eqref{New Drappeau's Result} can be improved to $q^{1+\varepsilon}.$
\end{lem}
\begin{proof}
We follow \cite[Section 4.3.3]{D2017} with necessary modifications. We assume that the sequence $b_{n, r, s}$ is supported on $\boldsymbol{N}<n \leq 2 \boldsymbol{N}$ losing a factor $(\log \boldsymbol{N})^{\frac{1}{2}}$ in the process. Also, let $s_0 \in (\mathbb{Z}/q\mathbb{Z})^{\times}$ be fixed and assume without loss of generality that
\[
b_{n, r, s}=0 \text { unless } s \equiv s_0\bmod q.
\]
Finally we sum over $s_0$ losing a factor of $q^{\frac{1}{2}}$ in the process. When $\boldsymbol{S}=1$, we do not lose this extra factor of $q^{\frac{1}{2}}$. Later, in Section \ref{Sec: Kloosterman Sum Applications}, we will actually apply Lemma \ref{New Drappeau's Theorem} with the choice $\boldsymbol{S}=1$. \par
By Poisson summation, the left-hand side of \eqref{New Drappeau's Result} is
\begin{align}\label{Poisson Summation}
\sum_{\substack{c, m, n, r, s \\
(q r, s c)=1, \, c \equiv c_0\bmod q}} \frac{b_{n, r, s}}{s c q} \ddot{g}(c, m / s c q, n, r, s) \mathrm{e}\left(\frac{-m d_0 \overline{s_0 c_0}}{q}\right) S(n \overline{r},-m \overline{q} ; s c)
\end{align}
where
\begin{align}
\ddot{g}(c, m, n, r, s)&:=\int_{-\infty}^{\infty} g(c, \xi, n, r, s) \mathrm{e}(\xi m) \, \mathrm{d} \xi \label{g dot definition}\\
\textnormal{and} \quad S(m, n, c)&=\sum_{(x, c)=1} \mathrm{e} \left(\frac{m x+n \overline{x}}{c}\right). \label{Kloosterman Sum Definition}
\end{align}
Let $M>0$ be a parameter to be chosen later. We break our sum over $m$ in \eqref{Poisson Summation} into $\mathcal{A}_0, \mathcal{A}_{\infty}$ and $\mathcal{B}$ where $\mathcal{A}_0$ is the contribution from $m=0$, $\mathcal{A}_{\infty}$ is the contribution from $\lvert m \rvert >M$ and $\mathcal{B}$ is the contribution from $0 < \lvert m \rvert \leq M$. Following the correction in \cite{BFI2019} (see also 
 \cite{DPR2023}), the bound for $\mathcal{A}_0$ is 
\[
\mathcal{A}_0 \ll \frac{1}{q} \sum_{\substack{c, n, r, s \\(q r, s c)=1, \,  c \equiv c_0\bmod q}} \frac{\left|b_{n, r, s}\right|}{s c}|\ddot{g}(c, 0, n, r, s)|(n, s c) \ll q^{-2}(\log \boldsymbol{S})^2 \boldsymbol{D}\{\boldsymbol{NR}\}^{\frac{1}{2}}\left\|b_{\boldsymbol{N}, \boldsymbol{R}, \boldsymbol{S}}\right\|_2.
\]
The treatment for $\mathcal{B}$ is precisely as in \cite{D2017}, except  that in order to satisfy the derivative conditions in \cite[Proposition 4.13]{D2017}, it suffices to enlarge our bound by a factor of at most $\mathcal{O}((\boldsymbol{CDNRS})^{12\sum_j \varepsilon_j})$. Finally for $\mathcal{A}_{\infty}$, by repeated integration by parts of \eqref{g dot definition}, for fixed $k \geq 1$ and $m \neq 0$ we have
\[
\ddot{g}(c, m /(s c q), n, r, s) \ll_k \boldsymbol{D}^{1-k\left(1-\varepsilon_2\right)}\left(\frac{s c q}{|m|}\right)^k.
\]
Choosing $k = \lceil \varepsilon_2^{-1} \rceil$ and $M =(q\boldsymbol{SCD} )^{\varepsilon+4\varepsilon_2} \boldsymbol{SC} q  \boldsymbol{D}^{-1}$, we have the bound
$$
\ddot{g}(c, m /(s c q), n, r, s) \ll_{\varepsilon} 1 / m^2 \quad(|m|>M).
$$
Bounding the Kloosterman sum in \eqref{Poisson Summation} trivially by $s c$, we obtain
$$
\mathcal{A}_{\infty} \ll_{\varepsilon}(q\boldsymbol{S C D})^{\varepsilon+4\varepsilon_2} q^{-2} \boldsymbol{D}\{\boldsymbol{NR/S}\}^{\frac{1}{2}}\left\|b_{\boldsymbol{N}, \boldsymbol{R}, \boldsymbol{S}}\right\|_2
$$
which is acceptable. This finishes the proof.
\end{proof}
Following the work of Drappeau--Pratt--Radziwi\l\l \hspace{0.025cm} \cite[Section 3]{DPR2023}, we can use current knowledge on the spectral gap of the Laplacian on congruence surfaces (see Kim--Sarnak \cite{KS2003}) to derive the following refinement of Lemma~\ref{New Drappeau's Theorem}. This will be crucial in extending the admissible ranges in Theorem~\ref{Main Theorem A}.
\par
Let $\theta \geq 0$ be a bound towards the Petersson--Ramanujan conjecture, in the sense of \cite[Eq. (4.6)]{D2017}. Selberg's $\frac{3}{16}$ theorem corresponds to $\theta \leq \frac{1}{4}$, and the Kim--Sarnak bound \cite{KS2003} asserts that $\theta \leq \frac{7}{64}$.
 \begin{lem}\label{New Drappeau's Theorem with Theta}
Let the notations and hypothesis be as in Lemma \ref{New Drappeau's Theorem}. Then
\begin{align}\label{New Drappeau's Result with Theta}
\underset{\substack{c \equiv c_0 \text { and } d \equiv d_0\bmod q \\ (qr d, s c)=1}}{\sum_c \sum_d \sum_n \sum_r \sum_s} &b_{n, r, s} g(c, d, n, r, s) \mathrm{e}\left(n \frac{\overline{r d}}{s c}\right) \notag \\
& \ll_{\varepsilon}(\boldsymbol{CDNRS})^{\varepsilon+12(\sum_{j}\varepsilon_j)} q^{\frac{3}{2}+\varepsilon} K(\boldsymbol{C}, \boldsymbol{D}, \boldsymbol{N}, \boldsymbol{R}, \boldsymbol{S})\left\|b_{\boldsymbol{N}, \boldsymbol{R}, \boldsymbol{S}}\right\|_2,
\end{align}
where $\left\|b_{\boldsymbol{N}, \boldsymbol{R}, \boldsymbol{S}}\right\|_2=\left(\sum_{n, r, s}\left|b_{n, r, s}\right|^2\right)^{\frac{1}{2}}$ and 
\begin{align}\label{K Definition with Theta}
K^2(\boldsymbol{C}, \boldsymbol{D}, \boldsymbol{N}, \boldsymbol{R}, \boldsymbol{S})&=\boldsymbol{CS}(\boldsymbol{RS} +\boldsymbol{N})(\boldsymbol{C}+\boldsymbol{DR}) \notag \\
&\quad+\boldsymbol{C}^{1+4\theta} \boldsymbol{DS} \left ( (\boldsymbol{RS}+\boldsymbol{N}) \boldsymbol{R}\right)^{1-2\theta} \left ( 1+\frac{q\boldsymbol{C}}{\boldsymbol{RD}} \right)^{1-4\theta}+\boldsymbol{D}^2 \boldsymbol{NR}.
\end{align}
\end{lem}
\begin{proof}
The proof follows by combining the arguments from  Drappeau--Pratt--Radziwi\l\l \hspace{0.025cm} \cite[Section 3]{DPR2023} along with the proof of Lemma \ref{New Drappeau's Theorem}.
\end{proof}
\section{Proof of Theorem \ref{Main Theorem}: Preliminary Reductions}\label{sec: Lemma 2 Part II} This section will be devoted to relevant simplifications towards estimating the error terms that arises from \eqref{Functional Equation} when $m\ell_1 \ell_2 \neq 1$. Here averaging over $q$ will be necessary. All implied constants are allowed to depend on $\eta_1, \eta_2$ and $\Psi$. It suffices to work with $D=\lfloor Q^{\eta_2}\rfloor$. We express our error term as 
\begin{align}
\mathcal{E}&=\sum_{q \in \mathcal{Q}(a,D)}  \Phi_1 \left ( \frac{q}{Q}\right) \dfrac{q}{\varphi(q)}\hspace{0.2cm} \underset{\substack{\ell_1 \leq y_1 , \, \ell_2 \leq y_2 \\
\left(\ell_1 \ell_2,q\right)=1}}{\sum \sum}  \frac{\mu(\ell_1) \mu(\ell_2) P_1[\ell_1] P_2[\ell_2]}{\sqrt{\ell_1 \ell_2}} \notag \\
&\quad \times \underset{\substack{(m n,q)=1}}{\sum \sum} \frac{1}{\sqrt{mn}} \mathcal{Z}\left(\frac{m n}{q}\right) \sideset{}{^+}\sum_{\chi \bmod q} \varepsilon(\chi)\chi\left(m \ell_1 \ell_2\right)\overline{\chi}(n),\label{Getting the Error Terms}
\end{align}
where we assume $m\ell_1 \ell_2 \neq 1$ without explicitly stating it. Comparing with the size of the expected main term in Lemma \ref{Lemma 2nd Moment} which is of order $Q^{2-\eta_2}$, it suffices to show that
\begin{align}
\mathcal{E} \ll_{\varepsilon}Q^{2-\eta_2-\varepsilon}\label{Error Term Goal A}.
\end{align} 
\subsection{Localizing Parameters} Let us write
\begin{align}\label{F Fourier Expansion}
\Phi_1\left(\frac{q}{Q}\right)=\sum_{\lvert k \rvert \leq K} b(k) \Psi\left(\frac{q}{Q}\right) \mathrm{e}\left(\frac{k q}{Q}\right)=\sum_{\lvert k \rvert \leq K}  b(k) \Theta_k\left(\frac{q}{Q}\right),
\end{align}
where we set
\begin{align}\label{Theta k Definition}
\Theta_k\left(\frac{q}{Q}\right) = \Psi\left(\frac{q}{Q}\right) \mathrm{e}\left(\frac{k q}{Q}\right).
\end{align}
We expand the $\varepsilon(\chi)$ factor in \eqref{Getting the Error Terms} using \eqref{Epsilon_chi definition} and apply Lemma \ref{Character Orthogonality} to obtain
\begin{align*}
\sideset{}{^+}\sum_{\chi \bmod q} \varepsilon(\chi)\chi\left(m \ell_1 \ell_2\right)\overline{\chi}(n) = \frac{1}{\sqrt{q}} \underset{\substack{vw=q \\ (v,w)=1}}{\sum \sum} \mu^2(v)\varphi(w)\cos\left( \frac{2\pi n \overline{m\ell_1 \ell_2 v}}{w}\right).
\end{align*}
Therefore we can rewrite \eqref{Getting the Error Terms} as
\begin{align*}
\mathcal{E}&= \sum_{\lvert k \rvert \leq K} b(k) \underset{\substack{v,w\\ vw \in \mathcal{Q}(a,D) \\ (v,w)=1}}{\sum \sum} \mu^2(v) \frac{v}{\varphi(v)}\frac{w^{\frac{1}{2}}}{v^{\frac{1}{2}}} \Theta_k\left(\frac{vw}{Q}\right)\underset{\substack{\ell_1 \leq y_1 , \, \ell_2 \leq y_2 \\
\left(\ell_1 \ell_2, vw\right)=1}}{\sum \sum}  \frac{\mu(\ell_1) \mu(\ell_2) P_1[\ell_1] P_2[\ell_2]}{\sqrt{\ell_1 \ell_2}} \notag \\
& \hspace{0.5cm} \times \underset{\substack{(m n, vw)=1}}{\sum \sum} \frac{1}{\sqrt{mn}} \mathcal{Z}\left(\frac{m n}{vw}\right) \cos \left(\frac{2 \pi n \overline{m \ell_1 \ell_2 v }}{w}\right) = \sum_{\lvert k \rvert \leq K} b(k) \mathcal{E}_{k},
\end{align*}
say. Now, fix $k$ such that $\lvert k \rvert \leq K$. Also, fix $v \equiv v_0 \bmod D$ for some $v_0 \in (\mathbb{Z}/D\mathbb{Z})^{\times}$ and let $w_0 \equiv a \overline{v_0} \bmod D$. In order to establish \eqref{Error Term Goal A}, it suffices to show that
\begin{align*}
\mathcal{E}_{k,v_0}&= \underset{\substack{(v,w)=1 \\ v \equiv v_0 \bmod D\\ w \equiv w_0 \bmod D}}{\sum \sum} \mu^2(v) \frac{v}{\varphi(v)}\frac{w^{\frac{1}{2}}}{v^{\frac{1}{2}}} \Theta_k\left(\frac{vw}{Q}\right)\underset{\substack{\ell_1 \leq y_1 , \, \ell_2 \leq y_2 \\
\left(\ell_1 \ell_2, vw\right)=1}}{\sum \sum} \frac{\mu(\ell_1) \mu(\ell_2) P_1[\ell_1] P_2[\ell_2]}{\sqrt{\ell_1 \ell_2}} \notag \\
& \hspace{0.5cm} \times \underset{\substack{(m n, vw)=1}}{\sum \sum} \frac{1}{\sqrt{mn}} \mathcal{Z}\left(\frac{m n}{vw}\right) \cos \left(\frac{2 \pi n \overline{m \ell_1 \ell_2 v }}{w}\right) \ll_{\varepsilon} Q^{2-2\eta_1-2\eta_2-\varepsilon},
\end{align*}
uniformly over $k$ and $v_{0}$. We apply partitions of unity to $\mathcal{E}_{k,v_0}$ and restrict the sums to dyadic intervals. We use $\mathop{\sum\nolimits^*}$ to indicate that the sum is over $v \equiv v_0 \bmod D$, and similarly for $w$. With this, we consider
\begin{align*}
\mathcal{E}_{k,v_0}(M,N,L_1,L_2,V,W)&=\underset{\substack{ (v,w)=1 }}{\mathop{\sum\nolimits^*}\mathop{\sum\nolimits^*}} \mu^2(v) \frac{v}{\varphi(v)} \frac{w^{\frac{1}{2}}}{v^{\frac{1}{2}}}G\left(\frac{v}{V} \right )G\left(\frac{w}{W} \right ) \Theta_k\left(\frac{v w}{Q}\right) \notag \\
& \hspace{0.6cm} \times \underset{\substack{\ell_1 \leq y_1 , \, \ell_2 \leq y_2 \\
\left(\ell_1 \ell_2,vw\right)=1}}{\sum \sum} \frac{\mu(\ell_1) \mu(\ell_2) P_1[\ell_1] P_2[\ell_2]}{\sqrt{\ell_1 \ell_2}} G\left(\frac{\ell_1}{L_1} \right )G\left(\frac{\ell_2}{L_2} \right )  \notag \\
& \hspace{0.6cm} \times \underset{\substack{(mn, vw)=1}}{\sum \sum} \frac{1}{\sqrt{mn}} \mathcal{Z}\left(\frac{m n}{v w}\right) G\left(\frac{m}{M} \right )G\left(\frac{n}{N} \right ) \cos \left(\frac{2 \pi n \overline{m \ell_1 \ell_2 v}}{w}\right).
\end{align*}
Here $G$ is a smooth nonnegative function supported in $[\frac{1}{2},2]$ satisfying $G^{(j)}(x) \ll_j Q^{j\varepsilon}$, and the numbers $M,N,L_1,L_2, V,W$ range over powers of $2$. We may assume 
\[M,N,L_1,L_2,V,W \gg 1, \quad VW \asymp Q \quad \textrm{and} \quad L_i \ll y_i.
\] 
For any arbitrary large $A>0$, $\mathcal{Z}$ satisfies $\mathcal{Z}(x) \ll_A (1+x)^{-A}$. Therefore if $MN >Q^{1+\varepsilon}$, then for any arbitrary large $B>0$, $\mathcal{Z}\left( \frac{mn}{vw}\right) \ll_A(1+Q^{\varepsilon})^{-A} \ll_B Q^{-B}.$ Hence we can assume $MN \leq Q^{1+\varepsilon}$. After applying partitions of unity, there are at most $Q^{o(1)}$ summands in each parameter. So overall we have $Q^{o(1)}$ dyadic sums. Therefore up to a change in the definition of $G$, we may rewrite $\mathcal{E}_{k,v_0}$ as
\begin{align*}
\mathcal{E}_{k,v_0} &=\frac{W^{\frac{1}{2}}}{(MNL_1L_2V)^{\frac{1}{2}}} \underset{\substack{ (v,w)=1 }}{\mathop{\sum\nolimits^*}\mathop{\sum\nolimits^*}} \alpha(v) G_1\left(\frac{v}{V} \right )G_2\left(\frac{w}{W} \right )\Theta_k\left(\frac{v w}{Q}\right) \notag \\
& \hspace{0.5cm} \times \underset{\substack{\ell_1 \leq y_1 , \, \ell_2 \leq y_2 \\
\left(\ell_1 \ell_2,vw\right)=1}}{\sum \sum} \beta(\ell_1) \gamma(\ell_2) G_3\left(\frac{\ell_1}{L_1} \right )G_4\left(\frac{\ell_2}{L_2} \right ) \notag \\
& \hspace{0.5cm} \times \underset{\substack{(mn, vw)=1}}{\sum \sum}   \mathcal{Z}\left(\frac{m n}{v w}\right) G_5\left(\frac{m}{M} \right )G_6\left(\frac{n}{N} \right )\cos \left(\frac{2 \pi n \overline{m \ell_1 \ell_2 v}}{w}\right),
\end{align*}
where 
\begin{align}\label{Parameter Bounds}
\alpha(v), \beta(\ell_1), \gamma(\ell_2) \ll Q^{o(1)} \quad \textnormal{and} \quad G_i^{(j)}(x) \ll_{\varepsilon} Q^{j\varepsilon}, \quad 1 \leq i \leq 6.
\end{align}
\subsection{Separation of Variables} To proceed, we need to separate the variables involved in \(\Theta_k\) and \(\mathcal{Z}\). We separate the variables in \(\mathcal{Z}\) by expressing \(\mathcal{Z}\) as an integral using \eqref{FE Integral}, and shifting the line of integration to \(\operatorname{Re}(s) = (\log Q)^{-1}\). Due to the rapid decay of the \(\Gamma\) function in vertical strips, we can restrict \(\operatorname{Im}(s)\) to \(\ll Q^{\varepsilon}\). By interchanging the sum and the integral, and suitably redefining the functions \(G_i\) and \(\alpha(v)\), it suffices to consider \(\mathcal{E}_{k,v_0}\) in the form
\begin{align}\label{Separating Z}
\mathcal{E}_{k,v_0} &=\frac{W^{\frac{1}{2}}}{(MNL_1L_2V)^{\frac{1}{2}}} \underset{\substack{ (v,w)=1 }}{\mathop{\sum\nolimits^*}\mathop{\sum\nolimits^*}} \alpha(v) G_1\left(\frac{v}{V} \right )G_2\left(\frac{w}{W} \right )\Theta_k\left(\frac{v w}{Q}\right) \notag \\
& \times \underset{\substack{\ell_1 \leq y_1 , \, \ell_2 \leq y_2 \\
\left(\ell_1 \ell_2,vw\right)=1}}{\sum \sum} \beta(\ell_1) \gamma(\ell_2) G_3\left(\frac{\ell_1}{L_1} \right )G_4\left(\frac{\ell_2}{L_2} \right )\underset{\substack{(mn, vw)=1}}{\sum \sum}  G_5\left(\frac{m}{M} \right )G_6\left(\frac{n}{N} \right )\cos \left(\frac{2 \pi n \overline{m \ell_1 \ell_2 v}}{w}\right),
\end{align}
where the conditions in \eqref{Parameter Bounds} still hold. The separation of variables for $\Theta_k$ is more subtle. Let $M_{\Theta_k}(s)$ be the Mellin transform of $\Theta_k$. We have the relations
\begin{align*}
M_{\Theta_k}(s) &= \int_{0}^{\infty} x^{s-1}\Theta_k(x) \, \mathrm{d} x \quad \textrm{and} \quad \Theta_k(x)= \frac{1}{2\pi i}\int_{c-i\infty}^{c+i\infty} x^{-s}M_{\Theta_k}(s) \, \mathrm{d} s,
\end{align*}
for any $c>0$. Since $\Theta_k$ is only supported on $[\frac{1}{2},\frac{3}{2}]$ and uniformly bounded, the contour of integration is within the range of convergence of $M_{\Theta_k}$. Using by-parts integration, we have
\begin{align*}
M_{\Theta_k}(c+it) &= \int_{0}^{\infty} x^{c+it-1}\Theta_k(x) \, \mathrm{d} x = \int_{\frac{1}{2}-\varepsilon}^{\frac{3}{2}+\varepsilon} x^{c+it-1}\Theta_k(x) \dd x = -\int_{\frac{1}{2}-\varepsilon}^{\frac{3}{2}+\varepsilon} \frac{x^{c+it}}{c+it}\Theta_k^\prime(x) \dd x \notag
\end{align*}
for any $\varepsilon>0$. Repeating this step multiple times, we obtain for any $j \in \mathbb{N}$
\[ \lvert M_{\Theta_k}(c+it) \rvert \ll_j \frac{1}{\lvert t \rvert^j} \sup\{ \Theta_k^{(i)}(x), 1 \leq i \leq j\} \ll_j \frac{\lvert k \rvert^j}{\lvert t \rvert^j}.
\]
On the other hand, note that
\begin{align}\label{Mellin I 2}
\Theta_k\left(\frac{vw}{Q}\right ) = \frac{1}{2\pi i} \left(\frac{vw}{Q}\right )^{-c}\int_{-\infty}^{\infty} \left(\frac{vw}{Q}\right )^{-it}M_{\Theta_k}(c+it) \dd t.
\end{align}
Let $c=(\log Q)^{-1}$. We break the integral in \eqref{Mellin I 2} into $\lvert  t \rvert \leq Q^{\varepsilon}\lvert k \rvert$ and $\lvert  t \rvert > Q^{\varepsilon}\lvert k \rvert$. When $\lvert  t \rvert > Q^{\varepsilon}\lvert k \rvert$,
\begin{align*}
\frac{1}{2\pi i} \left(\frac{vw}{Q}\right )^{-c}&\int_{\lvert  t \rvert > Q^{\varepsilon}k} \left(\frac{vw}{Q}\right )^{-it}M_{\Theta_k}(c+it) \dd t \notag \ll \left(\frac{vw}{Q}\right )^{-c}\int_{\lvert  t \rvert > Q^{\varepsilon}k} \,\frac{\lvert k\rvert^j}{\lvert t \rvert^j} \dd t \ll_\varepsilon Q^{-100},
\end{align*}
by choosing $j$ sufficiently large. So we concentrate on the part where $\lvert  t \rvert \leq Q^{\varepsilon}\lvert k \rvert$. Interchanging summations and the integral, up to some bounded absolute constant, we may rewrite \eqref{Separating Z} as
\begin{align*}
\mathcal{E}_{k,v_0} &=\frac{W^{\frac{1}{2}}}{ (MNL_1L_2V)^{\frac{1}{2}}} \int_{\lvert  t \rvert \leq Q^{\varepsilon}\lvert k \rvert} M_{\Theta_k}(c+it) \underset{\substack{ (v,w)=1 }}{\mathop{\sum\nolimits^*}\mathop{\sum\nolimits^*}} \alpha(v) \left(\frac{v}{V}\right )^{-c-it} G_1\left(\frac{v}{V} \right ) \left(\frac{w}{W}\right )^{-c-it}  G_2\left(\frac{w}{W} \right )  \notag \\
&\times \underset{\substack{\ell_1 \leq y_1 , \, \ell_2 \leq y_2 \\
\left(\ell_1 \ell_2,vw\right)=1}}{\sum \sum} \beta(\ell_1) \gamma(\ell_2) G_3\left(\frac{\ell_1}{L_1} \right )G_4\left(\frac{\ell_2}{L_2} \right ) \underset{\substack{(mn, vw)=1}}{\sum \sum}  G_5\left(\frac{m}{M} \right )G_6\left(\frac{n}{N} \right )\cos \left(\frac{2 \pi n \overline{m \ell_1 \ell_2 v}}{w}\right) \dd t. \notag
\end{align*}
We insert the factor $(v/V)^{-c-it}$ into the definition of $\alpha(v)$ and modify the definition of $G_2(x)$ by incorporating the factor $x^{-c-it}$. We will bound $M_{\Theta_k}(c+it)$ trivially by $\mathcal{O}(1)$ in our range of $t$.  Recalling that $\lvert k \rvert \leq K =Q^{2\eta_1+\varepsilon}$, in order to prove \eqref{Error Term Goal A}, it suffices to show that 
\begin{align}\label{Error Term Goal II}
\widetilde{\mathcal{E}}_{k,v_0} &=\frac{W^{\frac{1}{2}}}{(MNL_1L_2V)^{\frac{1}{2}}} \underset{\substack{ (v,w)=1 }}{\mathop{\sum\nolimits^*}\mathop{\sum\nolimits^*}} \alpha(v) G_1\left(\frac{v}{V} \right )G_2\left(\frac{w}{W} \right )  \underset{\substack{\ell_1 \leq y_1 , \, \ell_2 \leq y_2 \\
\left(\ell_1 \ell_2,vw\right)=1}}{\sum \sum}\beta(\ell_1) \gamma(\ell_2) G_3\left(\frac{\ell_1}{L_1} \right )G_4\left(\frac{\ell_2}{L_2} \right ) \notag \\
& \hspace{0.5cm} \times \underset{\substack{(mn, vw)=1}}{\sum \sum}  G_5\left(\frac{m}{M} \right )G_6\left(\frac{n}{N} \right ) \mathrm{e} \left(\frac{n \overline{m \ell_1 \ell_2 v}}{w}\right) \ll_{\varepsilon} Q^{2-4\eta_1-2\eta_2-\varepsilon}
\end{align}
\begin{align}\label{Parameter Bounds II}
\textrm{where} \quad \alpha(v), \beta(\ell_1), \gamma(\ell_2) \ll Q^{o(1)}, \quad  G_2^{(j)}(x) \ll_{\varepsilon} Q^{j(2\eta_1+\varepsilon)} \quad \textrm{and} \quad G_i^{(j)}(x) \ll_{\varepsilon} Q^{j\varepsilon}, i \neq 2.
\end{align}
\subsection{Character Orthogonality and Range Reductions}\label{sec: Lemma 2 Part III} We will utilize character orthogonality to simplify our problem to a form where Lemma \ref{New Drappeau's Theorem} is applicable. As we simplify \( \widetilde{\mathcal{E}}_{k,v_0} \) further by narrowing our problem to specific cases, we will continue to refer to the newly adjusted error term as \( \widetilde{\mathcal{E}}_{k,v_0} \), with a slight abuse of notation. Our ultimate goal is to prove \eqref{Error Term Goal II}.
\subsubsection{Reducing the range of $v$.} Let $y_1=y_2 = Q^{\frac{1}{2}-\xi}$, where
\begin{align}\label{xi conditions}
\xi = A_1\eta_1+A_2\eta_2+\varepsilon.
\end{align}
Here $A_1$ and $A_2$ are positive constants satisfying
\begin{align}\label{xi conditions II}
A_1 \geq 4, \quad \textrm{and} \quad A_2 \geq 2.
\end{align}
For the moment, these are rather loose restrictions on $A_1$ and $A_2$, but this is convenient for the following calculations. Later on, we shall enforce stronger conditions on $A_1$ and $A_2$. We have the trivial bound 
\[\widetilde{\mathcal{E}}_{k,v_0} \ll_{\varepsilon} Q^{\varepsilon} W^{\frac{3}{2}}(VMNL_1L_2)^{\frac{1}{2}} = \frac{Q^{2+\varepsilon}Q^{\frac{1}{2}-\xi}}{V}.\]
When $V \geq Q^{\frac{1}{2}+\varepsilon}$, the trivial bound is $\ll_{\varepsilon} Q^{2-\xi}$ which is acceptable. Therefore we may assume $V \leq Q^{\frac{1}{2}+\varepsilon}$ from now on. This would imply $W \gg Q^{\frac{1}{2}-\varepsilon}$.
\par
We reduce the range of $V$ even further. Using orthogonality of multiplicative characters, we write
$$
\mathrm{e}\left(\frac{n \overline{m \ell_1 \ell_2 v}}{w}\right)=\frac{1}{\varphi(w)} \sum_{\chi \bmod w} \tau(\overline{\chi}) \chi(n) \overline{\chi}\left(m \ell_1 \ell_2 v\right).
$$
We apply the Gauss sum bound $|\tau(\overline{\chi})| \ll w^{\frac{1}{2}}$ to obtain
$$
\widetilde{\mathcal{E}}_{k,v_0} \ll \frac{W}{\left(M N L_1 L_2 V\right)^{\frac{1}{2}}} \underset{\substack{ v \sim V, \, w \sim W }}{\mathop{\sum\nolimits^*}\mathop{\sum\nolimits^*}} \frac{\alpha(v)}{\varphi(w)} \sum_{\chi \bmod w} \bigg|\underset{\substack{(mn, vw)=1}}{\sum \sum} \chi(n) \overline{\chi}(m)\bigg|\bigg|\underset{\substack{\left(\ell_1 \ell_2, v\right)=1}}{\sum \sum} \overline{\chi}\left(\ell_1 \ell_2\right)\bigg|,
$$
where we have suppressed some notations for brevity. An application of the Cauchy--Schwarz inequality and character orthogonality shows that
$$
\sum_{\chi \bmod w}\bigg|\underset{\substack{m,n}}{\sum \sum}\bigg|\bigg|\underset{\substack{\ell_1,\ell_2}}{\sum \sum}\bigg| \ll Q^{o(1)}\left(M N L_1 L_2\right)^{\frac{1}{2}}(M N+W)^{\frac{1}{2}}\left(L_1 L_2+W\right)^{\frac{1}{2}}.
$$
In particular, we obtain the bound
\begin{align}\label{Orthogonality estimate}
\widetilde{\mathcal{E}}_{k,v_0} \ll_{\varepsilon} \frac{Q^{1-\eta_2+\varepsilon}(M N)^{\frac{1}{2}}\left(y_1 y_2\right)^{\frac{1}{2}}}{V^{\frac{1}{2}}}+\frac{Q^{\frac{3}{2}-\eta_2+\varepsilon}(M N)^{\frac{1}{2}}}{V}+\frac{Q^{\frac{3}{2}-\eta_2+\varepsilon}\left(y_1 y_2\right)^{\frac{1}{2}}}{V}+\frac{Q^{2-\eta_2+\varepsilon}}{V^{\frac{3}{2}}}.
\end{align}
Here we noted that at most one of $V$ or $W$ could be small, and thus we always save a factor of $D \asymp Q^{\eta_2}$. The above bound is acceptable when $V> Q^{4\eta_1+\eta_2+\varepsilon}$, the condition precisely arising from the second term in the right hand side of \eqref{Orthogonality estimate}. Therefore in what follows, we assume $V \leq Q^{4\eta_1+\eta_2+\varepsilon}$.
\subsubsection{Reducing the range of $n$} We now show that $\widetilde{\mathcal{E}}_{k,v_0}$ is small if $N$ is somewhat large. 
\begin{lem}
If $ \dfrac{M}{N} \leq Q^{2\xi-8\eta_1-4 \eta_2-\varepsilon}$ and $m \ell_1 \ell_2 \neq 1$, then $\widetilde{\mathcal{E}}_{k,v_0} \ll_{\varepsilon} Q^{2-4\eta_1-2\eta_2-\varepsilon}$. 
\end{lem} 
\begin{proof}
Consider the sum over $n$, say
$$
\Sigma_N=\sum_{(n, v w)=1} G_6\left(\frac{n}{N}\right) \mathrm{e}\left(\frac{n \overline{m \ell_1 \ell_2 v}}{w}\right).
$$
We break $n$ into residue classes mod $w$ and use Möbius inversion to replace $(n, v)=1$. We write
\begin{align*}
\Sigma_N &= \sum_{(b,w)=1} \sum_{\substack{n \equiv b \bmod w}} \bigg( \sum_{d \mid (n,v)} \mu(d)\bigg ) G_6\left( \frac{n}{N}\right)  \mathrm{e}\left(\frac{b \overline{m \ell_1 \ell_2 v}}{w}\right) \\
&=\sum_{d \mid v} \mu(d) \sum_{(b^\prime,w)=1}   \mathrm{e}\left(\frac{b^\prime d \overline{m \ell_1 \ell_2 v}}{w}\right) \sum_{n^\prime \equiv b^\prime \bmod w} G_6\left(\frac{d n^\prime}{N}\right) \\
&=\sum_{d \mid v} \mu(d) \sum_{(b,w)=1}   \mathrm{e}\left(\frac{b d \overline{m \ell_1 \ell_2 v}}{w}\right) \sum_{n \equiv b \bmod w} G_6\left(\frac{d n}{N}\right),
\end{align*}
where in the last line, we change notations for simplicity. By Poisson summation, for any fixed $A>0$,
\begin{align}\label{Fourier Transform}
\Sigma_N=\sum_{d \mid v} \mu(d) \sum_{(b, w)=1} \mathrm{e}\left(\frac{b d \overline{m \ell_1 \ell_2 v}}{w}\right) \frac{N}{d w} \sum_{|h| \leq W^{1+\varepsilon}d/N} \mathrm{e}\left(\frac{b h}{w}\right) \widehat{G}_6\left(\frac{h N}{d w}\right)+\mathcal{O}_{\varepsilon,A}\left(Q^{-A}\right).
\end{align}
The error term from \eqref{Fourier Transform} is negligible by allowing $A$ sufficiently large. The contribution of the zero frequency $h=0$ to $\Sigma_N$ is
$$
\widehat{G}_6(0) \frac{N}{w} \sum_{d \mid v} \frac{\mu(d)}{d} \sum_{(b, w)=1} \mathrm{e}\left(\frac{b d \overline{m \ell_1 \ell_2 v}}{w}\right)=\widehat{G}_6(0) \frac{N}{w} \frac{\varphi(v)}{v}  \mu(w).
$$
Summing this over all the other parameters in $\widetilde{\mathcal{E}}_{k,v_0}$, we see that the zero frequency term contributes
\begin{align}\label{Zero Frequency}
\frac{\left(WM L_1 L_2 \right)^{\frac{1}{2}}}{\left( N V\right)^{\frac{1}{2}}} &\underset{\substack{ v \sim V,\, w \sim W }}{\mathop{\sum\nolimits^*}\mathop{\sum\nolimits^*}} \widehat{G}_6(0) \frac{N}{w} \frac{\alpha(v) \varphi(v)}{v}  \mu(w)\ll_{\varepsilon} Q^{\varepsilon}(VWMN L_1 L_2)^{\frac{1}{2}} \ll_{\varepsilon} Q^{\frac{3}{2}-\xi+\varepsilon}.
\end{align}
This bound is acceptable by \eqref{xi conditions} and \eqref{xi conditions II}. For the non-zero frequencies, we write the sum as
$$
\sum_{d \mid v} \mu(d) \frac{N}{d w} \sum_{|h| \leq W^{1+\varepsilon} d / N} \widehat{G}_6\left(\frac{h N}{d w}\right) \sum_{(b, w)=1} \mathrm{e}\left(\frac{b d \overline{m \ell_1 \ell_2 v}}{w}+\frac{bh}{w}\right) .
$$
The inner sum is
\begin{align*}
\sum_{(b, w)=1} \mathrm{e}\left(\frac{b}{w}\left(d \overline{m \ell_1 \ell_2 v}+h\right)\right) &= \sum_{(b, w)=1} \mathrm{e}\left(\frac{b\overline{m \ell_1 \ell_2 v}}{w}\left(d +hm \ell_1 \ell_2 v\right)\right).
\end{align*}
Since $(m \ell_1 \ell_2 v,w)=1$, the inner sum is equal to the Ramanujan sum $c_w\left(h m \ell_1 \ell_2 v+d\right)$. Note that $h m \ell_1 \ell_2 v+d \neq 0$ because $m \ell_1 \ell_2 \neq 1$. Thus the nonzero frequencies contribute
\begin{align*}
&\ll_{\varepsilon} \frac{Q^{\varepsilon} W^{\frac{1}{2}} (ML_1L_2)}{(MNL_1L_2V)^{\frac{1}{2}}} \sup _{0<|k| \ll Q^{\mathcal{O}(1)}} \underset{\substack{ v \sim V \\ w \sim W }}{\mathop{\sum\nolimits^*}\mathop{\sum\nolimits^*}} \sum_{d \mid v} \frac{N}{dw} \sum_{|h| \leq W^{1+\varepsilon} d / N} \widehat{G}_6\left(\frac{h N}{d w}\right)   \left|c_w(k)\right|  \\
&\ll_{\varepsilon} \frac{Q^{\varepsilon} \left(V W L_1 L_2 M\right)^{\frac{1}{2}}}{N^{\frac{1}{2}}} \sup _{0<|k| \ll Q^{\mathcal{O}(1)}} \sum_{w \sim W}\left|c_w(k)\right|.
\end{align*}
Here we have used the bound $\widehat{G}_6 \ll 1$. Since $\left|c_w(k)\right| \leq (k, w)$, the sum over $w$ is 
\begin{align*}
\sideset{}{}\sum_{w \sim W}\left|c_w(k)\right| \ll \sideset{}{}\sum_{w \sim W}(k,w) \ll \sum_{\substack{d \mid k \\ d \ll W}} \hspace{0.1cm} \sideset{}{} \sum_{\substack{w \sim W \\ (k,w)=d}} d \ll k^{o(1)}W \ll W^{1+o(1)}.
\end{align*}
Combining with \eqref{Zero Frequency}, it follows that
$$
\widetilde{\mathcal{E}}_{k,v_0} \ll_{\varepsilon}  Q^{\frac{3}{2}-\xi+\varepsilon}+Q^{2-\xi+\varepsilon} \frac{M^{\frac{1}{2}}}{N^{\frac{1}{2}}}.
 $$
This bound for $\widetilde{\mathcal{E}}_{k,v_0}$ is acceptable by \eqref{xi conditions} and \eqref{xi conditions II} along with the condition
\begin{align}
\dfrac{M}{N} \leq Q^{2\xi-8\eta_1-4 \eta_2-\varepsilon}.
\end{align}
This completes the proof.
\end{proof}
Let us now assume that $M>NQ^{2\xi-8\eta_1-4 \eta_2-\varepsilon}$. Since $MN \ll Q^{1+\varepsilon}$, it follows that $N \leq Q^{\frac{1}{2}-\xi+4\eta_1+2\eta_2+\varepsilon} \leq Q^{\frac{1}{2}}$.
With $\widetilde{\mathcal{E}}_{k,v_0}$ in this form, we remove the condition $(n, w)=1$ by writing
$$
\mathbbm{1}_{(n, w)=1}=\sum_{\substack{f|(n,w)}} \mu(f) .
$$
We move the sum over $f$ outside. Note that $f \ll N$. We write $n = n_1 f, w = w_1 f$ and observe that
$$
\frac{n_1 f \overline{m \ell_1 \ell_2 v}}{ w_1 f} \equiv \frac{n_1 \overline{m \ell_1 \ell_2 v}}{w_1} \bmod 1.
$$
Therefore \eqref{Error Term Goal II} can be written as
\begin{align}\label{Error Term Goal III}
\widetilde{\mathcal{E}}_{k,v_0}&=\frac{W^{\frac{1}{2}}}{\left(M N L_1 L_2 V\right)^{\frac{1}{2}}} \sum_{\substack{f \ll N \\ (f,D)=1}} \mu(f) 
 \underset{\substack{ (v, w_1 f)=1 }}{\mathop{\sum\nolimits^*}\mathop{\sum\nolimits^*}}\alpha(v) G_1\left(\frac{v}{V}\right) G_2\left(\frac{w_1 f}{W}\right) \notag \\
& \times \underset{\substack{\ell_1 \leq y_1 , \, \ell_2 \leq y_2 \\
\left(\ell_1 \ell_2,vw_1f\right)=1}}{\sum \sum}  \beta(\ell_1) \gamma(\ell_2) G_3\left(\frac{\ell_1}{L_1} \right )G_4\left(\frac{\ell_2}{L_2} \right ) \underset{\substack{(m, vw_1f)=1 \\ (n_1,v)=1}}{\sum \sum} G_5\left(\frac{m}{M} \right )G_6\left(\frac{n_1f}{N} \right )\mathrm{e} \left(\frac{n_1 \overline{m \ell_1 \ell_2 v}}{w_1}\right).
\end{align}
\subsection{Reducing the range of $f$.}
We now reduce the size of $f$ as we did for $V$. By transitioning to multiplicative characters, the sum over $v, w, m, n, \ell_1, \ell_2$ is bounded by
$$
\ll_{\varepsilon} \frac{Q^{\varepsilon}W^{\frac{1}{2}} V^{\frac{1}{2}}}{f^{\frac{1}{2}}} \sideset{}{^*}\sum_{w_1 \sim W / f} \frac{w_1^{\frac{1}{2}}}{\varphi(w_1)}\bigg(\frac{(M N)^{\frac{1}{2}}}{f^{\frac{1}{2}}}+w_1^{\frac{1}{2}}\bigg)\left(\left(L_1 L_2\right)^{\frac{1}{2}}+w_1^{\frac{1}{2}}\right) \ll_{\varepsilon} \frac{Q^{2-\eta_2+\varepsilon}}{f^{\frac{3}{2}}}.
$$
For clarity, we mention here that under the conditions \eqref{Eta1 and Eta2 Conditions} and $V \leq Q^{4\eta_1+\eta_2+\varepsilon}$, we always save a factor of $D \asymp Q^{\eta_2}$ over the $w_1$ sum. If $f> Q^{8\eta_1+2\eta_2+\varepsilon}$, then 
\begin{align*}
\sum_{f > Q^{8\eta_1+2\eta_2+\varepsilon}} \frac{Q^{2-\eta_2+\varepsilon}}{f^{\frac{3}{2}}} \ll_{\varepsilon} Q^{2-\eta_2+\varepsilon} \int_{ Q^{8\eta_1+2\eta_2+\varepsilon}}^{\infty} \frac{1}{t^{\frac{3}{2}}} \dd t \ll_{\varepsilon} Q^{2-4\eta_1-2\eta_2-\varepsilon}.
\end{align*}
Therefore the contribution from $f>Q^{8\eta_1+2\eta_2+\varepsilon}$ is small. In conclusion, it suffices to show that
\begin{align}
\widetilde{\mathcal{E}}_{k,v_0}&=\frac{W^{\frac{1}{2}}}{\left(M N L_1 L_2 V\right)^{\frac{1}{2}}} \sum_{\substack{f \ll \min \left(N, Q^{8\eta_1+2\eta_2+\varepsilon}\right) \\ (f,D)=1}} \mu(f) 
 \underset{\substack{ (v, w_1 f)=1 }}{\mathop{\sum\nolimits^*}\mathop{\sum\nolimits^*}}\alpha(v) G_1\left(\frac{v}{V}\right) G_2\left(\frac{w_1 f}{W}\right) \notag \\
&\quad \times \underset{\substack{\ell_1 \leq y_1 , \, \ell_2 \leq y_2 \\
\left(\ell_1 \ell_2,vw_1f\right)=1}}{\sum \sum} \beta(\ell_1) \gamma(\ell_2) G_3\left(\frac{\ell_1}{L_1} \right )G_4\left(\frac{\ell_2}{L_2} \right ) \notag\\
&\quad \times \underset{\substack{(m, vw_1f)=1 \\ (n_1,v)=1}}{\sum \sum} G_5\left(\frac{m}{M} \right )G_6\left(\frac{n_1f}{N} \right ) \mathrm{e}\left(\frac{n_1 \overline{m \ell_1 \ell_2 v}}{w_1}\right) \ll_{\varepsilon} Q^{2-4\eta_1-2\eta_2-\varepsilon} \label{Making f smaller},
\end{align}
where along with the conditions \eqref{Parameter Bounds II}, we also have
\begin{align}\label{Additional Constraints}
V \leq Q^{4\eta_1+\eta_2+\varepsilon} \quad \textrm{and} \quad \dfrac{M}{N} > Q^{2\xi-8\eta_1-4 \eta_2-\varepsilon}.
\end{align}
\section{Proof of Theorem \ref{Main Theorem}: Application of Kloosterman Sum Bounds}\label{Sec: Kloosterman Sum Applications} 
To complete the proof of Lemma \ref{Lemma 2nd Moment}, and consequently Theorem \ref{Main Theorem}, we now apply Lemma \ref{New Drappeau's Theorem}.
\subsection{Application of Lemma \ref{New Drappeau's Theorem}}  Before proceeding, we need a little more groundwork. We remove the conditions $(m, f)=1$ and $(m, v)=1$ from the left side of \eqref{Making f smaller} using Möbius inversion. We have
\[
\mathbbm{1}_{(m, f)=1}=\sum_{\substack{h|(m,f)}} \mu(h) \quad \text{and}\quad 
\mathbbm{1}_{(m, v)=1}=\sum_{\substack{t|(m,v)}} \mu(t).\]
Noting that $(h,t)=1$, we write $m=m_1ht$ and $v=v_1t$. Therefore we may rewrite $\widetilde{\mathcal{E}}_{k,v_0}$ as
\begin{align*}
 &\sum_{\substack{f \ll \min \left(N, Q^{8\eta_1+2\eta_2+\varepsilon}\right) \\ (f,D)=1}} \mu(f)\sum_{h \mid f} \mu(h)\sum_{\substack{t \ll V \\ (t,D)=1}} \mu(t)\frac{W^{\frac{1}{2}}}{\left(M N L_1 L_2 V\right)^{\frac{1}{2}}} \underset{\substack{(v_1, w_1 f)=1 \\ (t,w_1f)=1}}{\mathop{\sum\nolimits^*}\mathop{\sum\nolimits^*}}\alpha(v_1t) G_1\left(\frac{v_1t}{V}\right) G_2\left(\frac{w_1 f}{W}\right) \\
&\times  \underset{\substack{\ell_1 \leq y_1 , \, \ell_2 \leq y_2 \\
\left(\ell_1 \ell_2,v_1w_1ft\right)=1}}{\sum \sum}  
\beta(\ell_1) \gamma(\ell_2) G_3\left(\frac{\ell_1}{L_1} \right )G_4\left(\frac{\ell_2}{L_2} \right ) \underset{\substack{(m_1, w_1)=1 \\ (ht,w_1)=1 \\(n_1,v_1t)=1}}{\sum \sum} G_5\left(\frac{m_1ht}{M} \right )G_6\left(\frac{n_1f}{N} \right )\mathrm{e} \left(\frac{n_1 \overline{m_1 ht^2\ell_1 \ell_2 v_1}}{w_1}\right).
\end{align*}
Since $t$ divides $v$ and $f$ divides $w$, we must have $(t,f)=1$. Hence the condition $(ht,w_1)=1$ implies that $(t,w_1f)=1$. So we can drop the condition $(t,w_1f)=1$ by adding $(t,f)=1$. Finally, we write $(m_1,D)=h_1$ and $m_1 = m_2h_1$ with $(m_2,D)=1$. Breaking $m_2$ into residue classes mod $ D$, we get
\begin{align}
\widetilde{\mathcal{E}}_{k,v_0}&=\frac{W^{\frac{1}{2}}}{\left(M N L_1 L_2 V\right)^{\frac{1}{2}}}\sum_{u \in (\mathbb{Z}/D\mathbb{Z})^{\times}}\sum_{h_1 \mid D}\sum_{\substack{f \ll \min \left(N, Q^{8\eta_1+2\eta_2+\varepsilon}\right) \\ (f,D)=1}} \mu(f)\sum_{h \mid f} \mu(h) \sum_{\substack{t \ll V \\ (t,Df)=1}} \mu(t) \notag \\
&\times \underset{\substack{(v_1, w_1 f)=1 \\ (hh_1t,w_1)=1}}{\mathop{\sum\nolimits^*}\mathop{\sum\nolimits^*}}\alpha(v_1t) G_1\left(\frac{v_1t}{V}\right) G_2\left(\frac{w_1f}{W}\right)  \underset{\substack{\ell_1 \leq y_1 , \, \ell_2 \leq y_2 \\
\left(\ell_1 \ell_2,v_1w_1ft\right)=1}}{\sum \sum} \beta(\ell_1) \gamma(\ell_2) G_3\left(\frac{\ell_1}{L_1} \right )G_4\left(\frac{\ell_2}{L_2} \right ) \notag \\
&\times \underset{\substack{(m_2, w_1)=1 \\ m_2 \equiv u \bmod D  \\(n_1,v_1t)=1}}{\sum \sum} G_5\left(\frac{m_2hh_1t}{M} \right )G_6\left(\frac{n_1f}{N} \right ) \mathrm{e} \left(\frac{n_1 \overline{m_2 hh_1t^2\ell_1 \ell_2 v_1}}{w_1}\right).\label{Error Term Goal V}
\end{align}
For fixed choices of $f, h, h_1$ and $t$, consider the following sequence: for $n_1,r \in \mathbb{N}$, we set
\[
b_{n_1,r}(f,h,h_1,t) = 0, \quad \textrm{if} \quad hh_1t^2 \textrm{ does not divide } r.
\]
Otherwise, we write $r = hh_1t^2k$ and set
\begin{align}\label{BNRS Definition}
b_{n_1, r} (f,h,h_1,t)= \mathbbm{1}_{(k,f)=1} G_6\left(\frac{n_1 f}{N}\right) \underset{\substack{\ell_1, \ell_2,   v_1 \\ \ell_1 \ell_2 v_1=k \\ \left(\ell_1 \ell_2, v_1t\right)=1 }}{\mathop{\sum\nolimits}\mathop{\sum\nolimits}\mathop{\sum\nolimits^*}}&\mathbbm{1}_{(n_1, v_1t)=1}  \beta\left(\ell_1\right) \gamma\left(\ell_2\right) \alpha(v_1t)  \notag \\
&\quad \times G_1\left(\frac{v_1 t}{V}\right) G_3\left(\frac{\ell_1}{L_1}\right) G_4\left(\frac{\ell_2}{L_2}\right).
\end{align}
It follows that if $b_{n_1, r} \neq 0$, then $n_1 \sim N / f$. Also $b_{n_1, r} \neq 0$ implies that $r \sim \boldsymbol{R} \asymp h h_1t L_1 L_2 V$ with $r \equiv 0 \bmod hh_1 t^2 $ and $r \equiv hh_1t \ell_1 \ell_2 v_0 \bmod D$. 
\par
We first observe that given $n_1,r$ such that $b_{n_1, r} \neq 0$, $b_{n_1, r} \ll_{\varepsilon} Q^{\varepsilon}$. To see this, note that all the summands are of size at most $Q^{o(1)}$ and if $b_{n_1, r} \neq 0$, then $k \sim \frac{L_1L_2V}{t} \ll Q^{\mathcal{O}(1)}$. Therefore the length of the triple sum is at most $Q^{o(1)}$. It follows that
$$
\left\|b_{N/f, \boldsymbol{R}}\right\|_2 =\bigg (\sum_{n_1 \sim N/f} \sideset{}{^*}\sum_{r \sim \boldsymbol{R}} \left|b_{n_1, r}\right|^2\bigg)^{\frac{1}{2}} \ll_{\varepsilon} Q^{-\frac{\eta_2}{2}+\varepsilon} \frac{\left(h_1 N L_1 L_2 V\right)^{\frac{1}{2}}}{(f t)^{\frac{1}{2}}},
$$
where $\sideset{}{^*}\sum$ denotes that $r$ runs over a fixed residue class modulo $ ht^2D$. Again, we mention that we save a factor of $D$ over the $r$ sum under \eqref{Eta1 and Eta2 Conditions}, along with the constraint \eqref{Final Conditions} which we impose later. We now rewrite \eqref{Error Term Goal V} as
\begin{align}
\widetilde{\mathcal{E}}_{k,v_0}&=\frac{W^{\frac{1}{2}}}{\left(M N L_1 L_2 V\right)^{\frac{1}{2}}}\sum_{u \in (\mathbb{Z}/D\mathbb{Z})^{\times}}\sum_{h_1 \mid D}\sum_{\substack{f \ll \min \left(N, Q^{8\eta_1+2\eta_2+\varepsilon}\right) \\ (f,D)=1}}\sum_{h \mid f}   \sum_{\substack{t \ll V \\ (t,Df)=1}} \mu(f)\mu(h) \mu(t) \notag \\
&\quad \times \underset{\substack{(Drm_2,w_1)=1}}{\mathop{\sum\nolimits^*}_{w_1} \mathop{\sum\nolimits^*}_{m_2} \sum_{n_1} \sum_{r}}  b_{n_1,r}(f,h,h_1,t) G_2\left(\frac{w_1f}{W}\right) G_5\left(\frac{m_2hh_1t}{M} \right )\mathrm{e}\left(n_1 \frac{\overline{m_2 r}}{w_1}\right). \label{Error Term Goal VI}
\end{align}
Here we have noted that the conditions $(v_1,f) = 1$, $(\ell_1\ell_2, v_1ft) = 1$, and $(n_1, v_1t) = 1$ in \eqref{Error Term Goal V} are accounted for by \eqref{BNRS Definition}. The remaining coprimality conditions are addressed by $(Drm_2, w_1) = 1$. For clarity, we mention that the notation $\sideset{}{^*}\sum$ in \eqref{Error Term Goal VI} indicates that $w_1$ (and similarly $m_2$) is restricted to a fixed residue class modulo $D$.
 Now, the sum over $w_1,m_2,n_1$ and $r$ is of the form to which Lemma \ref{New Drappeau's Theorem} can be applied with the following dictionary. We choose 
\begin{align*}
&c=w_1, \quad d=m_2, \quad n =n_1, \quad s=1,\\
&\boldsymbol{C} \sim W/f,\quad \boldsymbol{D} \sim M/(hh_1t), \quad \boldsymbol{N} \sim N/f, \quad \boldsymbol{R} \sim hh_1t L_1L_2V, \quad \boldsymbol{S}=1,\\
&g(c,d,n,r,s) := g(w_1,m_2) = G_2\left(\frac{w_1f}{W}\right)G_5\left(\frac{m_2hh_1t}{M}\right),\\
&\varepsilon_1 = 2\eta_1+\varepsilon, \quad \varepsilon_2 = \varepsilon, \quad \varepsilon_3,\varepsilon_4,\varepsilon_5 = 0.
\end{align*}
Since $\boldsymbol{S}=1$, the factor $q^{\frac{3}{2}+\varepsilon}$ in the bound stated in Lemma \ref{New Drappeau's Theorem} can be reduced to $q^{1+\varepsilon}$. We have
\begin{align*}
\widetilde{\mathcal{E}}_{k,v_0} & \ll_{\varepsilon} Q^{-\frac{\eta_2}{2}+\varepsilon} \sum_{u \in (\mathbb{Z}/D\mathbb{Z})^{\times}} \sum_{h_1 \mid D} \sum_{f \ll Q^{8\eta_1+2\eta_2+\varepsilon}} \sum_{h \mid f} \sum_{t \ll Q^{4\eta_1+\eta_2+\varepsilon}}  \frac{W^{\frac{1}{2}}}{\left(M N L_1 L_2 V\right)^{\frac{1}{2}}} \frac{\left(h_1 N L_1 L_2 V\right)^{\frac{1}{2}}}{(f t)^{\frac{1}{2}}} \\
& \times \left(\frac{WMNL_1L_2V}{f^2} \right)^{24\eta_1}D^{1+\varepsilon}\left\{\frac{D^{\frac{1}{2}}W^{\frac{1}{2}}}{f^{\frac{1}{2}}}\bigg(\left(hh_1 t L_1 L_2 V\right)^{\frac{1}{2}}+\frac{N^{\frac{1}{2}}}{f^{\frac{1}{2}}}\bigg)\bigg(\frac{W^{\frac{1}{2}}}{f^{\frac{1}{2}}}+\left(M L_1 L_2 V\right)^{\frac{1}{2}}\bigg)\right. \\
& \left.+\frac{W}{f} \frac{M^{\frac{1}{2}}}{(hh_1 t)^{\frac{1}{2}}}\bigg(\left(hh_1 t L_1 L_2 V\right)^{\frac{1}{2}}+\left(\frac{h h_1t L_1 L_2 N V}{f}\right)^{\frac{1}{4}}\bigg)+\frac{M}{hh_1 t} \left(\frac{hh_1 t L_1 L_2 N V}{f}\right)^{\frac{1}{2}}\right\} \\
& \ll_{\varepsilon} Q^{72\eta_1+2\eta_2+\varepsilon} \sum_{h_1 \mid D} \sum_{f \ll Q^{8\eta_1+2\eta_2+\varepsilon}} \frac{1}{f^{\frac{1}{2}}} \sum_{h \mid f} \sum_{t \ll Q^{4\eta_1+\eta_2+\varepsilon}} \frac{1}{t^{\frac{1}{2}}}\bigg(\sum_{i=1}^7 \mathcal{T}_i \bigg),
\end{align*}
where
\begin{align*}
    \mathcal{T}_1 &= \frac{W^{\frac{3}{2}}(hh_1t L_1L_2 VD)^{\frac{1}{2}}}{fM^{\frac{1}{2}}}, \,\mathcal{T}_2 = \frac{WL_1L_2V(hh_1tD)^{\frac{1}{2}} }{f^{\frac{1}{2}}}, \, \mathcal{T}_3 = \frac{W^{\frac{3}{2}}N^{\frac{1}{2}} D^{\frac{1}{2}}}{f^{\frac{3}{2}}M^{\frac{1}{2}}}, \, \mathcal{T}_4= \frac{W(NL_1L_2VD)^{\frac{1}{2}}}{f},\\
    \mathcal{T}_5 &= \frac{W^{\frac{3}{2}}(L_1L_2V)^{\frac{1}{2}}}{f}, \, \mathcal{T}_6 = \frac{W^{\frac{3}{2}}(NL_1L_2V)^{\frac{1}{4}}}{f^{\frac{5}{4}}(hh_1t)^{\frac{1}{4}}} \quad \textrm{and} \quad \mathcal{T}_7 = \frac{(VWMNL_1L_2)^{\frac{1}{2}}}{(hh_1tf)^{\frac{1}{2}}}.
\end{align*}
By carefully estimating each of the above terms and imposing the conditions  
\begin{align}\label{Final Conditions}
A_1 \geq 80, \quad A_2 \geq 6, \quad \text{and} \quad 9\eta_1+\eta_2 < \tfrac{1}{16}-\varepsilon,
\end{align}
we conclude that the total contribution of all terms is bounded by $\ll_{\varepsilon} Q^{2-4\eta_1-2\eta_2-\varepsilon}.$ For instance, the contribution from $\mathcal{T}_1$ to $\widetilde{\mathcal{E}}_{k,v_0}$ is
\[
\ll_{\varepsilon} Q^{72\eta_1+3\eta_2+\varepsilon} W^{\tfrac{3}{2}}(L_1L_2V)^{\tfrac{1}{2}} 
   \sum_{f} f^{-\tfrac{3}{2}} \sum_{t} 1 
   \ll_{\varepsilon} Q^{2+76\eta_1+4\eta_2-\xi+\varepsilon},
\]  
which is admissible under \eqref{Final Conditions}. Likewise, using $N \leq Q^{\frac{1}{2}-\xi+4\eta_1+2\eta_2+\varepsilon}$, the contribution from $\mathcal{T}_6$ is  
\[
\ll_{\varepsilon} Q^{72\eta_1+2\eta_2+\varepsilon} W^{\tfrac{3}{2}} (NL_1L_2V)^{\tfrac{1}{4}} 
   \sum_{f} f^{-\tfrac{7}{4}} \sum_{t} t^{-\tfrac{3}{4}} 
   \ll_{\varepsilon} Q^{\tfrac{15}{8}+74\eta_1+3\eta_2-\tfrac{3}{4}\xi+\varepsilon}.
\]  
This bound is certainly $\ll_{\varepsilon} Q^{2-4\eta_1-2\eta_2-\varepsilon}$ under \eqref{Final Conditions}. The remaining terms can be treated in a similar fashion. We note that there is further scope for slightly improving the permissible ranges of $\eta_1$ and $\eta_2$ in \eqref{Final Conditions}. For simplicity, we have opted to present the argument in this form. Combining all the previous steps, \eqref{Error Term Goal A} follows. This establishes \eqref{Reduction Step 1}.
\par
To finish the proof of Lemma \ref{Lemma 2nd Moment}, it only remains to study $\mathcal{S}_{2,2}$ given by \eqref{Tail Separation 2}. We have
\begin{align}\label{S-22 First Step}
\mathcal{S}_{2,2} &\ll \sum_{q \in \mathcal{Q}(a,D)} \bigg \lvert \sum_{\lvert k \rvert >K} b(k) \mathrm{e}\left ( \frac{kq}{Q}\right)\bigg \rvert  \Psi \left ( \frac{q}{Q}\right) \dfrac{q}{\varphi(q)} \hspace{0.2cm} \sideset{}{^+}\sum_{\chi \bmod q}\left|L(\tfrac{1}{2}, \chi )\mathcal{M}(\chi)\right|^2 \notag \\
&\ll \frac{Q^{2\eta_1}}{K} \sum_{q \in \mathcal{Q}(a,D)} \Psi \left ( \frac{q}{Q}\right) \dfrac{q}{\varphi(q)} \hspace{0.2cm} \sideset{}{^+}\sum_{\chi \bmod q}\left|L(\tfrac{1}{2}, \chi )\mathcal{M}(\chi)\right|^2.
\end{align}
The sum over $q$ in \eqref{S-22 First Step} can be treated exactly how we established \eqref{Reduction Step 1}. Indeed, the main term is   
\[
\asymp \sum_{q \in \mathcal{Q}(a,D)} \Psi \left ( \frac{q}{Q}\right) \dfrac{q}{\varphi(q)} \varphi^{+}(q) \asymp Q^{2-\eta_2}.
\]
The error terms are $o(Q^{2-\eta_2})$, following the arguments used in the proof of $\eqref{Error Term Goal A}$. Furthermore, in this case since we don't have the extra oscillation factor, we can in fact choose $\varepsilon_1=\varepsilon$ while applying Lemma \ref{New Drappeau's Theorem}. Finally, recall that $K = Q^{2\eta_1+\varepsilon}$ by \eqref{Choice of K}. Therefore, we see that
\begin{align}\label{S-22 Final Step}
\mathcal{S}_{2,2} \ll_{\varepsilon} Q^{2-\eta_2-\varepsilon}.
\end{align}
Putting together \eqref{Reduction Step 1} and \eqref{S-22 Final Step}, we arrive at
\begin{align}\label{Final Equality}
\mathcal{S}_{2} = \left( \lambda_1+\lambda_2+2P_1(1)P_2(1)+o(1) \right )\sum_{q \in \mathcal{Q}(a,D)} \Phi_1 \left ( \frac{q}{Q}\right)\dfrac{q}{\varphi(q)}\varphi^{+}(q)+\mathcal{O}_{\varepsilon}(Q^{2-\eta_2-\varepsilon}).
\end{align}
Noting that 
\[
\sum_{q \in \mathcal{Q}(a,D)} \Phi_2 \left ( \frac{q}{Q}\right)\dfrac{q}{\varphi(q)}\varphi^{+}(q) \ll_{\varepsilon} Q^{2-\eta_2-\varepsilon},
\]
it follows that
\begin{align*}
\mathcal{S}_{2} = \left( \lambda_1+\lambda_2+2P_1(1)P_2(1)+o(1) \right )\sum_{q \in \mathcal{Q}(a,D)} \Phi \left ( \frac{q}{Q}\right)\dfrac{q}{\varphi(q)}\varphi^{+}(q),
\end{align*}
which completes the proof of Lemma \ref{Lemma 2nd Moment}.
\subsection{Proof of Theorem \ref{Main Theorem}: Final Details}\label{sec: Proof of Main Theorem}
We now present the proof of Theorem \ref{Main Theorem}.
\begin{proof}[Proof of Theorem \ref{Main Theorem}] Fix $\varepsilon>0$ arbitrarily small. Given $\eta_1,\eta_2 \geq 0$, define $\xi=80\eta_1+6\eta_2+\varepsilon$ and let $\vartheta = \vartheta_1 = \vartheta_2 =\frac{1}{2}-\xi$. We apply Lemma \ref{Lemma 1st Moment} and \ref{Lemma 2nd Moment} in \eqref{CS Inequality} to see that for $Q$ sufficiently large depending at most on $\eta_1,\eta_2$ and $\varepsilon$, we have
\begin{align*}
\bigg \lvert \sum_{q \in \mathcal{Q}(a,D)} \Phi\left ( \frac{q}{Q}\right) \dfrac{q}{\varphi(q)}\hspace{0.2cm}\sideset{}{^+}\sum_{\substack{\chi \bmod q \\ L(\frac{1}{2}, \chi) \neq 0}}1 \bigg \rvert \geq \frac{(P_1(1)+P_2(1)+o(1))^2}{\lambda_1+\lambda_2+2P_1(1)P_2(1)+o(1)}  \sum_{q \in \mathcal{Q}(a,D)} \Phi \left ( \frac{q}{Q}\right)\dfrac{q}{\varphi(q)} \varphi^{+}(q),
\end{align*}
where $\lambda_1$ and $\lambda_2$ are given by $\eqref{Lambda Definition}$. Define
\begin{align}\label{C_eta definition}
c(\eta_1,\eta_2) :=  \frac{1-2(80\eta_1+6\eta_2)}{2-2(80\eta_1+6\eta_2)}.
\end{align}
Then 
\[
\lim_{\substack{\eta_1 \to 0 \\ \eta_2 \to 0}} c(\eta_1, \eta_2) = \frac{1}{2}.
\]
Also, we have
\begin{align}\label{Sandwich Inequality}
\max \frac{(P_1(1)+P_2(1))^2}{\lambda_1+\lambda_2+2P_1(1)P_2(1)} \geq \frac{4\vartheta_1 \vartheta_2}{4\vartheta_1\vartheta_2+\vartheta_1+\vartheta_2},
\end{align}
where the maximum is taken over all smooth polynomials $P_i$ with $P_i(0)=0$ and $P_i(1)=1$. For e.g., the right hand side of \eqref{Sandwich Inequality} is already obtained when $P_i(x)=x$. It follows that putting $\vartheta = \vartheta_1 = \vartheta_2 =\frac{1}{2}-\xi$,
\begin{align}\label{Sandwich Inequality II}
\max \frac{(P_1(1)+P_2(1))^2}{\lambda_1+\lambda_2+2P_1(1)P_2(1)} \geq \frac{2\vartheta}{1+2\vartheta} > c(\eta_1,\eta_2)-\varepsilon,
\end{align}
and hence, we obtain
\begin{align}\label{Main Theorem 2nd Step}
\bigg \lvert \sum_{q \in \mathcal{Q}(a,D)} \Phi\left ( \frac{q}{Q}\right) \dfrac{q}{\varphi(q)}\hspace{0.2cm}\sideset{}{^+}\sum_{\substack{\chi \bmod q \\ L(\frac{1}{2}, \chi) \neq 0}}1 \bigg \rvert > (c(\eta_1,\eta_2)-\varepsilon)\sum_{q \in \mathcal{Q}(a,D)} \Phi \left ( \frac{q}{Q}\right)\dfrac{q}{\varphi(q)} \varphi^{+}(q).
\end{align}
A similar argument holds for the odd characters, thereby completing the proof.
\end{proof}
\begin{rem}For certain tuples $(\eta_1,\eta_2)$ satisfying \eqref{Eta 1-2 Constraint}, the non-vanishing proportion obtained using \eqref{Main Theorem 2nd Step} does not improve upon the currently best known proportion for a single moduli $q$. Of course, once we let $\eta_1,\eta_2$ sufficiently small, the proportion of non-vanishing $c(\eta_1,\eta_2)$ exceeds both $\tfrac{7}{19}$ and $\tfrac{5}{13}$.
\end{rem}
\section{Proof of Theorem \ref{Main Theorem A}: Initial Setup}\label{Theorem 2 Initial Setup}
 We shall reuse some of the notations from Section \ref{subsection : Notations}. For fixed $\varepsilon>0$ arbitrarily small, we set $K=Q^{2 \eta_1+\varepsilon}$. Also, $\Phi_1$ and $\Phi_2$ are defined by \eqref{Breaking H}. Similar to the proof of Theorem \ref{Main Theorem}, we again focus our analysis on $\Phi_1$, and later show that the contribution from $\Phi_2(t)$ to both sides of \eqref{Main Theorem Inequality A} is small.  Here for $\eta_1, \eta_2 \geq 0$, we define
\begin{align}\label{ceta definition A}
\widetilde{c}(\eta_1,\eta_2) : = \frac{50}{1093}-(86\eta_1+6\eta_2).
\end{align}Our key lemma is the following.
\begin{lem}\label{Main Theorem Inequality Twisted} Let \(\eta_1, \eta_2 \geq 0\) be fixed such that 
\begin{align}\label{Eta 1-2 Constraint A1}
86\eta_1+6\eta_2 < \frac{50}{1093}. 
\end{align}
Let \(\Psi\) be a fixed nonnegative smooth function compactly supported in \([\frac{1}{2}, \frac{3}{2}]\) with \(\Psi(1) > 0\). Consider \(a,D, \textrm{ and } Q  \in \mathbb{N}\) such that $1 \leq a \leq D \leq Q^{\eta_2}$ with \((a, D) = 1\). Define the set \(\mathcal{Q}(a, D) = \{ q \in \mathbb{N} : q \equiv a \bmod D \}\). Then for any nonnegative smooth function $\phi$ with supp $\widehat{\phi} \subset (-2-\widetilde{c}(\eta_1,\eta_2),2+\widetilde{c}(\eta_1,\eta_2))$,
\begin{align}\label{Key Short Interval Estimate}
\sum_{q \in \mathcal{Q}(a,D)} \Phi_{1} \left( \frac{q}{Q}\right)\sum_{\substack{\chi\bmod q\\ \chi\:  \textup{primitive}}}\sum_{\gamma_{\chi}} \phi \left ( \frac{\log Q}{2\pi}\gamma_{\chi}\right)= \widehat{\phi}(0) \sum_{q \in \mathcal{Q}(a,D)}  \Phi_{1} \left( \frac{q}{Q}\right) \sum_{\substack{\chi\bmod q\\ \chi\: \textup{primitive}}}1+o(Q^{2} D^{-1}),
\end{align}
as $Q \to \infty$, where $\Phi_1(t)$ defined as in \eqref{Breaking H}, and $\widetilde{c}(\eta_1,\eta_2)$ is given by \eqref{ceta definition A}.
\end{lem}
Our goal from here on till Section \ref{sec: Dispersion} is to prove Lemma \ref{Main Theorem Inequality Twisted}. Following that, in Section \ref{sec: Proof of Theorem A: Final Details}, we combine Lemma \ref{Main Theorem Inequality Twisted} with a brief analysis on $\Phi_2$ to prove Theorem \ref{Main Theorem A}. To prove Lemma \ref{Main Theorem Inequality Twisted}, we follow the framework of Drappeau–Pratt–Radziwiłł \cite{DPR2023}, together with additional arguments to control the exponential oscillation which detects short intervals and allows us to execute the dispersion method. The contribution of the Kloosterman sums in the error terms will be handled via Lemmas \ref{New Drappeau's Theorem} and \ref{New Drappeau's Theorem with Theta}.

\par
\subsection{Lemmas on primes in arithmetic progressions} To proceed, we introduce the following notation as in \cite{DPR2023}. For $w \in \mathbb{N}, n \in \mathbb{Z}$ and $R \geq 1$, let
\begin{align}\label{u definition}
\mathfrak{u}_R(n, w):=\mathbbm{1}_{n \equiv 1 \bmod w}-\frac{1}{\varphi(w)} \sum_{\substack{\chi \bmod w \\ \operatorname{cond}(\chi) \leq R}} \chi(n) .
\end{align}
Note the trivial bound
\begin{align}\label{Trivial Bound}
\left|\mathfrak{u}_R(n, w)\right| \ll \mathbbm{1}_{n \equiv 1 \bmod w}+\frac{R \tau(w)}{\varphi(w)} .
\end{align}
We will require two results about primes in arithmetic progressions.
\begin{lem}\label{Primes in AP}
Let $A>0, X, Q, R \geqslant 2$ satisfy $1 \leq R \leq Q$ and $X \geqslant Q^2(\log Q)^{-A}$, and $f$ be a smooth function with support $\subseteq [\frac{1}{2},3]$ and $\left\|f^{(j)}\right\|_{\infty} \ll_j 1$. Then
$$
\sum_{Q \leq q \leq 2Q}\bigg|\sum_{n \in \mathbb{N}} f\left(\frac{n}{X}\right) \Lambda(n) \mathfrak{u}_R(n, q)\bigg| \ll Q(\log Q)^{\mathcal{O}(1)} \sqrt{X}\bigg(1+\frac{\sqrt{X}}{R Q}+\frac{X^{\frac{3}{8}}}{Q}\bigg) .
$$
The implied constant depends at most on A and the implied constants in the hypothesis.
\end{lem}
\begin{proof}
See \cite[Lemma 3]{DPR2023}.
\end{proof}
The second estimate will require the major part of our analysis. Let \(\Psi\) be a fixed nonnegative smooth function compactly supported in \([\frac{1}{2}, \frac{3}{2}]\) with \(\Psi(1) > 0\). For $k \in \mathbb{Z}$, define
\begin{align}\label{Theta tilde new Definition}
\widetilde{\Theta}_k(x,\Psi) = x \Psi(x) \mathrm{e}(kx), \quad x \in \mathbb{R}.
\end{align}
By abuse of notation, we shall write $\widetilde{\Theta}_k(x)$ to denote $\widetilde{\Theta}_k(x,\Psi)$ when $\Psi$ is clearly implied.
\begin{lem}\label{Primes in AP 2}
Let \(\eta_1, \eta_2 \geq 0\) be fixed such that 
\begin{align*}
86\eta_1+6\eta_2 < \frac{50}{1093}.
\end{align*}
Let $\kappa \in (0, \widetilde{c}(\eta_1,\eta_2))$ where $\widetilde{c}(\eta_1,\eta_2)$ is as in \eqref{ceta definition A}. Let $\varepsilon>0$ be small enough in terms of $\eta_1, \eta_2$ and $\kappa$. Suppose \(\Psi\) is a fixed nonnegative smooth function compactly supported in \([\frac{1}{2}, \frac{3}{2}]\) with \(\Psi(1) > 0\). Let $f$ be a smooth function with support $\subseteq [\frac{1}{2},3]$. Let $A>0, X, Q, W, R \geqslant 1, b,D \in \mathbb{N}$ satisfying
\begin{align*}
Q^2(\log Q)^{-A} &\ll X \ll Q^{2+\kappa}, \quad X^{\frac{11}{20}} Q^{-1+2\eta_1+\eta_2} \leq R \leq Q^{\frac{2-10\eta_1-4\eta_2}{3}} X^{-\frac{2}{9}}, \notag \\
b &\leq Q^{\eta_1+\eta_2+\varepsilon}, \quad D=\lfloor Q^{\eta_2}\rfloor \quad \textrm{and} \quad Q^{1-2\eta_1-2\eta_2-\varepsilon} \ll W \ll Q.
\end{align*}
Furthermore, assume that $\left\|f^{(j)}\right\|_{\infty},\left\|\Psi^{(j)}\right\|_{\infty} \ll_j 1$. Then for any $\lvert k \rvert \leq K$ and $w_0 \in (\mathbb{Z}/D\mathbb{Z})^{\times}$,
\begin{align}\label{Lemma 6.3 Statement}
\sum_{w \equiv w_0 \bmod D} \widetilde{\Theta}_k\left(\frac{w}{W}\right) \sum_{n \in \mathbb{N}} \Lambda(n) f\left(\frac{n}{X}\right) \mathfrak{u}_R(n, b w) \ll Q^{1-2\eta_1-\eta_2-\varepsilon} \sqrt{X},
\end{align}
where $\widetilde{\Theta}_k(x)$ and $\mathfrak{u}_R(n, w)$ are given by \eqref{Theta tilde new Definition} and \eqref{u definition} respectively. The implied constant in \eqref{Lemma 6.3 Statement} depends at most on $\eta_1, \eta_2,\kappa, \varepsilon$ and the implied constants in the hypotheses.
\end{lem} 
We will prove Lemma \ref{Primes in AP 2} in Section \ref{sec: Dispersion}. In the remaining part of this section, we prove Lemma \ref{Main Theorem Inequality Twisted} assuming Lemma \ref{Primes in AP 2}.
\subsection{Explicit formula and Character Orthogonality} We let $\kappa \in (0, \widetilde{c}(\eta_1,\eta_2))$ where $\widetilde{c}(\eta_1,\eta_2)$ is given by \eqref{ceta definition A}, such that $\textrm{supp } \widehat{\phi} \subset (-2-\kappa, 2+\kappa). $
In what follows, all implied constants are allowed to depend on $\eta_1, \eta_2, \kappa, \Psi$ and $\phi$. To prove Lemma \ref{Main Theorem Inequality Twisted}, it suffices to work with $D = \lfloor Q^{\eta_2} \rfloor$. Following \cite[Section 2.2]{DPR2023}, we rewrite the left-hand side of \eqref{Key Short Interval Estimate} by applying the explicit formula, e.g. \cite[Theorem 2.2]{S1998}. For $q>1$ and $\chi \bmod q$ primitive,
\begin{align*}
\sum_{\substack{\rho \in \mathbb{C} \\
\operatorname{Re}(\rho) \in(0,1) \\
L(\rho, \chi)=0}} &\phi\bigg(\frac{\left(\rho-\frac{1}{2}\right) \log Q}{2 \pi i}\bigg) =\mathcal{O}\left(\frac{1}{\log Q}\right)+\widehat{\phi}(0) \frac{\log q}{\log Q}-\frac{1}{\log Q} \sum_{n \geqslant 1}(\chi(n)+\overline{\chi}(n)) \frac{\Lambda(n)}{\sqrt{n}} \widehat{\phi}\left(\frac{\log n}{\log Q}\right).
\end{align*}
Let $\widetilde{\Phi}_1(x)=\Phi_1(x) x$. Summing
the above relation over $\chi$ and $q$, it is enough to show that
\begin{align*}
S_\phi(Q)&=\sum_{q \in \mathcal{Q}(a,D)} \frac{1}{q} \widetilde{\Phi}_1\left(\frac{q}{Q}\right) \sum_{\substack{\chi \bmod q \\ \text { primitive }}} \frac{1}{\log Q} \sum_{n \geqslant 1}(\chi(n)+\overline{\chi}(n)) \frac{\Lambda(n)}{\sqrt{n}} \widehat{\phi}\left(\frac{\log n}{\log Q}\right)=o(Q^{1-\eta_2}).
\end{align*}
Using character orthogonality (see third display in the proof of \cite[Lemma 4.1]{BM2011}), we write $S_\phi(Q)$ as
\begin{align}\label{Bui-Milinovich Lemma}
S_\phi(Q)=\frac{2}{\log Q} \underset{\substack{v,w \\ vw \in \mathcal{Q}(a,D)}}{\sum \sum }
\widetilde{\Phi}_1 \left(\frac{v w}{Q}\right) \frac{\mu(v)}{v} \frac{\varphi(w)}{w} \sum_{n \equiv 1 \bmod w} \frac{\Lambda(n)}{\sqrt{n}} \widehat{\phi}\left(\frac{\log n}{\log Q}\right) .
\end{align}
Let $V$ be any smooth function supported in $[\frac{1}{2},3]$ generating the partition of unity, that is, 
$
\sum_{j \in \mathbb{Z}} V\left(\frac{x}{2^j}\right)=1,
$
for all $x>0$. Inserting this in \eqref{Bui-Milinovich Lemma}, we obtain
\begin{align*}
S_\phi(Q)=\frac{2}{\log Q} &\sum_{\substack{j \in \mathbb{Z} \\ 1 / 2 \leq X:=2^j \leq 2 Q^{2+\kappa}}}  \underset{\substack{v,w \\ vw \in \mathcal{Q}(a,D)}}{\sum \sum } \widetilde{\Phi}_1\left(\frac{v w}{Q}\right) \frac{\mu(v)}{v} \frac{\varphi(w)}{w}  \notag \\
&\times \sum_{n \equiv 1 \bmod w} \frac{\Lambda(n)}{\sqrt{n}} V\left(\frac{n}{X}\right) \widehat{\phi}\left(\frac{\log n}{\log Q}\right) .
\end{align*}
We set $f_j(x)=x^{-\frac{1}{2}} V(x) \widehat{\phi}\left(\frac{\log (2^j x )}{\log Q}\right),$ for $\frac{1}{2} \leq 2^j \leq 2 Q^{2+\kappa}.$ Differentiating with respect to $x$, for all $\ell \geq 0$, there exists $C_{\phi, \ell} \geqslant 0$ such that $\|f_j^{(\ell)}\|_{\infty} \leq C_{\phi, \ell}$ for all $j$. It follows that
\begin{align}\label{Eq: S(Q) Simplification}
S_\phi(Q) \ll \sup _{1 \ll X \ll Q^{2+\kappa}} X^{-\frac{1}{2}} \sup _f|T(Q, X)|,
\end{align}
where $f$ varies among smooth functions supported in $[\frac{1}{2},3]$ satisfying $\left\|f^{(\ell)}\right\|_{\infty} \leq C_{\phi, \ell}$ and
\begin{align}\label{Eq: T(Q,X) Definition}
T(Q,X):=  \underset{\substack{v,w \\ vw \in \mathcal{Q}(a,D)}}{\sum \sum } \widetilde{\Phi}_1\left(\frac{v w}{Q}\right) \frac{\mu(v)}{v} \frac{\varphi(w)}{w} \sum_{n \equiv 1 \bmod w} \Lambda(n)f\left(\frac{n}{X}\right).
\end{align}
\par
\subsection{Trivial Range} Similar to \cite[Section 2.3]{DPR2023}, we handle the small values of $X$ by the trivial bound
$$
\sum_{n \equiv 1 \bmod w} \Lambda(n) f\left(\frac{n}{X}\right) \ll \log Q \sum_{\substack{X / 2<n<3 X \\ n \neq 1, n \equiv 1 \bmod w}} 1 \ll \frac{X \log Q}{w}.
$$
Substituting the above bound in \eqref{Eq: T(Q,X) Definition} and using the trivial bound  $ \Phi \left(vw/Q\right) \ll Q^{\eta_1}$, we have
\begin{align}\label{Eq: T(Q,X) with Divisor}
T(Q, X) &\ll\frac{X \log Q}{Q} 
\underset{\substack{v,w \\ vw \in \mathcal{Q}(a,D)}}{\sum \sum }  \widetilde{\Phi}_1\left(\frac{v w}{Q}\right)  \frac{\varphi(w)}{w} \notag \\
&\ll \frac{X \log Q}{Q} \sum_{\substack{ q \sim Q \\ q \in \mathcal{Q}(a,D)}} \Phi\left (\frac{q}{Q} \right)\tau(q) +\frac{X \log Q}{Q} \sum_{\substack{ q \sim Q \\ q \in \mathcal{Q}(a,D)}} \bigg \lvert \Phi_2 \left (\frac{q}{Q} \right) \bigg \rvert \tau(q) \notag \\
&\ll \frac{X \cdot Q^{\eta_1} \log Q}{Q} \sum_{\substack{ \lvert q- Q \rvert \leq Q^{1-\eta} \\ q \in \mathcal{Q}(a,D)}} \tau(q) +\frac{X \log Q}{Q} \sum_{\substack{ q \sim Q \\ q \in \mathcal{Q}(a,D)}}\bigg( \sum_{\lvert k \rvert >K}\lvert b(k) \rvert \bigg) \tau(q).
\end{align}
The first sum in \eqref{Eq: T(Q,X) with Divisor} is over the divisor function in short intervals and in arithmetic progressions. Using \cite[Lemma 3.2]{D2017}, we obtain the bound
\begin{align}\label{Divisor Sum: Short Intervals and AP}
\sum_{\substack{ \lvert q-Q \rvert \leq Q^{1-\eta_1} \\ q \in \mathcal{Q}(a,D)}} \tau(q) \ll Q^{1-\eta_1-\eta_2} (\log Q).  
\end{align}
By \eqref{Choice of K}, \eqref{Eq: T(Q,X) with Divisor} and \eqref{Divisor Sum: Short Intervals and AP}, we arrive at
\[ T(Q, X) \ll X Q^{-\eta_2}( \log Q)^2+\frac{X Q^{2\eta_1}\log Q}{QK} \sum_{\substack{ q \sim Q \\ q \in \mathcal{Q}(a,D)}} \tau(q) \ll X Q^{-\eta_2}( \log Q)^2.
\]
This bound is acceptable if $X$ is small, say $X \ll Q^{2}(\log Q)^{-6}$. Therefore, in what follows, we aim to show that for
$
Q^{2}(\log Q)^{-6} \ll X \ll Q^{2+\kappa},
$
we have
$
T(Q, X) \ll_{\varepsilon} Q^{1-\eta_2-\varepsilon} \sqrt{X}.
$
\subsection{Subtracting the main term.} We precisely follow the arguments from \cite[Section 2.4]{DPR2023}. Breaking $\Phi_1$ into $\Phi$ and $\Phi_2$ and using the trivial bound $ \Phi \left(vw/Q\right) \ll Q^{\eta_1},$ we have
\begin{align*}
\underset{\substack{v,w \\ vw \in \mathcal{Q}(a,D)}}{\sum \sum }  \widetilde{\Phi}_1\left(\frac{v w}{Q}\right) \frac{\mu(v)}{v} \frac{\varphi(w)}{w} \sum_{\substack{n \equiv 1 \bmod w \\
(n, v)>1}} \Lambda(n) f\left(\frac{n}{X}\right) \ll_{\varepsilon} Q^{1-\eta_2+\varepsilon}.
\end{align*}
Therefore we may insert the coprimality condition $(n, v)=1$ to \eqref{Eq: T(Q,X) Definition}, arriving at
$$
T(Q, X)= \underset{\substack{v,w \\ vw \in \mathcal{Q}(a,D)}}{\sum \sum } \widetilde{\Phi}_1\left(\frac{v w}{Q}\right) \frac{\mu(v)}{v} \frac{\varphi(w)}{w} \sum_{\substack{n \equiv 1 \bmod w \\(n, v)=1}} \Lambda(n) f\left(\frac{n}{X}\right)+\mathcal{O}_{\varepsilon}\left(Q^{1-\eta_2+\varepsilon}\right).
$$
Let $1 \leq R<Q / 2$ so that $R<v w$ for any $v, w$ appearing in the sum. We replace the condition $n \equiv 1 \bmod w$ by $\mathfrak{u}_R(n, w)$. The difference is
\begin{align*}
 \underset{\substack{v,w \\ vw \in \mathcal{Q}(a,D)}}{\sum \sum }& \widetilde{\Phi}_1\left(\frac{v w}{Q}\right) \frac{\mu(v)}{v} \frac{1}{w} \sum_{\substack{(n, v)=1}} \Lambda(n) f\left(\frac{n}{X}\right) \sum_{\substack{\chi \bmod w \\ r=\operatorname{cond}(\chi) \leq R}} \chi(n) \\
&=\sum_{q \in \mathcal{Q}(a,D)} \frac{1}{q} \widetilde{\Phi}_1\left(\frac{q}{Q}\right) \sum_{\substack{\chi \bmod q \\ r=\operatorname{cond}(\chi) \leq R,\, r \mid q}} \sum_{\substack{(n, q)=1}} \Lambda(n) f\left(\frac{n}{X}\right) \chi(n) \sum_{v \mid q / r} \mu(v)=0,
\end{align*}
since $r<q$ by our choice of $R$. It follows that
$$
T(Q, X)= \underset{\substack{v,w \\ vw \in \mathcal{Q}(a,D)}}{\sum \sum } \widetilde{\Phi}_1\left(\frac{v w}{Q}\right) \frac{\mu(v)}{v} \frac{\varphi(w)}{w} \sum_{\substack{(n, v)=1}} \Lambda(n) f\left(\frac{n}{X}\right)\mathfrak{u}_R(n, w)+\mathcal{O}_{\varepsilon}\left(Q^{1-\eta_2+\varepsilon}\right).
$$
We may now remove the coprimality condition on $n$ using the trivial bound \eqref{Trivial Bound}. Since $R \ll Q$, it's easy to check that the arising error terms are acceptable. We get
$$
T(Q, X)=T(Q, X, R)+\mathcal{O}_{\varepsilon}\left(Q^{1-\eta_2+\varepsilon}\right),
$$
where
$$
\begin{aligned}
T(Q, X, R) & := \underset{\substack{v,w \\ vw \in \mathcal{Q}(a,D)}}{\sum \sum } \widetilde{\Phi}_1\left(\frac{v w}{Q}\right) \frac{\mu(v)}{v} \frac{\varphi(w)}{w} \Delta(w) \quad \textrm{and} \quad \Delta(w) :=\sum_n \Lambda(n) f\left(\frac{n}{X}\right) \mathfrak{u}_R(n, w) .
\end{aligned}
$$
\subsection{Critical Range} We now impose the additional conditions
\begin{align}\label{Hypothesis for R}
Q^{\frac{\kappa}{2} +\eta_1+\eta_2+\varepsilon} \leq R \leq Q^{\frac{1}{2}}, \quad \textrm{and} \quad \kappa<\frac{8}{3}\left(1-\eta_1-\eta_2 \right)-2 .
\end{align}
Let $B \in [1, Q^{\frac{1}{2}}]$ be a parameter. In $T(Q, X, R)$, we write $\frac{\varphi(w)}{w}=\sum_{b \mid w} \frac{\mu(b)}{b}$ to obtain
$$
\begin{aligned}
T(Q, X, R) & \leq \sum_{ \substack{b, v \\ (b,D)= (v,D)=1}} \frac{1}{b v}\bigg| \sum_{w \equiv a \overline{bv} \bmod D}  \widetilde{\Phi}_1\left(\frac{b v w}{Q}\right) \Delta(b w)\bigg| \\
& \ll(\log B)^2 \sup _{b, v \leq B}\bigg|\sum_{w \equiv a \overline{bv} \bmod D}  \widetilde{\Phi}_1 \left(\frac{b v w}{Q}\right) \Delta(b w)\bigg|+E_1+E_2,
\end{aligned}
$$
where $E_1$ (resp. $E_2$ ) corresponds to the sum over $b, v$ restricted to $b>B$ (resp. $\left.v>B\right)$. We have
\begin{align*}
E_1 & \ll Q^{\eta_1}\sum_{\substack{b,v, w \\
bvw \leq 3 Q, \, b>B}} \frac{1}{bv}|\Delta(b w)| \ll_{\varepsilon} B^{-1} Q^{\eta_1+\varepsilon} \sum_{q \leq 3 Q}|\Delta(q)| \ll_{\varepsilon} Q^{1+\eta_1+\varepsilon} \sqrt{X} B^{-1},
\end{align*}
using Lemma \ref{Primes in AP} and our hypothesis \eqref{Hypothesis for R}. On the other hand, we obtain
\begin{align*}
E_2 &\ll_{\varepsilon} Q^{\eta_1+\varepsilon} \sum_{\substack{b, w \\
b w \leq 3Q/B}} \frac{|\Delta(b w)|}{b}  \ll_{\varepsilon} Q^{\eta_1+\varepsilon} \sum_{q \leq 3Q/B}|\Delta(q)|  \ll_{\varepsilon} Q^{1+\varepsilon} \sqrt{X}\left(\frac{Q^{\eta_1}}{B}+\frac{Q^{\eta_1}\sqrt{X}}{RQ}+\frac{Q^{\eta_1} X^{\frac{3}{8}}}{Q}  \right ),
\end{align*}
again using Lemma \ref{Primes in AP}. By our conditions \eqref{Hypothesis for R}, if we let $B=Q^{\eta_1+\eta_2+\varepsilon}$, then both $E_1$ and $E_2$ are $\ll_{\varepsilon} Q^{1-\eta_2-\varepsilon}\sqrt{X}$. Therefore it will suffice to show that
$$
\sum_{w \equiv w_0 \bmod D}  \widetilde{\Phi}_1 \left(\frac{b v w}{Q}\right) \Delta(b w) \ll_{\varepsilon} Q^{1-\eta_2-\varepsilon} \sqrt{X}
$$
uniformly for $b, v \leq Q^{\eta_1+\eta_2+\varepsilon}$ and $w_0 \in (\mathbb{Z}/D\mathbb{Z})^{*}$. We further write 
\begin{align*}
\sum_{w \equiv w_0 \bmod D}  \widetilde{\Phi}_1 \left(\frac{b v w}{Q}\right) \Delta(b w)
&=\sum_{\lvert k \rvert \leq K} b(k) \sum_{w \equiv w_0 \bmod D}  \widetilde{\Theta}_k\left(\frac{bvw}{Q}\right) \Delta(b w),
\end{align*}
where $\widetilde{\Theta}_k(x)$ is given by \eqref{Theta tilde new Definition}. To prove Lemma \ref{Main Theorem Inequality Twisted}, it is enough to show that for each $\lvert k \rvert \leq K$, and
$
Q^{2}(\log Q)^{-6} \ll X \ll Q^{2+\kappa},
$
\[
\sum_{w \equiv w_0 \bmod D}  \widetilde{\Theta}_k\left(\frac{bvw}{Q}\right) \Delta(b w) \ll_{\varepsilon}  Q^{1-2\eta_1-\eta_2-\varepsilon} \sqrt{X}, 
\]
where the implied constant depends at most on $\eta_1,\eta_2, \kappa, \Psi, \phi$ and $\varepsilon$.
\par
At this stage, we invoke Lemma \ref{Primes in AP 2}. It's easy to verify that under the conditions \eqref{Hypothesis for R}, we can choose $R$ which also satisfies Lemma \ref{Primes in AP 2}. Choosing $W=Q(bv)^{-1}$ in Lemma \ref{Primes in AP 2}, we arrive at our desired conclusion. This completes the proof of Lemma \ref{Main Theorem Inequality Twisted}.
\section{Proof of Lemma \ref{Primes in AP 2} : Kloosterman Sum Bounds and the Dispersion Method}\label{sec: Dispersion}
Our main goal in this section is to prove Lemma \ref{Primes in AP 2}. We first collect two results from \cite{DPR2023} regarding the combinatorial decomposition of the von Mangoldt function.

\begin{lem}\label{Combinatorial Decomposition 1}
Let $\left\{t_j\right\}_{1 \leq j \leq J} \in \mathbb{R}$ be nonnegative real numbers such that $\sum_j t_j=1$. Let $\lambda, \sigma, \delta \geq 0$ be real numbers such that
\begin{align*}
\delta<\frac{1}{12}, \quad \sigma \leq \frac{1}{6}-\frac{1}{2} \delta \quad  \textrm{and} \quad 2 \lambda+\sigma<\frac{1}{3}.
\end{align*}
Then at least one of the following must occur:
\begin{itemize}
    \item Type $d_1 :$ There exists $t_j$ with $t_j \geq \frac{1}{3}+\lambda$.
    \item Type $d_2 :$ There exist $i, j, k$ such that $\frac{1}{3}-\delta<t_i, t_j, t_k<\frac{1}{3}+\lambda$, and
    $$
\sum_{t_j^* \notin\left\{t_i, t_j, t_k\right\}} t_j^*<\sigma .$$
    \item Type $\mathrm{II} :$ There exists $S \subset\{1, \ldots, J\}$ such that
$$
\sigma \leq \sum_{j \in S} t_j \leq \frac{1}{3}-\delta .
$$ 
\end{itemize}
\end{lem}
\begin{proof}
See \cite[Lemma 9]{DPR2023}.
\end{proof}
\begin{coro}\label{Cor: Decomposition}
Let $f$ be a smooth function supported in $[\frac{1}{2},3]$, $u: \mathbb{N} \rightarrow \mathbb{C}$ be any map, and $X \geqslant 1$. Then there exists a sequence $\left(C_j\right)_{j \geqslant 0}$ of positive numbers, depending only on $f$, such that
$$
\bigg|\sum_{n \in \mathbb{N}} \Lambda(n) f\left(\frac{n}{X}\right) u(n)\bigg| \ll(\log X)^8\left(T_1+T_2+T_{\mathrm{II}}\right),
$$
where
\begin{align}
T_1&=\sup _{\substack{N \gg X^{1 / 3+\lambda} \\ M N \asymp X}} \sup _{\substack{g \in \mathcal{G} \\ \beta \in \mathcal{S}}}\bigg|\sum_{\substack{n \in \mathbb{N} \\ m \sim M}} g\left(\frac{n}{N}\right) \beta_m u(m n)\bigg|, \label{T1 Def}\\
T_2&=\sup _{\substack{X^{1 / 3-\delta} \ll N_2 \leq N_1 \ll X^{1 / 3+\lambda} \\ M N_1 N_2 \asymp X}} \sup _{\substack{g_1, g_2 \in \mathcal{G} \\ \beta \in \mathcal{S}}}\bigg|\sum_{\substack{n_1, n_2 \in \mathbb{N} \\ m \sim M}} g_1\left(\frac{n_1}{N_1}\right) g_2\left(\frac{n_2}{N_2}\right) \beta_m u\left(m n_1 n_2\right)\bigg|, \label{T2 Def}\\
T_{\mathrm{II}}&=\sup _{\substack{X^\sigma \ll N \ll X^{1 / 3-\delta} \\
M N \asymp X}} \sup _{\alpha, \beta \in \mathcal{S}}\bigg| \underset{\substack{n \sim N \\ m \sim M}}{\sum \sum} \sum_m \alpha_m \beta_n u(m n)\bigg| \label{T3 Def},
\end{align}
the implied constants are absolute, $\mathcal{G}$ is the set of smooth functions $g$ supported in $[\frac{1}{2},3]$ satisfying $\left\|g^{(j)}\right\|_{\infty} \leq C_j$, and $\mathcal{S}$ is the set of sequences $\left(\beta_n\right)$ satisfying $\left|\beta_n\right| \leq \tau(n)^8$.
\end{coro}
\begin{proof}
    See \cite[Corollary 10]{DPR2023}.
\end{proof}
We write $X=Q^{2+\omega}$ so that 
\[
-o(1) \leq \omega \leq \kappa+o(1), \quad \textrm{as} \quad Q \to \infty.\]
We now carry out the analysis for $T_1, T_2$ and $T_{\mathrm{II}}$, where we choose
$$
u(n):=\underset{w}{\sum\nolimits^*} \widetilde{\Theta}_k\left(\frac{w}{W}\right) \mathfrak{u}_R(n, b w).
$$
Here the notation $\sum_w\nolimits^*$ indicates that the sum is over $w \equiv w_0 \bmod D$. We also write $R=X^\rho.$
\subsection{Type $d_1$ sums.} We suppose $M$ and $N$ are given as in \eqref{T1 Def}. We want to bound
\begin{align*}
T_1(M, N)=\underset{w}{\sum\nolimits^*}  \widetilde{\Theta}_k\left(\frac{w}{W}\right) \sum_{\substack{m \sim M \\(m, b w)=1}} \beta_m\bigg (&\sum_{\substack{n \in \mathbb{N} \\ m n \equiv 1 \bmod b w}} g\left(\frac{n}{N}\right) \\
&\quad -\frac{1}{\varphi(b w)} \sum_{\substack{\chi \bmod b w \\ \operatorname{cond}(\chi) \leq R}} \chi(m) \sum_{\substack{(n, b w)=1}} \chi(n) g\left(\frac{n}{N}\right)\bigg ) .
\end{align*}
By Poisson summation and bounds on Gauss sums \cite[Lemma 3.2]{IK2004}, we have for any fixed $A>0$,
\begin{align*}
\sum_{n=\overline{m} \bmod b w} g\left(\frac{n}{N}\right)&=\frac{N}{b w} \widehat{g}(0)+\frac{N}{b w} \sum_{0<|h| \leq H} \widehat{g}\left(\frac{N h}{b w}\right) \mathrm{e}\left(\frac{\overline{m} h}{b w}\right)+\mathcal{O}_{\varepsilon,A}\left(Q^{-A}\right), \\
\textrm{and} \quad \frac{1}{\varphi(b w)} \sum_{(c, b w)=1} \chi(c) g\left(\frac{c}{N}\right)&=\frac{N}{b w} \widehat{g}(0) \mathbbm{1}\left(\chi=\chi_0\right)+\mathcal{O}_{\varepsilon}\bigg(\frac{Q^{\varepsilon} R^{\frac{1}{2}}}{bW}\bigg),
\end{align*}
where $H = bW^{1+\varepsilon}N^{-1}$. Therefore we obtain
$$
T_1(M, N)=\frac{N}{b} \underset{w}{\sum\nolimits^*} \frac{1}{w} \widetilde{\Theta}_k\left(\frac{w}{W}\right) \sum_{\substack{(m, b w)=1 \\ m \sim M}} \beta_m \sum_{0<|h| \leq H} \widehat{g}\left(\frac{N h}{b w}\right) \mathrm{e}\left(\frac{\overline{m} h}{b w}\right)+\mathcal{O}_{\varepsilon}(M R^{\frac{3}{2}} Q^{-\eta_2+\varepsilon}) .
$$
By reciprocity, we get
$$
\frac{\overline{m} h}{b w} \equiv-\frac{\overline{b w} h}{m}+\frac{h}{m b w} \bmod 1,
$$
which implies that
\begin{align*}
T_1(M, N)&=\frac{N}{b}  \underset{w}{\sum\nolimits^*} \frac{1}{w} \widetilde{\Theta}_k\left(\frac{w}{W}\right) \sum_{\substack{(m, b w)=1 \\ m \sim M}} \beta_m \sum_{0<|h| \leq H} \widehat{g}\left(\frac{N h}{b w}\right) \mathrm{e}\left(-\frac{\overline{b w} h}{m}\right) \notag \\
&\quad +\mathcal{O}_{\varepsilon}(M R^{\frac{3}{2}} Q^{-\eta_2+\varepsilon}+Q^{1-\eta_2+\varepsilon} N^{-1}).
\end{align*}
We detect the condition $w \equiv w_0 \bmod D$ using additive characters to arrive at
\begin{align}\label{Detecting Additive Characters}
\frac{N}{bD W} \sum_{\ell=0}^{D-1} e \left ( - \frac{\ell w_0}{D} \right) \sum_{\substack{(m, b)=1 \\ m \sim M}} \beta_m \sum_{0<|h| \leq H} \sum_{(w, m)=1} \widehat{g}\left(\frac{N h}{b w}\right) \frac{W}{w} \widetilde{\Theta}_k\left(\frac{w}{W}\right) \mathrm{e}\left(\frac{\overline{b w} h}{m}+ \frac{\ell w}{D}\right) .
\end{align}
The sum 
\begin{align}\label{Kloosterman Sum}
\sum_{\substack{(w,m)=1, \, w \asymp W}} \mathrm{e}\left(\frac{\overline{b w} h}{m}+ \frac{\ell w}{D}\right)
\end{align}
is a Kloosterman sum with modulus $[m,D]$. Using standard techniques for completing Kloosterman sums, the sum in \eqref{Kloosterman Sum} is bounded by 
\[
\ll \frac{W}{[m,D]} \cdot \tau([m,D]) D^{\frac{3}{2}}(h,m)^{\frac{1}{2}} [m,D]^{\frac{1}{2}}+ \tau([m,D]) D^{\frac{3}{2}}(h,m)^{\frac{1}{2}} [m,D]^{\frac{1}{2}} (\log [m,D]).
\]
Here we used $(\overline{b}h[m,D]/m,[m,D]) \leq D^3 (h,m)$. By partial summation, the sum over $w$ in \eqref{Detecting Additive Characters} is
$$
\ll_{\varepsilon} Q^{2\eta_1+\varepsilon} D^2 M^{\frac{1}{2}} (h,m)^{\frac{1}{2}} \left( W M^{-1}+1\right)  .
$$
Summing over $h, m$ and $\ell$, we obtain the bound
$$
T_1(M, N) \ll_{\varepsilon} M^{\frac{1}{2}} Q^{1+2\eta_1+2\eta_2+\varepsilon}+M^{\frac{3}{2}} Q^{2\eta_1+2\eta_2+\varepsilon}+M R^{\frac{3}{2}} Q^{-\eta_2+\varepsilon}+Q^{1-\eta_2+\varepsilon} N^{-1} .
$$
This bound is acceptably small provided
$$
N \gg\left(\frac{X}{Q}\right)^{\frac{2}{3}}Q^{\frac{4\eta_1}{3}+2\eta_2+\varepsilon}=X^{\frac{1}{3}+\frac{1}{3} \frac{\omega+4\eta_1+6\eta_2}{(2+\omega)}+\frac{\varepsilon}{2+\omega}} \quad \text { and } \quad N \gg \frac{X^{\frac{1}{2}} R^{\frac{3}{2}} Q^{\eta_1+\varepsilon}}{Q}=X^{\frac{1}{2} \frac{\omega+2\eta_1}{2+\omega}+\frac{3}{2} \rho+\frac{\varepsilon}{2+\omega}} \text {. }
$$
These inequalities are satisfied for all sufficiently small $\varepsilon>0$, under the assumptions
\begin{align}\label{Eq: Conditions on Lambda and Eta}
\lambda>\frac{\omega+8\eta_1+6\eta_2}{3(2+\omega)} \quad  \text { and } \quad \rho<\frac{4+\omega+4\eta_1+12\eta_2}{9(2+\omega)} .
\end{align}
Hence we conclude the following.
\begin{lem}\label{Dispersion I}
Under the notation and hypotheses of Corollary \ref{Cor: Decomposition}, and assuming \eqref{Eq: Conditions on Lambda and Eta}, we have
$$
T_1 \ll_{\varepsilon} Q^{1-2\eta_1-\eta_2-\varepsilon} \sqrt{X} .
$$
The implied constant may also depend at most on $\lambda, \rho$ and $\omega$.
\end{lem}
\subsection{Type $d_2$ sums.} For brevity, we rename $(N_1, N_2, M)$ into $(M, N, L)$ so that $M N L \asymp X$. We have
$$
\begin{aligned}
T_2(M, N, L)=\sum_{\ell \sim L} \beta_{\ell} \underset{(w, \ell)=1}{\sum\nolimits^*} \widetilde{\Theta}_k\left(\frac{w}{W}\right)   
 & \bigg(\underset{\substack{m, n \\
\ell m n \equiv 1 \bmod b w}}{\sum \sum} g_1\left(\frac{m}{M}\right) g_2\left(\frac{n}{N}\right)\\
& -\frac{1}{\varphi(b w)} \sum_{\substack{\chi \bmod b w \\
\operatorname{cond}(\chi) \leq R}} \chi(\ell) \underset{\substack{m,n \\ (m n, b w)=1}}{\sum \sum} g_1\left(\frac{m}{M}\right) g_2\left(\frac{n}{N}\right) \chi(m n)\bigg) .
\end{aligned}
$$
By Poisson summation on the $m$-sums, we get for any fixed $A>0$,
$$
\begin{aligned}
\sum_{m \equiv \overline{\ell n} \bmod b w} g_1\left(\frac{m}{M}\right) & =\frac{M}{b w} \sum_{|h| \leq H} \widehat{g}_1\left(\frac{M h}{b w}\right) \mathrm{e} \bigg(\frac{\overline{\ell n} h}{b w}\bigg)+\mathcal{O}_{\varepsilon,A}\left(Q^{-A}\right) \\
\textrm{and} \quad \frac{1}{\varphi(b w)}\sum_{(m, b w)=1} \chi(m) g_1\left(\frac{m}{M}\right) & =\frac{M}{b w} \widehat{g}_1(0) \mathbbm{1}\left(\chi=\chi_0\right)+\mathcal{O}_{\varepsilon}\bigg( \frac{Q^{\varepsilon} R^{\frac{1}{2}}}{bW}\bigg),
\end{aligned}
$$
where $H=bW^{1+\varepsilon} M^{-1}$. The contribution of the error terms is
$
\ll_{\varepsilon} L N R^{\frac{3}{2}} Q^{-\eta_2+\varepsilon}.
$
The zero frequency from the Poisson summation cancels out. Next, we separate the variables $h$ and $w$ by writing
$$
\widehat{g}_1\left(\frac{M h}{b w}\right)=\frac{w}{M} \int_{\mathbb{R}} g_1\left(\frac{w y}{M}\right) e\left(-\frac{h y}{b}\right) \dd y.
$$
The integral is restricted to $y \asymp M / W$. We interchange integral and summation, and break the sum over $n$ into residue classes $n_0 \bmod D$ to see that $T_2(M, N, L)$ is 
\begin{align*}
& \ll \frac{MD}{b W} \sup_{\substack{y \asymp M / W \\ n_0 \in (\mathbb{Z}/D\mathbb{Z})^{*} }}\bigg|\sum_{\ell} \beta_{\ell} \sum_{0<|h| \leq H} \mathrm{e}\left(-\frac{h y}{b}\right) \underset{\substack{w \\ (bw, \ell)=1}}{\sum\nolimits^*}  \underset{\substack{n \\ (n,bw)=1 }}{\sum\nolimits^*} \widetilde{\Theta}_k\left(\frac{w}{W}\right) g_1\left(\frac{w y}{M}\right) g_2\left(\frac{n}{N}\right) \mathrm{e}\left(-h\frac{\overline{\ell n}}{bw}\right)\bigg| \notag \\
&\quad \hspace{0.5cm}+\mathcal{O}_{\varepsilon}(L N R^{\frac{3}{2}} Q^{-\eta_2+\varepsilon}) .
\end{align*}
Here the notation $\sum_n\nolimits^*$ indicates that the sum is over $n \equiv n_0 \bmod D$. Now we use Lemma \ref{New Drappeau's Theorem} with the following dictionary:
\begin{align*}
&c=w, \quad d=n, \quad n =h, \quad, r = \ell, \quad s=b',\\
&\boldsymbol{C} \sim W,\quad \boldsymbol{D} \sim N, \quad \boldsymbol{N} \sim H, \quad \boldsymbol{R} \sim L, \quad \boldsymbol{S} \sim b,\\
&g(c,d,n,r,s) := g(w,n) = \widetilde{\Theta}_k\left(\frac{w}{W}\right) g_1\left(\frac{w y}{M}\right) g_2\left(\frac{n}{N}\right),\\
&b_{n,r,s} = \mathbbm{1}_{b^{\prime}=b} \cdot {\mathrm{e}}(-h y / b) \beta_{\ell} , \\
&\varepsilon_1 = 2\eta_1+\varepsilon, \quad \varepsilon_2 = \varepsilon, \quad \varepsilon_3,\varepsilon_4,\varepsilon_5 = 0.
\end{align*}
Before estimating the final bound, we note that $H \ll L$ and $N \ll_{\varepsilon} W^{1+\varepsilon}$, if we let
\[
\lambda < \frac{1}{6}-\frac{\omega+4\eta_1+4\eta_2}{2(2+\omega)}.
\]
Using Lemma \ref{New Drappeau's Theorem} , we find the bounds
\begin{align*}
K(\boldsymbol{C}, \boldsymbol{D}, \boldsymbol{N}, \boldsymbol{R}, \boldsymbol{S}) &\ll \left (  D (bW(bL+H)(W+LN)+W^2Nb \sqrt{(bL+H)L} +N^2LH\right)^{\frac{1}{2}} \notag \\
& \ll bD^{\frac{1}{2}}  \left (WL^{\frac{1}{2}} +(WN)^{\frac{1}{2}}L+W(LN)^{\frac{1}{2}}+b^{-1}N(LH)^{\frac{1}{2}} \right)
\end{align*}
and $\left\|b_{\boldsymbol{N}, \boldsymbol{R}, \boldsymbol{S}}\right\|_2 \ll_{\varepsilon} L^{\varepsilon}(H L)^{\frac{1}{2}} $. It follows that
\begin{align*}
T_2(M, N, L) &\ll_{\varepsilon} L N R^{\frac{3}{2}} Q^{-\eta_2+\varepsilon}+ \bigg \{ Q^{\varepsilon} \frac{MD}{b W} \cdot \left ( WNHLb \right)^{24 \eta_1} \cdot D^{\frac{3}{2}} \cdot (LH)^{\frac{1}{2}} \cdot bD^{\frac{1}{2}} \notag\\
&\quad \times \left (WL^{\frac{1}{2}} +(WN)^{\frac{1}{2}}L+W(LN)^{\frac{1}{2}}+b^{-1}N(LH)^{\frac{1}{2}} \right) \bigg \} \notag \\
&\ll_{\varepsilon} L N R^{\frac{3}{2}} Q^{-\eta_2+\varepsilon}+Q^{\varepsilon}\left ( WNHLb \right)^{24 \eta_1} D^3 b^{\frac{1}{2}} X^{\frac{1}{2}}  (L+(WL)^{\frac{1}{2}}).
\end{align*}
Assuming $\delta<\frac{1}{12}$, we have $WNHLb \ll Q^4$, which implies that the contribution from the above error terms is acceptable provided
$$
\begin{gathered}
M \gg X^{\frac{\omega+4\eta_1}{2(2+\omega)}+\frac{3}{2} \rho+\varepsilon} \quad \textrm{and} \quad M N \gg Q^{197\eta_1+9\eta_2} X^{\frac{1}{2}+\frac{\omega}{2(2+\omega)}+\varepsilon} = X^{\frac{1}{2}+\frac{\omega+2(197\eta_1+9\eta_2)}{2(2+\omega)}+\varepsilon}.
\end{gathered}
$$
These inequalities are satisfied for all sufficiently small $\varepsilon>0$, under the assumptions
\begin{align} \label{Eq: Conditions on Lambda and Eta II}
\delta<\frac{1}{12}-\frac{\omega+197\eta_1+9\eta_2}{2(2+\omega)}, \quad \lambda<\frac{1}{6}-\frac{\omega+4\eta_1+4\eta_2}{2(2+\omega)}, \quad \textrm{and} \quad \rho<\frac{1}{6}.
\end{align}
We therefore conclude the following.
\begin{lem}\label{Dispersion II}
Under the notation and hypotheses of Corollary \ref{Cor: Decomposition}, and assuming \eqref{Eq: Conditions on Lambda and Eta II}, we have
$$
T_2 \ll_{\varepsilon} Q^{1-2\eta_1-\eta_2-\varepsilon} \sqrt{X}.
$$
The implied constant may also depend at most on $\lambda, \delta, \rho$ and $\omega$.
\end{lem}
\subsection{Type II sums.} For the type II case \eqref{T3 Def}, we use the dispersion method of Linnik \cite{Linnik}, closely following the frameworks developed by Fouvry \cite{Fouvry}, Bombieri--Friendlander--Iwaniec \cite{BFI1986}, Drappeau \cite{D2017}, Fouvry--Radziwiłł \cite{FouvryRadziwill2018} and Drappeau–Pratt–Radziwiłł \cite{DPR2023}. We want to establish the bound
$$
T_{\mathrm{II}}(M, N):=\underset{w}{\sum\nolimits^*} \widetilde{\Theta}_k\left(\frac{w}{W}\right) \sum_{m} \sum_n \alpha_m \beta_n \mathfrak{u}_R(m n, b w) \ll_{\varepsilon}  Q^{1-2\eta_1-\eta_2-\varepsilon} \sqrt{X},
$$
where $\alpha_m$ is supported around $m \sim M, \beta_n$ is supported around $n \sim N$, $M N \asymp X$, and $X^\sigma \ll N \ll X^{\frac{1}{3}-\delta}$. We have the bounds $\left|\alpha_m\right| \leq \tau(m)^{\mathcal{O}(1)}$, and similarly for $\beta$. 
\par
Following \cite[Section 4.4]{DPR2023}, we interchange summation and apply the triangle inequality to get
$$
\bigg|T_{\mathrm{II}}(M, N)\bigg| \leq \sum_m\bigg|\sum_w \sum_n\bigg| .
$$
Applying the Cauchy--Schwarz inequality, we arrive at
$$
T_{\mathrm{II}}(M, N)^2 \ll M(\log M)^{\mathcal{O}(1)} \mathcal{D},
$$
where
$$
\mathcal{D}=\sum_m f\left(\frac{m}{M}\right)\bigg| \underset{w}{\sum\nolimits^*} \widetilde{\Theta}_k\left(\frac{w}{W}\right) \sum_{mn \equiv 1 \bmod bw} \beta_n-\frac{1}{\varphi(b w)} \sum_{\substack{\chi \bmod b w \\ \operatorname{cond}(\chi) \leq R}} \underset{w}{\sum\nolimits^*} \sum_{\substack{n \\(m n, b w)=1}} \widetilde{\Theta}_k\left(\frac{w}{W}\right) \beta_n \chi(m n)\bigg|^2 .
$$
Here $f$ is some fixed, nonnegative smooth function supported in $[\frac{1}{2},3]$ and majorizing $\mathbbm{1}_{[1,2]}$. It suffices to show that
$$
\mathcal{D} \ll_{\varepsilon} N Q^{2-4\eta_1-2\eta_2-\varepsilon} .
$$
We open the square and arrive at

$$
\mathcal{D}=\mathcal{D}_1-2 \operatorname{Re} \mathcal{D}_2+\mathcal{D}_3,
$$
say. We will now study each sum $\mathcal{D}_i$.
\subsubsection{Evaluation of $\mathcal{D}_3$.} By definition we have
$$
\mathcal{D}_3:=\sum_m f\left(\frac{m}{M}\right) \underset{\substack{w_1, w_2, n_1, n_2 \\\left(m n_1, b w_1\right)=1 \\\left(m n_2, b w_2\right)=1}} {\sum\nolimits^*} \sum_{\substack{\chi_1, \chi_2 \\ \chi_j \bmod b w_j \\ \operatorname{cond}\left(\chi_j\right) \leq R}} \widetilde{\Theta}_k\left(\frac{w_1}{W}\right) \widetilde{\Theta}_k \left(\frac{w_2}{W}\right) \beta_{n_1} \overline{\beta}_{n_2} \frac{\chi_1\left(m n_1\right) \overline{\chi_2}\left(m n_2\right)}{\varphi\left(b w_1\right) \varphi\left(b w_2\right)}.
$$
We follow the computations in \cite[pp. 711-712]{D2017} with the choice
\begin{align}\label{Gamma Sequence}
\gamma(q)=\mathbbm{1}(b \mid q) \widetilde{\Theta}_k(q /(b W)) \mathbbm{1}(qb^{-1} \equiv w_0 \bmod D) 
\end{align}
and the modification that $\operatorname{cond}\left(\chi_1 \overline{\chi_2}\right) \leq R^2$. Writing $H=Q^{\varepsilon} b\left[w_1, w_2\right] M^{-1}$, we get
$$
\begin{aligned}
\mathcal{D}_3 & =\mathcal{M}_3+\mathcal{O}_{\varepsilon}\bigg(Q^{\varepsilon} \underset{\substack{w_1, w_2 \asymp W \\
n_1, n_2 \asymp N}}{\sum\nolimits^*} \frac{1}{\varphi\left(b w_1\right) \varphi\left(b w_2\right)} \sum_{\substack{\chi_1, \chi_2 \\
\text { cond }\left(\chi_j\right) \leq R}} \frac{M}{b\left[w_1, w_2\right]} \sum_{0<|h| \leq H} R \sum_{d \mid\left(h, b\left[w_1, w_2\right]\right)} d\bigg) \\
& =\mathcal{M}_3+\mathcal{O}_{\varepsilon}\left(Q^{\varepsilon} N^2 R^5 Q^{-2\eta_2}\right),
\end{aligned}
$$
where the main term provided by \cite[p. 712]{D2017} is 
$$
\mathcal{M}_3:=M \widehat{f}(0) \underset{\substack{w_1, w_2, n_1, n_2 \\\left(m n_1, b w_1\right)=1 \\\left(m n_2, b w_2\right)=1}} {\sum\nolimits^*} \sum_{\substack{\chi \text { primitive } \\ \text { cond }(\chi) \leq R \\ \operatorname{cond}(\chi) \mid b\left(w_1, w_2\right)}} \widetilde{\Theta}_k\left(\frac{w_1}{W}\right) \widetilde{\Theta}_k\left(\frac{w_2}{W}\right) \beta_{n_1} \overline{\beta}_{n_2} \chi\left(n_1 \overline{n_2}\right) \frac{\varphi\left(b w_1 w_2\right)}{b w_1 w_2 \varphi\left(b w_1\right) \varphi\left(b w_2\right)}.
$$
The error term is acceptable provided
$N R^5 \ll_{\varepsilon} Q^{2-4\eta_1-\varepsilon}$. Since $N \ll X^{\frac{1}{3}}$, this is acceptable provided
\begin{align}\label{Rho condition I}
\rho<\frac{4-\omega-12\eta_1}{15(2+\omega)}.
\end{align}
\subsubsection{Evaluation of $\mathcal{D}_2$.} We have

$$
\mathcal{D}_2:=\underset{\substack{w_1, w_2, n_1, n_2 \\\left(m n_1, b w_1\right)=1 \\\left(m n_2, b w_2\right)=1}} {\sum\nolimits^*}  \sum_{\substack{\chi \bmod b w_2 \\ \operatorname{cond}(\chi) \leq R}} \widetilde{\Theta}_k \left(\frac{w_1}{W}\right) \widetilde{\Theta}_k\left(\frac{w_2}{W}\right) \overline{\beta_{n_1}} \beta_{n_2} \frac{\chi\left(n_2\right)}{\varphi\left(b w_2\right)} \sum_{\substack{m n_1 \equiv 1\left(b w_1\right) \\\left(m, w_2\right)=1}} \chi(m) f\left(\frac{m}{M}\right) .
$$
Again, we follow the computations in \cite[pp. 712-713]{D2017} with the choice
\eqref{Gamma Sequence} to obtain
\begin{align}\label{Dispersion Term 2}
\mathcal{D}_2=\mathcal{M}_3+\mathcal{O}_{\varepsilon}\bigg(Q^{\varepsilon} R^{\frac{3}{2}} N^2(bW)^{-1} \underset{\substack{w_1 \asymp W, \,  w_2 \asymp W}} {\sum\nolimits^* \sum\nolimits^*} (bw_2,bw_1^{\infty})\bigg).
\end{align}
To address the sum over $w_1$ and $w_2$, we write
\begin{align}\label{Bw sums}
\underset{\substack{w_1 \asymp W , \, w_2 \asymp W}} {\sum\nolimits^* \sum\nolimits^*} (bw_2,bw_1^{\infty}) &= \sum_{d \leq 3 bW} d \underset{\substack{W/2 \leq w_1, w_2 \leq 3W \\ (bw_2,bw_1^{\infty})=d }} {\sum\nolimits^* \sum\nolimits^*} 1 = \sum_{\substack{d \leq 3 bW \\ 5bW < 2[\operatorname{ker}(d),D]}} +\sum_{\substack{d \leq 3 bW \\ 5bW \geq 2[\operatorname{ker}(d),D]}},
\end{align}
For the second sum, using the change of variables $w_1'=bw_1/\operatorname{ker}(d)$ and $w_2' = bw_2/d$, we have
\begin{align}\label{AP Condition}
\sum_{\substack{d \leq 3 bW \\ 5bW \geq 2[\operatorname{ker}(d),D]}} &\ll  \sum_{\substack{d \leq 3 bW \\ 5bW \geq 2[\operatorname{ker}(d),D]}} d \sum_{\substack{ w_1' \asymp bW/\operatorname{ker}(d) \\ w_1' \equiv b w_0 \operatorname{ker}(d)^{-1} \bmod D/(\operatorname{ker}(d),D) }} \sum_{\substack{ w_2' \asymp bW/d \\ w_2' \equiv b w_0 d^{-1} \bmod D/(d,D) }} 1 \notag \\
&\ll \sum_{d \ll bW} d \left ( \frac{bW (\operatorname{ker}(d),D)}{\operatorname{ker}(d)D}\right) \left ( \frac{bW (d,D)}{dD}+1 \right) \ll_{\varepsilon} Q^{\varepsilon} \frac{(bW)^2}{D} \sum_{d \ll bW} \frac{(\operatorname{ker}(d),D)}{\operatorname{ker}(d)}.
\end{align}
Let $d=d_1d_2$, where $p \!\mid\! d_1$ only if $p \!\mid\! D$, and $(d_2,D)=1$. Using the bound $\sum_{n \ll X} \operatorname{ker}(n)^{-1} \ll_{\varepsilon} X^{\varepsilon}$, (see Bruijn \cite{Bruijn}), we arrive at
\begin{align}\label{Breaking d}
\sum_{d \ll bW} \frac{(\operatorname{ker}(d),D)}{\operatorname{ker}(d)} \ll_{\varepsilon} (bW)^{\varepsilon} \sum_{\substack{d_1 \ll bW \\ p \mid d_1 \implies p \mid D }} 1. 
\end{align}
Let $\mathcal{C}\!=\!\{ n \leq X : p\!\mid\! n \text{ only if } p \!\mid\! D \}$ and $\widehat{\mathcal{C}}\!=\!\{ n \leq X : p \!\mid\! n \text{ only if } p \in \{2,3,\dots, p_{\omega(D)} \} \}$, where $p_{\omega(D)}$ is the $\omega(D)$-th prime. There is a canonical injection from $\widehat{\mathcal{C}}$ to $ \mathcal{C}$. Therefore, if $\log X \asymp \log D$, then for any $\delta>0$, $\#\mathcal{C} \leq \# \widehat{\mathcal{C}} \leq \#\mathcal{S}(X, (\log X)^{1+\delta})$ where $\mathcal{S}(X,y)$ denotes the set of $y$-smooth numbers up to $X$. Using estimates for smooth numbers (see Granville \cite[Eq. 1.12]{Granville}) and allowing $\delta>0$ small depending on $\varepsilon$, \eqref{Breaking d} is $\ll_{\varepsilon} (bW)^{\varepsilon}$. Substituting this in \eqref{AP Condition}, the contribution from the second sum in \eqref{Bw sums} is $\ll_{\varepsilon} Q^{\varepsilon}(bW)^2D^{-1}$.  
\par
For the first sum, the condition on $d$ forces $w_1$ and $w_2$ to be equal and therefore, these are diagonal contributions which is also $\ll_{\varepsilon} Q^{\varepsilon}(bW)^2D^{-1}$. Hence the error term from \eqref{Dispersion Term 2} is acceptable if
\begin{align}\label{Rho condition II}
\rho<\frac{2\delta}{3}+\frac{2(1-\omega)-6(5\eta_1+2\eta_2)}{9(2+\omega)}.
\end{align}
\subsubsection{Evaluation of $\mathcal{D}_1$} We have
$$
\mathcal{D}_1:=\underset{\substack{w_1, w_2, n_1, n_2 \\\left(n_j, b w_j\right)=1 \\ n_1 \equiv n_2 \bmod b}} {\sum\nolimits^* } \widetilde{\Theta}_k\left(\frac{w_1}{W}\right) \widetilde{\Theta}_k\left(\frac{w_2}{W}\right) \beta_{n_1} \overline{\beta}_{n_2} \sum_{m n_j \equiv 1 \bmod b w_j } f\left(\frac{m}{M}\right).
$$
We separate the variables $w_1, w_2, n_1, n_2$ from each other, following \cite{FouvryRadziwill2018} and \cite{DPR2023}:
$$
\left\{\begin{aligned}
&d_0=\left(n_1, n_2\right),  \quad  n_1=d_0d_1d_2 v_1 \text { with } d_1 \mid d_0^{\infty}, d_2 \mid D^{\infty} \text { and }\left(d_0D, v_1\right)=1, \quad n_2=d_0 v_2, \\
&q_0=\left(w_1, w_2\right), \quad w_i=q_0 q_i \text { for } i \in\{1,2\} .
\end{aligned}\right.
$$
The coprimality conditions in $\mathcal{D}_1$ imply
$\left(d_0 d_1 d_2v_1, bq_0 q_1\right)=\left(d_0 v_2, bq_0 q_2\right)=1.$ Thus we have
$$
\begin{aligned}
\mathcal{D}_1= & \sum_{(d_0, b)=1} \sum_{d_1 \mid d_0^{\infty}} \sum_{\substack{d_2 \mid D^{\infty} \\ (d_2,b)=1 }} \sum_{\substack{\left(q_0, d_0\right)=1 \\ (q_0,D)=1}} \mathcal{D}_1\left(d_0, d_1, d_2, q_0\right) \\
\mathcal{D}_1(\cdots)= & \underset{\substack{q_1, q_2, v_1, v_2 \\
(d_1d_2 v_1, v_2)=\left(q_1, q_2\right)=1 \\
\left(q_1 q_2, d_0\right)=\left(v_1, d_0D\right)=1 \\
\left(v_1, q_1\right)=\left(v_2, q_2\right)=\left(v_1 v_2, b q_0\right)=1}}{\sum\nolimits^*} \widetilde{\Theta}_k\left(\frac{q_0 q_1}{W}\right) \widetilde{\Theta}_k\left(\frac{q_0 q_2}{W}\right) \beta_{d_0 d_1 d_2 v_1} \overline{\beta}_{d_0 v_2} \sum_{\substack{m d_0 d_1 d_2 v_1 \equiv 1 \bmod b q_0 q_1 \\
m d_0 v_2 \equiv 1 \bmod b q_0 q_2 }} f\left(\frac{m}{M}\right) .
\end{aligned}
$$
Using smooth partitions of unity, we can break the variables into dyadic ranges $d_0 \asymp D_0, d_1 \asymp D_1, d_2 \asymp D_2$ and $ q_0 \asymp Q_0$. The contribution from $d_0 \asymp D_0, d_1 \asymp D_1$ and $d_2 \asymp D_2$ is
\begin{align}\label{Bound 1}
& \ll_{\varepsilon} Q^{\varepsilon} M \sum_{d_0 \asymp D_0} \sum_{\substack{d_1 \mid d_0^{\infty} \\
d_1 \asymp D_1}} \sum_{\substack{d_2 \mid D^{\infty} \\ d_2 \asymp D_2}} \sum_{v_1 \asymp N / d_0 d_1 d_2} \sum_{v_2 \asymp N / d_0}\left|\beta_{d_0 d_1 d_2 v_1}\right|\left|\beta_{d_0 v_2}\right| \notag \\
&\ll_{\varepsilon} Q^{\varepsilon} M N^2 \sum_{d_0 \asymp D_0} \frac{1}{d^2} \sum_{d_1 \mid d_0^{\infty}}  \frac{\tau\left(d_1\right)^{O(1)}}{d_1}\left(\frac{d_1}{D_1}\right)^{1-\varepsilon^2} \sum_{\substack{d_2 \mid D^{\infty}}} \frac{\tau\left(d_2\right)^{O(1)}}{d_2}\left(\frac{d_2}{D_2}\right)^{1-\varepsilon^2} \notag \\ 
&\ll_{\varepsilon} Q^{\varepsilon} M N^2 (D_1D_2)^{-1+\varepsilon^2} D_0^{-1}
\end{align}
where the sum over $q_0$ and $q_1$ was bounded by $\mathcal{O}\left(\tau_3\left(\left|m d_0 d_1 d_2 v_1-1\right|\right)\right)=\mathcal{O}_{\varepsilon}\left(Q^{\varepsilon}\right)$, likewise for the sum over $q_2$ (since $m d_0 v_2 \neq 1$ and $m d_0 d_1 d_2 v_1 \neq 1$ ). This bound is acceptable provided
$
D_0 D_1 D_2 \gg X Q^{-2+4\eta_1+2\eta_2 +\varepsilon},
$
so we assume $D_0 D_1 D_2 \ll X Q^{-2+4\eta_1+2\eta_2 +\varepsilon}$.
The contribution from $q_0 \asymp Q_0$ is
\begin{align}\label{Bound 2}
& \ll_{\varepsilon} Q^{\varepsilon} \sum_{q_0 \asymp Q_0} \sum_{q_1 \asymp Q / q_0} \sum_{\substack{n_1 \equiv n_2 \bmod q_0 \\
n_j \asymp N}} \sum_{\substack{m \asymp M \\
m \equiv \overline{n_1} \bmod q_0 q_1 }} 1 \notag \\
& \ll_{\varepsilon} Q^{\varepsilon} M \sum_{q_0 \asymp Q_0} \sum_{q_1 \asymp Q / q_0} \frac{1}{q_0 q_1} \sum_{\substack{n_1 \equiv n_2 \bmod q_0 \\
n_j \asymp N}} 1 \ll_{\varepsilon} Q^{\varepsilon}\left(M N^2 Q_0^{-1}+MN\right),
\end{align}
where in the first line the sum over $q_2$ was again bounded by $\tau\left(\left|m d_0 v_2-1\right|\right)$. This is acceptable if
$$
N \gg X Q^{-2+4\eta_1+2\eta_2 +\varepsilon} \quad \text { and } \quad Q_0 \gg X Q^{-2+4\eta_1+2\eta_2 +\varepsilon},
$$
so we assume $Q_0 \ll X Q^{-2+4\eta_1+2\eta_2+\varepsilon}$.
We now use Poisson summation, as in \cite[714-716]{D2017}. Let
$$
\widetilde{q}=b q_0 q_1 q_2 \quad \text { and } \quad \mu \equiv\left\{\begin{array}{cc}
\overline{d_0 d_1 d_2 v_1} & \bmod \hspace{0.1cm} b q_0 q_1, \\
\overline{d_0 v_2} & \bmod \hspace{0.1cm} b q_0 q_2 .
\end{array}\right.
$$
With $H=\widetilde{q}^{1+\varepsilon} M^{-1} \ll bW^{2+\varepsilon} /\left(q_0 M\right)$, we get for any fixed $A>0$,
\begin{align}\label{Poisson Step}
\sum_{m \equiv \mu \bmod \widetilde{q}} f\left(\frac{m}{M}\right)=\frac{M}{\widetilde{q}} \sum_{|h| \leq H} \widehat{f}\left(\frac{h M}{\widetilde{q}}\right) \mathrm{e}\left(\frac{\mu h}{\widetilde{q}}\right)+\mathcal{O}_{\varepsilon,A}(Q^{-A}) .
\end{align}
The zero frequency contributes the main term. We sum over $d_0, d_1, d_2$ and $q_0$ (and reintegrate the values $D_0 D_1D_2$ and $Q_0$ larger than $X Q^{-2+4\eta_1+2\eta_2+\varepsilon}$ which were discarded earlier) to obtain
$$
\mathcal{M}_1:=\frac{M}{b} \widehat{f}(0) \underset{\substack{w_1, w_2, n_1, n_2 \\ (n_j, b w_j) =1 \\ n_1 \equiv n_2 \bmod b\left(w_1, w_2\right) }}{\sum\nolimits^*}  \widetilde{\Theta}_k\left(\frac{w_1}{W}\right) \widetilde{\Theta}_k\left(\frac{w_2}{W}\right) \beta_{n_1} \overline{\beta}_{n_2} \frac{1}{\left[w_1, w_2\right]} .
$$
The error term from the reintegration process can be bounded similarly how we obtained \eqref{Bound 1} and \eqref{Bound 2}. Therefore in the ranges $D_0D_1D_2 \gg X Q^{-2+4\eta_1+2\eta_2+\varepsilon}$ and $Q_0 \gg X Q^{-2+4\eta_1+2\eta_2+\varepsilon}$, these errors are again acceptable. Choosing $A>0$ sufficiently large and fixed, the error induced in $\mathcal{D}_1$ from the error in \eqref{Poisson Step} is also acceptable. 
\par
Following \cite{DPR2023}, we now solve the congruence conditions on $\mu$ by writing
$$
d_1 d_2 v_1-v_2=b q_0 t, \quad \mu d_0 d_1d_2 v_1=1+b q_0 q_1 \ell, \quad \mu d_0 v_2=1+b q_0 q_2 m
$$
with $t, \ell, m \in \mathbb{Z}$. We deduce that
$$
\mu d_0t=q_1 \ell-q_2 m, \quad t=q_1 v_2 \ell-q_2 d_1 d_2 v_1 m .
$$
A little calculation shows that
$$
\begin{aligned}
\frac{\mu}{\widetilde{q}}=\frac{\mu}{b q_0 q_1 q_2}\equiv \frac{1}{d_0 d_1 d_2 v_1 b q_0 q_1 q_2}+\frac{d_1 d_2 v_1-v_2}{b q_0} \frac{\overline{q_1 v_2 d_0 d_1 d_2}}{v_1 q_2}-\frac{\overline{b q_0 q_1 v_1 q_2}}{d_0 d_1 d_2} \bmod 1.
\end{aligned}
$$
Then we have
\begin{align}\label{Reciprocity}
\mathrm{e}\left(\frac{h \mu}{\widetilde{q}}\right)=\mathrm{e}\left(h \frac{d_1 d_2 v_1-v_2}{b q_0} \frac{\overline{q_1 v_2 d_0 d_1 d_2}}{v_1 q_2}-\frac{h \overline{b q_0 q_1 v_1 q_2}}{d_0 d_1 d_2}\right)+O\left(\frac{H q_0}{bN W^2}\right).
\end{align}
The error term here is $\ll_{\varepsilon} Q^{\varepsilon} X^{-1}$, which contributes to $\mathcal{D}_1\left(d_0, d_1, d_2,  q_0\right)$ at most
$$
\frac{Q^{2+\varepsilon} N}{X q_0^2 d_0 d_1 d_2}\left(1+\frac{N}{d_0bq_0}\right)
$$
and then after summing over $d_0, d_1, d_2$ and $q_0$, this contributes to $\mathcal{D}_1$ at most $\mathcal{O}_{\varepsilon}\left(Q^{2+\varepsilon} N^2 X^{-1}\right)$. This error is acceptable if $N \ll Q^{2-4\eta_1-2\eta_2-\varepsilon}$. We now insert the first term of \eqref{Reciprocity} in \eqref{Poisson Step}, and write $\widehat{f}$ in terms of the Fourier integral. Therefore after a change of variables, the nonzero frequencies contribute
$$
\begin{aligned}
\mathcal{E}\left(d_0, d_1, d_2, q_0\right) = &\frac{M q_0}{b W^2} \int  \underset{\substack{q_1, q_2, v_1, v_2 \\
\left(d_1 d_2 v_1, v_2\right)=\left(q_1, q_2\right)=1 \\
\left(q_1 q_2, d_0\right)=\left(v_1, d_0 D\right)=1 \\
\left(v_1, q_1\right)=\left(v_2, q_2\right)=\left(v_1 v_2, b q_0\right)=1 \\ d_1d_2 v_1 \equiv v_2 \bmod bq_0}}{\sum\nolimits^*}  \sum_{0<|h| \leq H} \widetilde{\Theta}_k\left(\frac{q_0 q_1}{W}\right) \widetilde{\Theta}_k\left(\frac{q_0 q_2}{W}\right) \beta_{d_0 d_1 d_2 v_1} \overline{\beta}_{d_0 v_2} \\
& \times f\left(t \frac{q_0^2 q_1 q_2}{W^2}\right) \mathrm{e}\left(h \frac{d_1 d_2 v_1-v_2}{b q_0} \frac{\overline{q_1 v_2 d_0 d_1 d_2}}{v_1 q_2}-\frac{h \overline{b q_0 q_1 v_1 q_2}}{d_0 d_1 d_2}\right) \mathrm{e}\left(\frac{-h t M q_0}{b W^2}\right) \mathrm{d} t.
\end{aligned}
$$
If we let
$$
\begin{aligned}
\mathcal{E}_1 :=\sum_{\substack{Q_0, D_0 D_1 D_2 \ll X Q^{-2+4\eta_1+2\eta_2+\varepsilon} \\ Q, D_0, D_1, D_2 \text { dyadic }}} \sum_{\substack{d_0 \asymp D_0, d_1 \asymp D_1 \\ d_2 \asymp D_2, q_0 \asymp Q_0 \\ d_1 \mid d_0^{\infty}, d_2 \mid D^{\infty}}} \mathcal{E}\left( d_0, d_1, d_2, q_0\right), 
\end{aligned}
$$
then so far, under the conditions 
\begin{align}\label{Preliminary N Condition 1}
X Q^{-2+4\eta_1+2\eta_2+\varepsilon} \ll  N \ll  Q^{2-4\eta_1-2\eta_2-\varepsilon},
\end{align}
we have the estimate 
$$
\mathcal{D}_1=\mathcal{M}_1+\mathcal{E}_1+\mathcal{O}_{\varepsilon}\left(N Q^{2-4\eta_1-2\eta_2-\varepsilon}\right).
$$
We now address $\mathcal{E}$. Breaking the sum over $q_1$ and $q_2$ mod $d_0 d_1 d_2$, we rewrite $\mathcal{E}\left(d_0, d_1, d_2, q_0\right)$ as
\begin{align}\label{Rewriting}
\frac{M q_0}{b W^2} \underset{\substack{\lambda_1 \bmod d_0 d_1 d_2 \\ \lambda_2 \bmod d_0 d_1 d_2}}{\sum\nolimits^*} &\int \underset{\substack{v_1, v_2 \\
\left(d_1 d_2 v_1, v_2\right) = \left(v_1, d_0 D \right)=1 \\
\left(v_1 v_2 , b q_0\right)=1 \\ d_1d_2 v_1 \equiv v_2 \bmod bq_0}}{\sum}  \sum_{0<|h| \leq H} \beta_{d_0 d_1 d_2 v_1} \overline{\beta}_{d_0 v_2} \mathrm{e}\left(-\frac{h t M q_0}{b W^2} - \frac{h \overline{b q_0 \lambda_1 v_1 \lambda_2}}{d_0 d_1 d_2}\right)   \notag\\
& \hspace{-1.75cm} \times \underset{\substack{q_1 \equiv \widetilde{\lambda}_1 \bmod [d_0 d_1 d_2,D] \\ q_2 \equiv \widetilde{\lambda}_2 \bmod [d_0 d_1 d_2,D] \\
\left(q_1, q_2 \right)=\left(v_1, q_1\right)=\left(v_2, q_2 \right)=1}}{\sum }  \widetilde{\Theta}_k\left(\frac{q_0 q_1}{W}\right) \widetilde{\Theta}_k\left(\frac{q_0 q_2 }{W}\right) f\left(t \frac{q_0^2 q_1 q_2 }{W^2}\right) \mathrm{e}\left(h \frac{d_1 d_2 v_1-v_2}{b q_0} \frac{\overline{q_1 v_2 d_0 d_1 d_2}}{v_1 q_2}\right) \mathrm{d} t.
\end{align}
We mention that the notation $\mathop{\sum\nolimits^*}_{\lambda_i}$ indicates that the sum is over $\lambda_1,\lambda_2$ coprime to $d_0d_1d_2$. We plan to apply Lemma \ref{New Drappeau's Theorem with Theta} with the following dictionary:
\begin{align*}
&c=q_2, \quad d=q_1, \quad n =h \frac{d_1d_2v_1-v_2}{bq_0}, \quad r = d_0d_1 d_2v_2, \quad s=v_1,\\
&\boldsymbol{C} \sim W/q_0,\quad  \boldsymbol{D} \sim W/q_0, \quad \boldsymbol{N} \sim \frac{NH}{d_0bq_0}, \quad \boldsymbol{R} \sim Nd_1 d_2 , \quad \boldsymbol{S} \sim N/d_0d_1 d_2,\\
&g(c,d,n,r,s) := g(q_1,q_2) = \widetilde{\Theta}_k\left(\frac{q_0q_1}{W}\right) \widetilde{\Theta}_k\left(\frac{q_0q_2}{W}\right) f\left(t \frac{q_0^2 q_1 q_2 }{W^2}\right) ,\\
&b_{n,r,s} = \mathbbm{1}_{\substack{r \equiv 0 \bmod d_0d_1 d_2 \\ d_1 d_2 s \equiv r/(d_0d_1 d_2) \bmod bq_0 \\ n \equiv 0 \bmod (d_1 d_2 s- r/(d_0d_1 d_2)) (bq_0)^{-1}  } }\times \mathbbm{1}_{ \substack{ (d_1 d_2s,r/(d_0d_1 d_2))=1 \\ (s,d_0D)=1 \\ (rs/(d_0d_1d_2), bq_0)=1}}  \notag \\
&\quad \quad \quad  \quad \times \beta_{d_0 d_1 d_2 s} \overline{\beta}_{r/d_1d_2} \mathrm{e}\left( \left (-\frac{t M q_0}{b W^2} - \frac{\overline{b q_0 s\lambda_1  \lambda_2}}{d_0 d_1 d_2}\right) \frac{n bq_0}{d_1d_2s-r/(d_0d_1 d_2)}  \right),\notag   \\
&\varepsilon_1 = 2\eta_1+\varepsilon, \quad \varepsilon_2 =2\eta_1+ \varepsilon, \quad \varepsilon_3,\varepsilon_4,\varepsilon_5 = 0.
\end{align*} 
With the above choices, we rewrite \eqref{Rewriting} as
\begin{align}\label{Our Lemma setup}
\frac{M q_0}{b W^2} \underset{\substack{\lambda_1 \bmod d_0 d_1 d_2 \\ \lambda_2 \bmod d_0 d_1 d_2}}{\sum\nolimits^*} &\int \underset{\substack{c,d,n,r,s \\c \equiv \widetilde{\lambda}_2 \bmod [d_0 d_1 d_2,D] \\ d \equiv \widetilde{\lambda}_1 \bmod [d_0 d_1 d_2,D] \\
(rd [d_0d_1d_2,D],sc)=1}}{\sum}  b_{n,r,s} g(c,d,n,r,s) \mathrm{e} \left ( n \frac{\overline{rd}}{sc} \right)   \mathrm{d} t.
\end{align}
Before estimating our bounds, we make some observations to simplify our calculations. Let $X=Q^2 Y$. Then $Y=Q^{\omega}$. We note the estimates
$$
\boldsymbol{R} \boldsymbol{S} \asymp N^2 D_0^{-1}, \quad \boldsymbol{N} \ll \frac{Q^{\varepsilon} N^2}{YD_0Q_0^2}\ll Q^{\varepsilon} \boldsymbol{R} \boldsymbol{S}, \quad \boldsymbol{C} \ll Q^{\varepsilon} \boldsymbol{R} \boldsymbol{D} \quad \textrm{and} \quad \boldsymbol{CDNRS} \ll \frac{Q^{2+\varepsilon} N^4}{D_0^2Q_0^4} .
$$
Now applying Lemma \ref{New Drappeau's Theorem with Theta}, the quantity inside the integral in \eqref{Our Lemma setup} is
\begin{align}\label{First Bound}
\ll_{\varepsilon} Q^{\varepsilon} \left ( \frac{Q^{2} N^4}{D_0^2Q_0^4} \right)^{48\eta_1} \left(D_0 D_1D_2D\right)^{\frac{3}{2}} K(\boldsymbol{C}, \boldsymbol{D}, \boldsymbol{N}, \boldsymbol{R}, \boldsymbol{S}) \left\|b_{\boldsymbol{N}, \boldsymbol{R}, \boldsymbol{S}}\right\|_2,
\end{align}
where
\begin{align}\label{K bound}
Q^{-\varepsilon} K^2 &\ll_{\varepsilon} Q^2 N^4 D_0^{-1} (DD_1D_2) Q_0^{-2}+Q^{2+4 \theta} N^{4-6 \theta} D_0^{-2+2 \theta} D_1^{-2 \theta} D_2^{-2 \theta} Q_0^{-2-4 \theta}\left(1+\frac{D_0D}{N}\right)^{1-4 \theta} \notag \\
&\quad+Q^2 N^3 Y^{-1} D_0^{-1} D_1D_2 Q_0^{-4}.
\end{align}
To bound the term $\left\|b_{\boldsymbol{N}, \boldsymbol{R}, \boldsymbol{S}}\right\|_2$, we assume that
\begin{align}\label{Preliminary N Condition 2}
X Q^{-2+4\eta_1+2\eta_2+\varepsilon}=o(N)
\end{align}
so that $D_0=o(N)$. Then the case $d_1 d_2 v_1=v_2$ (which only occurs if both expressions are equal to 1) never occurs in $b_{n,r,s}$. Thus we have
\begin{align}\label{B sequence norm Bound}
\left\|b_{\boldsymbol{N}, \boldsymbol{R}, \boldsymbol{S}}\right\|_2^2 & \leq \sum_{\substack{v_1, v_2, h \\
d_1 d_2 v_1 \equiv v_2 \bmod bq_0 \\
0<|h|<H}}\left|\beta_{d d_1 d_2 v_1} \beta_{d v_2}\right|^2 \ll_{\varepsilon} \frac{bQ^{2+\varepsilon}}{Q_0 M} \frac{N}{D_0 D_1 D_2}\left(\frac{N}{bD_0 Q_0}+1\right) \notag \\
& \ll_{\varepsilon} Q^{\varepsilon} b \left(N^3 Y^{-1}D_0^{-2} (D_1D_2)^{-1} Q_0^{-2}+N^2 Y^{-1} (D_0D_1D_2)^{-1}  Q_0^{-1}\right).
\end{align}
Combining \eqref{First Bound}, \eqref{K bound} and \eqref{B sequence norm Bound}, we deduce that the quantity inside the integral in \eqref{Our Lemma setup} is
$$
\ll_{\varepsilon} Q^{\varepsilon} b^{\frac{1}{2}} \left ( \frac{Q^{2} N^4}{D_0^2Q_0^4} \right)^{48\eta_1} \sum_{k=1}^6 Q^{\delta_{1,k}} N^{\delta_{2,k}} Y^{\delta_{3,k}} D^{\delta_{4,k}} D_0^{\delta_{5,k}} D_1^{\delta_{6,k}} D_2^{\delta_{7,k}} Q_0^{\delta_{8,k}}
$$
where $\delta=\left(\delta_{j, k}\right)_{1 \leq j \leq 8, 1 \leq k \leq 6}$ is given by
\[
\delta =
\left(
\begin{array}{rrrrrr}
1 & 1 & 2\theta + 1 & 2\theta + 1 & 1 & 1 \\[3pt]
\tfrac{7}{2} & 3 & \tfrac{7}{2} - 3\theta & 3 - 3\theta & 3 & \tfrac{5}{2} \\[3pt]
-\tfrac{1}{2} & -\tfrac{1}{2} & -\tfrac{1}{2} & -\tfrac{1}{2} & -1 & -1 \\[3pt]
2 & 2 & \tfrac{3}{2} & \tfrac{3}{2} & \tfrac{3}{2} & \tfrac{3}{2} \\[3pt]
0 & \tfrac{1}{2} & \theta - \tfrac{1}{2} & \theta & 0 & \tfrac{1}{2} \\[3pt]
\tfrac{3}{2} & \tfrac{3}{2} & 1 - \theta & 1 - \theta & \tfrac{3}{2} & \tfrac{3}{2} \\[3pt]
\tfrac{3}{2} & \tfrac{3}{2} & 1 - \theta & 1 - \theta & \tfrac{3}{2} & \tfrac{3}{2} \\[3pt]
-2 & -\tfrac{3}{2} & -2\theta - 2 & -2\theta - \tfrac{3}{2} & -3 & -\tfrac{5}{2}
\end{array}
\right).
\]
Note that the integral over $t$ in \eqref{Our Lemma setup} is supported around $\asymp_f 1$. Therefore integrating over $t$, summing over $\lambda_1, \lambda_2$, and multiplying by $M q_0\left(b W^2\right)^{-1}$ which is $\ll_{\varepsilon} Q^{4\eta_1+4\eta_2+\varepsilon} N^{-1} Y Q_0$, we get
\begin{align*}
\mathcal{E}\left(d_0, d_1, d_2, q_0\right) &\ll_{\varepsilon} Q^{4\eta_1+4\eta_2+\varepsilon} \left ( \frac{Q^{2} N^4}{D_0^2Q_0^4} \right)^{48\eta_1} \notag \\
&\quad \times \sum_{k=1}^6 Q^{\delta_{1,k}} N^{\delta_{2,k}-1} Y^{\delta_{3,k}+1} D^{\delta_{4,k}+2} D_0^{\delta_{5,k}+2} D_1^{\delta_{6,k}+2} D_2^{\delta_{7,k}+2} Q_0^{\delta_{8,k}+1}.
\end{align*}
We sum over $d_0, d_1, d_2$ and $q_0$ in dyadic intervals of lengths $D_0, D_1, D_2$ and $Q_0$, obtaining
\begin{align*}
\sum_{\substack{d_0 \asymp D_0, d_1 \asymp D_1 \\ d_2 \asymp D_2, q_0 \asymp Q_0 \\ d_1 \mid d_0^{\infty}, d_2 \mid D^{\infty}}} \mathcal{E}\left( d_0, d_1, d_2, q_0\right) &\ll_{\varepsilon} Q^{4\eta_1+4\eta_2+\varepsilon} \left ( \frac{Q^{2} N^4}{D_0^2Q_0^4} \right)^{48\eta_1}  \notag \\
&\quad \times \sum_{k=1}^6 Q^{\delta_{1,k}} N^{\delta_{2,k}-1} Y^{\delta_{3,k}+1} D^{\delta_{4,k}+2} D_0^{\delta_{5,k}+3} D_1^{\delta_{6,k}+2} D_2^{\delta_{7,k}+2} Q_0^{\delta_{8,k}+2}.
\end{align*}
Now we sum dyadically over $Q_0, D_0,D_1$ and $D_2$, subject to $Q_0+D_0D_1 D_2 \ll Y Q^{4\eta_1+2\eta_2+\varepsilon}$. We get
\begin{align*}
\mathcal{E}_1 &\ll_{\varepsilon} Q^{100\eta_1+4\eta_2+\varepsilon} N^{192 \eta_1} \notag \\
&\quad \times \sum_{k=1}^6 Q^{\delta_{1, k}} N^{\delta_{2,k}-1} D^{\delta_{4,k}+2} Y^{\delta_{3,k}+1} (Y Q^{4\eta_1+2\eta_2})^{\max \left(0, \delta_{8,k}+2-192\eta_1\right)+\max \left(0, \delta_{5,k}+3-96\eta_1, \delta_{6,k}+2, \delta_{7,k}+2\right) }.
\end{align*}
Note that the terms for $k=5,6$ are majorized by the term $k=2$. Furthermore, to make the following computations slightly easier, we use the bound $N \ll X^{\frac{1}{3}} \ll Q^{\frac{3}{4}}$. It follows that
$$
\mathcal{E}_1 \ll_{\varepsilon} Q^{244\eta_1+4\eta_2+\varepsilon} D^4 Q^{4(4\eta_1+2\eta_2)} \sum_{k=1}^4 Q^{\widetilde{\delta}_{1, k}} N^{\widetilde{\delta}_{2,k}} Y^{\widetilde{\delta}_{3,k}} \ll_{\varepsilon} Q^{260\eta_1+16\eta_2+\varepsilon} \sum_{k=1}^4 Q^{\widetilde{\delta}_{1, k}} N^{\widetilde{\delta}_{2,k}} Y^{\widetilde{\delta}_{3,k}},
$$
where $\widetilde{\delta}=(\widetilde{\delta}_{j, k})_{1 \leq j \leq 3, 1 \leq k \leq 4}$ is given by
\[
\widetilde{\delta} =
\left(
\begin{array}{rrrrrr}
1 & 1 & 1+2\theta & 1+2\theta \\[3pt]
\tfrac{5}{2} & 2 & \tfrac{5}{2} - 3\theta & 2 - 3\theta \\[3pt]
4 & \tfrac{9}{2} & \tfrac{7}{2} -\theta & 4-\theta \\
\end{array}
\right).
\]
We conclude that
$$
\mathcal{D}_1=\mathcal{M}_1+\mathcal{O}_{\varepsilon} \left(NQ^{2-4\eta_1-2\eta_2-\varepsilon}\right)
$$
on the condition that 
\begin{align}\label{N Conditions 0}
N \ll Q^{-\varepsilon} \, \min \Big(& 
Q^{\frac{2}{3}-166\eta_1-12\eta_2} Y^{-\frac{8}{3}}, 
Q^{1-264\eta_1-18\eta_2} Y^{-\frac{9}{2}}, \notag\\
& Q^{\frac{2}{3}-\frac{166\eta_1+12\eta_2}{1-2\theta}} Y^{-\frac{7-2 \theta}{3(1-2 \theta)}}, 
Q^{\frac{1-2 \theta}{1-3 \theta} -\frac{264\eta_1+18\eta_2}{1-3\theta}} Y^{-\frac{4-\theta}{1-3 \theta}} 
\Big).
\end{align}
Upon using $\theta \leq \frac{7}{64}$, these conditions are implied by
\begin{align}\label{N Conditions I}
N \ll X^{-\varepsilon} \min \left(X^{\frac{2-9\omega-2(264\eta_1+18\eta_2)}{2(2+\omega)}}, X^{\frac{50-217\omega-96(166\eta_1+12\eta_2)}{75(2+\omega)}}, X^{\frac{50-249\omega-64(264\eta_1+18\eta_2)}{43(2+\omega)}}\right),
\end{align}
along with \eqref{Preliminary N Condition 1} and \eqref{Preliminary N Condition 2}, which implies that
\begin{align}\label{N Conditions II}
X^{\frac{\omega+4\eta_1+2\eta_2}{2+\omega}+\varepsilon} = o(N) \quad \textrm{and} \quad  N \ll  X^{\frac{2-4\eta_1-2\eta_2}{2+\omega}-\varepsilon}.
\end{align}
For clarity, we mention that we have allowed 
$Y$ to attain its maximal value, and consequently 
$\omega$ to be as large as possible, in order to deduce \eqref{N Conditions I} from \eqref{N Conditions 0}.
 \subsubsection{Final Conclusion.} The main terms $\mathcal{M}_1$ and $\mathcal{M}_3$ combine to give us
$$
\begin{aligned}
\mathcal{M}_1-\mathcal{M}_3 & =M \widehat{f}(0) \underset{\substack{w_1, w_2}}{\sum\nolimits^*}  \widetilde{\Theta}_k\left(\frac{w_1}{W}\right) \widetilde{\Theta}_k\left(\frac{w_2}{W}\right) \frac{1}{b\left[w_1, w_2\right] \varphi\left(b\left(w_1, w_2\right)\right)} \notag \\
&\quad \times \sum_{\substack{\chi \text { primitive } \\
\text { cond }(\chi)>R \\
\text { cond }(\chi) \mid b\left(w_1, w_2\right)}} \sum_{\substack{n_1, n_2 \\
\left(n_j, b w_j\right)=1}} \beta_{n_1} \overline{\beta}_{n_2}
\chi\left(n_1\right) \overline{\chi}\left(n_2\right).
\end{aligned}
$$
We use \cite[p. 717]{D2017} with the choice \eqref{Gamma Sequence} to obtain
$$
\left|\mathcal{M}_3-\mathcal{M}_1\right| \ll_{\varepsilon} Q^{\varepsilon} M\left(N+N^2 R^{-2}\right) \ll_{\varepsilon} Q^{\varepsilon}\left(X+N X R^{-2}\right).
$$
This is acceptable provided
\begin{align}\label{N and R conditions}
N \gg X^{\frac{\omega+4\eta_1+2\eta_2}{2+\omega}+\varepsilon} \quad \text { and } \quad R \gg X^{\frac{\omega+4\eta_1+2\eta_2}{2(2+\omega)}+\varepsilon}.
\end{align}
After a little calculation, we see that the conditions \eqref{Rho condition I}, \eqref{Rho condition II}, \eqref{N Conditions I}--\eqref{N and R conditions} are all satisfied if 
\begin{align}
\omega+75\eta_1+6 \eta_2 &< \frac{1}{8}, \label{Prefinal I} \\
\frac{\omega+4\eta_1+2\eta_2}{2+\omega}&<\sigma<\frac{1}{3}-\delta<\frac{1}{3}-\frac{242 \omega+96(166\eta_1+12\eta_2)}{75(2+\omega)}, \label{Prefinal II} \\
\textrm{and} \quad \frac{\omega+4\eta_1+2\eta_2}{2(2+\omega)}&<\rho<\frac{1}{9}-\frac{\omega+(10\eta_1+4\eta_2)}{3(2+\omega)} \label{Prefinal III}.
\end{align}
We therefore conclude the following.
\begin{lem}\label{Dispersion III}
Under the notation and hypotheses of Corollary \ref{Cor: Decomposition}, and assuming \eqref{Prefinal I}--\eqref{Prefinal III}, we have
$$
T_{\mathrm{II}} \ll_{\varepsilon} Q^{1-2\eta_1-\eta_2-\varepsilon} \sqrt{X}.
$$
The implied constant may also depend at most on $\lambda, \delta, \rho$ and $\omega$.
\end{lem}
\subsection{Proof of Lemma \ref{Primes in AP 2}} We now combine Lemmas \ref{Combinatorial Decomposition 1}, \ref{Dispersion I}, \ref{Dispersion II} and \ref{Dispersion III}. First, we choose
\[
\sigma = \frac{\omega+4\eta_1+2\eta_2}{2+\omega}+\varepsilon.
\]
Putting together all the conditions, we see that
\[
\sum_{w \equiv w_0 \bmod D}  \widetilde{\Theta}_k\left(\frac{bvw}{Q}\right) \Delta(b w) \ll_{\varepsilon}  Q^{1-2\eta_1-\eta_2-\varepsilon} \sqrt{X},
\]
if we allow
\begin{align}
\omega+75\eta_1+6 \eta_2 &< \frac{1}{8}, \label{final I} \\
\frac{\omega+8\eta_1+6\eta_2}{3(2+\omega)}&< \lambda < \frac{1}{6}-\frac{\omega+4\eta_1+4\eta_2}{2(2+\omega)}, \label{final II} \\
\frac{242 \omega+96(166\eta_1+12\eta_2)}{75(2+\omega)} &< \delta < \frac{1}{12}-\frac{\omega+197\eta_1+9\eta_2}{2(2+\omega)} , \label{final III} \\
\textrm{and} \quad \frac{\omega+4\eta_1+2\eta_2}{2(2+\omega)}&<\rho<\frac{1}{9}-\frac{\omega+(10\eta_1+4\eta_2)}{3(2+\omega)} \label{final IV}.
\end{align}
The conditions on $R$ in the statement of Lemma \ref{Primes in AP 2} ensure that \eqref{final IV} holds. The remaining conditions hold under the constraint 
\[
-o(1) \leq \omega < \widetilde{c}(\eta_1,\eta_2) - o(1).
\]
This completes the proof. \qed
\section{Proofs of Theorem \ref{Main Theorem A} and \ref{Main Theorem C}}\label{sec: Proof of Theorem A: Final Details}
We now present the proof of Theorem \ref{Main Theorem A}.
\begin{proof}[Proof of Theorem \ref{Main Theorem A}]
 It suffices to consider $D= \lfloor Q^{\eta_2} \rfloor$. Similar to the proof of Theorem \ref{Main Theorem}, we write the left hand side of \eqref{Main Theorem Inequality A} as $\mathcal{S}_1+\mathcal{S}_2$, where
\begin{align*}
\mathcal{S}_1 &= \sum_{q \in \mathcal{Q}(a,D)} \Phi_{1} \left( \frac{q}{Q}\right)\sum_{\substack{\chi\bmod q\\ \chi\:  \textup{primitive}}}\sum_{\gamma_{\chi}} \phi \left ( \frac{\log Q}{2\pi}\gamma_{\chi}\right), \notag \\
\textrm{and} \quad \mathcal{S}_2 &= \sum_{q \in \mathcal{Q}(a,D)} \Phi_{2} \left( \frac{q}{Q}\right)\sum_{\substack{\chi\bmod q\\ \chi\:  \textup{primitive}}}\sum_{\gamma_{\chi}} \phi \left ( \frac{\log Q}{2\pi}\gamma_{\chi}\right).
\end{align*}
where $\Phi_1(t)$ and $\Phi_2(t)$ are as in \eqref{Breaking H}. The estimate for $\mathcal{S}_1$ follows from Lemma \ref{Main Theorem Inequality Twisted}. For $\mathcal{S}_2$, we have
\begin{align}\label{S2 Case}
\mathcal{S}_2 &\ll \sum_{q \in \mathcal{Q}(a,D)} \bigg \lvert \sum_{\lvert k \rvert >K} b(k) \mathrm{e}\left ( \frac{kq}{Q}\right)\bigg \rvert \Psi \left( \frac{q}{Q}\right)\sum_{\substack{\chi\bmod q\\ \chi\:  \textup{primitive}}}\sum_{\gamma_{\chi}} \phi \left ( \frac{\log Q}{2\pi}\gamma_{\chi}\right) \notag \\
& \ll \frac{Q^{2\eta_1}}{K}  \sum_{q \in \mathcal{Q}(a,D)} \Psi \left( \frac{q}{Q}\right)\sum_{\substack{\chi\bmod q\\ \chi\:  \textup{primitive}}}\sum_{\gamma_{\chi}} \phi \left ( \frac{\log Q}{2\pi}\gamma_{\chi}\right)
\end{align}
The sum over $q$ in \eqref{S2 Case} can be treated exactly how we established Lemma \ref{Main Theorem Inequality Twisted}. Since $K = Q^{2\eta_1+\varepsilon}$, we see that $\mathcal{S}_{2} \ll_{\varepsilon} Q^{2-\eta_2-\varepsilon}.$ Putting together $\mathcal{S}_1$ and $\mathcal{S}_2$, we arrive at
\begin{align}\label{Final Equality 2}
\mathcal{S}_{1}+\mathcal{S}_2 = \widehat{\phi}(0) \sum_{q \in \mathcal{Q}(a,D)}  \Phi_{1} \left( \frac{q}{Q}\right) \sum_{\substack{\chi\bmod q\\ \chi\: \textup{primitive}}}1+o(Q^{2-\eta_2}),
\end{align}
as $Q \to \infty$. Noting that 
\[
\sum_{q \in \mathcal{Q}(a,D)} \Phi_2 \left ( \frac{q}{Q}\right) \sum_{\substack{\chi\bmod q\\ \chi\: \textup{primitive}}}1\ll_{\varepsilon} Q^{2-\eta_2-\varepsilon},
\]
it follows that the left hand side of \eqref{Main Theorem Inequality A} is equal to 
\begin{align*}
\widehat{\phi}(0) \sum_{q \in \mathcal{Q}(a,D)}  \Phi \left( \frac{q}{Q}\right) \sum_{\substack{\chi\bmod q\\ \chi\: \textup{primitive}}}1+o(Q^{2-\eta_2}),
\end{align*}
which completes the proof of Theorem \ref{Main Theorem A}.
\end{proof}
Following the arguments from \cite[p. 808, Proof of Corollary 2]{DPR2023}, we now prove Theorem \ref{Main Theorem C}.
\begin{proof}[Proof of Theorem \ref{Main Theorem C}]
Assume GRH. Let \(\eta_1, \eta_2 \geq 0\) be fixed satisfying \eqref{Eta 1-2 Constraint C}. Let $\widetilde{c}(\eta_1,\eta_2)$ be the constant provided by Theorem \ref{Main Theorem A}. Let $0<\kappa <\widetilde{c}(\eta_1,\eta_2)$ be fixed and $\lambda>1$ be small enough that $\kappa^{\prime}:=2(\lambda-1)+\lambda \kappa \in\left(0, \widetilde{c}(\eta_1,\eta_2)\right)$ as well. We define
$$
\widetilde{\phi}(x)=\lambda\left(\frac{\sin \pi(2+\kappa') x}{\pi(2+\kappa') x}\right)^2 \quad \textrm{ and } \quad \phi=\widetilde{\phi} * u,
$$
where $u$ is a smooth, positive approximation of unity such that $\phi(0) \geq \lambda^{-1} \widetilde{\phi}(0)=1$. Since
$$
1-\sum_{\gamma_\chi} \phi\left(\frac{\log Q}{2 \pi} \gamma_\chi\right) \leq \mathbbm{1}_{L\left(\frac{1}{2}, \chi\right) \neq 0},
$$
using Theorem \ref{Main Theorem A}, we see that for $Q$ sufficiently large in terms of $\eta_1,\eta_2$ and $\varepsilon$,
\begin{align}\label{Main Theorem Inequality C2}
\sum_{q \equiv a \bmod D} \Psi \bigg(\frac{q}{Q}\bigg) \mathcal{H}_{Q^{\eta_1}} &\bigg( \frac{q}{Q}\bigg) \sum_{\substack{\chi\bmod q\\ \chi\:  \textup{primitive}\\ L(\frac{1}{2}, \chi)\neq 0}}1 \notag \\
&\geq (1-\frac{\lambda}{2+\kappa'}-\varepsilon) \sum_{q \equiv a \bmod D} \Psi \bigg( \frac{q}{Q}\bigg) \mathcal{H}_{Q^{\eta_1}} \bigg( \frac{q}{Q}\bigg) \sum_{\substack{\chi\bmod q\\ \chi\: \textup{primitive}}}1.
\end{align}
Choosing $\kappa = \widetilde{c}(\eta_1,\eta_2)-\varepsilon$ and
\begin{align}\label{c_3 definition}
c^{*}(\eta_1,\eta_2) = \frac{1}{2}+\frac{\widetilde{c}(\eta_1,\eta_2)}{2+\widetilde{c}(\eta_1,\eta_2)} = \frac{50 - 1093(86\eta_1 + 6\eta_2)}{2\big(2236 - 1093(86\eta_1 + 6\eta_2)\big)},
\end{align}
and readjusting $\varepsilon$, the desired conclusion follows.
\end{proof}
\section{Acknowledgements}
The author is grateful to Martin \v{C}ech, Kaisa Matomäki, Kunjakanan Nath, Kyle Pratt, Igor Shparlinski, Jesse Thorner and Alexandru Zaharescu for many useful comments and discussions on the subject of the paper.
\printbibliography
\end{document}